\pdfoutput=1
\RequirePackage{ifpdf}
\ifpdf 
\documentclass[pdftex]{sigma}
\else
\documentclass{sigma}
\fi

\usepackage{amscd}
\usepackage{mathtools}
\usepackage[normalem]{ulem}
\usepackage{tikz-cd}
\usepackage[enableskew]{youngtab}
\usepackage{scalerel}
\newcommand\sbullet[1][.75]{\mathbin{\ThisStyle{\vcenter{\hbox{%
					\scalebox{#1}{$\SavedStyle\bullet$}}}}}%
}

\newcommand{\geneg}{\mathfrak g _-}
\newcommand{\gebar}{\bar{\mathfrak g}}

\newcommand{\R}{\mathbb R}
\newcommand{\Z}{\mathbb Z}
\newcommand{\C}{\mathbb C}

\newcommand{\E}{\mathcal E}
\newcommand{\cL}{\mathcal L}
\newcommand{\fL}{\mathfrak L}
\newcommand{\cD}{\mathcal D}
\newcommand{\cP}{\mathcal P}
\newcommand{\cR}{R}

\newcommand{\f}{\mathfrak f}
\newcommand{\uf}{\underline{\f}}
\newcommand{\g}{\mathfrak g}
\newcommand{\gl}{\mathfrak{gl}}
\newcommand{\fsp}{\mathfrak{sp}}
\newcommand{\lf}{\mathfrak l}
\newcommand{\kf}{\mathfrak k}
\newcommand{\pf}{\mathfrak p}
\newcommand{\m}{\mathfrak m}
\newcommand{\n}{\mathfrak n}
\newcommand{\wJ}{\widehat J}
\newcommand{\wj}{\hat j}
\newcommand{\wX}{\hat X}
\newcommand{\wY}{\hat Y}

\newcommand{\om}{\omega}

\newcommand{\V}{V}

\newcommand{\sll}{\mathfrak{sl}}
\newcommand{\p}{\partial}
\newcommand{\Ad}{\operatorname{Ad}}
\newcommand{\ad}{\operatorname{ad}}
\newcommand{\id}{\operatorname{id}}

\newcommand{\Aut}{\operatorname{Aut}}

\newcommand{\Gr}{\operatorname{Gr}}
\newcommand{\Flag}{\operatorname{Flag}}
\newcommand{\Order}{\operatorname{Order}}

\newcommand{\Sym}{\operatorname{Sym}}
\newcommand{\Sol}{\operatorname{Sol}}
\newcommand{\im}{\operatorname{im}}
\newcommand{\gr}{\operatorname{gr}}
\newcommand{\Hom}{\operatorname{Hom}}
\newcommand{\word}{\operatorname{w-ord}}

\newcommand{\Ker}{\operatorname{Ker}}
\newcommand{\Prol}{\operatorname{Prol}}
\newcommand{\Tr}{\operatorname{Tr}}
\newcommand{\tr}{\operatorname{tr}}

\DeclareRobustCommand{\amgiS}{\text{\reflectbox{$\Sigma$}}}

\newtheorem{thm}{Theorem}[section]
\newtheorem*{thm0}{Theorem}
\newtheorem{lem}[thm]{Lemma}
\newtheorem{prop}[thm]{Proposition}
\newtheorem*{cor}{Corollary}

{\theoremstyle{definition}
\newtheorem{df}[thm]{Definition}
\newtheorem{ex}[thm]{Example}
\newtheorem{rem}[thm]{Remark}}

\numberwithin{equation}{section}

\begin{document}

\newcommand{\arXivNumber}{1904.05687}

\renewcommand{\PaperNumber}{061}

\FirstPageHeading

\ShortArticleName{Extrinsic Geometry and Linear Differential Equations}

\ArticleName{Extrinsic Geometry and Linear Differential Equations}

\Author{Boris DOUBROV~$^{\rm a}$, Yoshinori MACHIDA~$^{\rm b}$ and Tohru MORIMOTO~$^{\rm cd}$}

\AuthorNameForHeading{B.~Doubrov, Y.~Machida and T.~Morimoto}

\Address{$^{\rm a)}$~Faculty of Mathematics and Mechanics, Belarusian State University,\\
\hphantom{$^{\rm a)}$}~Nezavisimosti ave.~4, Minsk 220030, Belarus}
\EmailD{\href{mailto:doubrov@bsu.by}{doubrov@bsu.by}}

\Address{$^{\rm b)}$~Shizuoka University, Shizuoka 422-8529, Japan}
\EmailD{\href{mailto:yomachi212@gmail.com}{yomachi212@gmail.com}}

\Address{$^{\rm c)}$~Seki Kowa Institute of Mathematics, Yokkaichi University, Yokkaichi 512-8045, Japan}

\Address{$^{\rm d)}$~Institut Kiyoshi Oka de Math\'ematiques, Nara Women's University, Nara 630-8506, Japan}
\EmailD{\href{mailto:morimoto@cc.nara-wu.ac.jp}{morimoto@cc.nara-wu.ac.jp}}

\ArticleDates{Received June 04, 2020, in final form June 04, 2021; Published online June 17, 2021}

\Abstract{We give a unified method for the general equivalence problem of extrinsic geo\-met\-ry, on the basis of our formulation of a general extrinsic geometry as that of	an osculating map $\varphi\colon (M,\mathfrak f) \to L/L^0 \subset \operatorname{Flag}(V,\phi)$ from a filtered manifold $(M,\mathfrak f)$ to a homogeneous space~$L/L^0$ in a flag variety $\operatorname{Flag}(V,\phi)$, where $L$ is a finite-dimensional Lie group and $L^0$ its closed subgroup. We establish an algorithm to obtain the complete systems of invariants for the osculating maps which satisfy the reasonable regularity condition of constant symbol of~type $(\mathfrak g_-, \operatorname{gr} V, L)$. We show the categorical isomorphism between the extrinsic geometries in flag varieties and the (weighted) involutive systems of linear differential equations of~finite type. Therefore we also obtain a complete system of invariants for a general involutive systems of linear differential equations of~finite type and of constant symbol. The invariants of an osculating map (or an involutive system of linear differential equations) are proved to be controlled by the cohomology group $H^1_+(\mathfrak g_-, \mathfrak l / \bar{\mathfrak g})$ which is defined algebraically from the symbol of the osculating map (resp.~involutive system), and which, in many cases (in~particular, if the symbol is associated with a simple Lie algebra and its irreducible representation), can be computed by the algebraic harmonic theory, and the vanishing of which gives rigidity theorems in various concrete geometries. We also extend the theory to the case when $L$ is infinite dimensional.}

\Keywords{extrinsic geometry; filtered manifold; flag variety; osculating map; involutive systems of linear differential equations; extrinsic Cartan connection; rigidity of rational homogeneous varieties}

\Classification{53A55; 53C24; 53C30; 53D10}

\begin{flushright}
\it Dedicated to \'Elie Cartan on the 150th anniversary of his birth
\end{flushright}

\section{Introduction}\label{sec1}

In geometry the distinction between intrinsic and extrinsic geometry is fundamental. The former treats spaces, while the latter figures.

A space may be defined to be a set $S$ equipped with a geometric structure $\gamma$ on $S$, which can be defined by assigning a subset of certain associated set to~$S$. A figure may be understood to be a subset of a space, or rather a map $\varphi \colon X \to A $ from a space $X$ to a space $A$.

Two spaces $X =(S_X, \gamma_X) $ and $Y =(S_Y, \gamma_Y) $ are said to be (intrinsically) equivalent or~iso\-mor\-phic if there exists a bijection $f\colon S_X \to S_Y $ satisfying $f_* \gamma_X= \gamma_Y$, where $f_* $ denotes the associated map to $f$.

Two figures $\varphi\colon X \to A $ and $ \psi\colon Y \to B $ are said to be (extrinsically) equivalent or isomorphic if there exist isomorphisms of spaces $f\colon X \to Y $ and $F\colon A \to B$ such that $ F\circ \varphi = \psi \circ f $. Without loss of generality, we may assume that the ambient spaces
$A$ and $B$ coincide and equal to a homogeneous space $L/L^0$ with a group $L$ and its subgroup $L^0$. Then the isomorphisms $F\colon L/L^0 \to L/L^0$
should be understood to be the left translations $\Lambda _a $ by $a \in L$.

We may then say that intrinsic geometry studies those properties of spaces $X$ that are invariant under intrinsic equivalences, and extrinsic geometry those of figures $\varphi\colon X \to L/L^0$ invariant under extrinsic equivalences.

In both geometries one of the fundamental problems is the so-called equivalence problem, that is to find the criteria to determine whether two spaces or two figures are equivalent or not, which, in smooth or analytic category, leads to the problem of finding (the complete) differential invariants of a space or a figure arbitrary given.

For intrinsic geometry we know now well established general theories which have been deve\-lo\-ped by Lie, Klein, Cartan and then
by many authors, in particular, in~\cite{car1908,car1936,Mor1983,Mor1993,sin-st65,st64,tan70,tan79}.

For extrinsic geometry which has much longer history we know a great amount of works done in various concrete problems. However, the general theory does not seem to have been fully developed.

The main purpose of the present paper is to develop a systematic study on extrinsic geometry in general setting of nilpotent geometry so as to be comparable to the theories in intrinsic geometry.

The objects of extrinsic geometry with which we are mainly concerned and which are the most general by many reasons are the morphisms, or in geometrical terms, osculating maps,
\[
\varphi\colon\ (M, \mathfrak f) \to L/L^0 \subset \Flag(V, \phi)
\]
from a filtered manifold $(M, \mathfrak f)$ to a homogeneous space $L/L^0$ in a flag variety $\Flag(V, \phi)$.

We shall then give an algorithm to find the complete invariants of the osculating maps $\varphi$ under certain reasonable conditions.

The framework developed here is a nilpotent generalization in extrinsic geometry of the Cartan method (of bundles) of moving frames, and is in good harmony with the one in intrinsic geometry developed in~\cite{Mor1993,tan70,tan79}. The two methods in intrinsic and extrinsic geometry thus integrated will give a unified view for clearer understanding of geometry as well as for applications to differential equations.

Now let us explain the content of paper more precisely by following the
sections.

In Section~\ref{sec2} first we recall the notion of filtered manifold.
It is a generalization and a~refi\-nement of that of usual manifold and is defined to be a smooth manifold $M$ endowed with a~tangential filtration $\f = \{\f^p\}_{p \in \Z}$. At each point $x \in M $ there is associated a nilpotent graded Lie algebra $\gr\f_x$
which, taking the place of usual tangent space, plays a fundamental role in the method of nilpotent approximation initiated by Tanaka~\cite{tan70,tan79}.

We then fix the notation for the flag variety $\Flag(V, \phi)$ by defining it as the manifold consisting of all filtration of $V$ isomorphic to a given filtration $\phi$, where $V$ is a vector space over~$\mathbb R$ or~$\mathbb C$ assumed to be finite dimensional unless otherwise stated,
therefore it is represented as a~homogeneous space $\Flag(V, \phi )= {\rm GL}(V)/{\rm GL}(V)^0$, where ${\rm GL}(V)^0$ is the isotropy subgroup at~$\phi$.
Moreover it is a filtered manifold with natural left invariant tangential filtration.

Then we call a map $\varphi \colon (M, \mathfrak f ) \to L/L^0 \subset \Flag(V, \phi)$ \emph{osculating} if
\[
\underline{\f^p }\,\underline{\varphi^q }
\subset \underline{\varphi^{p+q} }\qquad \text {for all}\quad p, q \in \mathbb Z.
\]
This is equivalent to saying that the map $\varphi\colon(M, \mathfrak f ) \to \Flag(V, \phi)$ is a morphism of filtered manifolds.

In Section~\ref{sec:cat} we introduce, for a given homogeneous space $L/L^0 \subset \Flag(V, \phi) $, three categories, that is, those of $L/L^0$ extrinsic geometries, $L/L^0$ differential equations and $L/L^0$ extrinsic bundles.

The first one is just the category whose objects are the osculating maps $\varphi\colon (M, \mathfrak f) \to L/L^0 \subset \Flag(V, \phi )$.

The second one represents the category of (weighted) involutive systems of linear differential equations defined in weighted jet bundles (of finite type if $V$ is finite dimensional). See~\cite{Mor2002} for the definition of (weighted) involutive systems.

The third one is the category of the principal fibre bundles $P$ over filtered manifolds $(M, \f )$ with structure group $K^0 \subset L^0$ equipped with an $\mathfrak l$ valued 1-form $\omega$ satisfying above all ${\rm d}\omega + \frac 12 [\omega, \omega] = 0$ as well as some natural conditions with respect to the bundle structure.

In each category we define the notion of a congruence class to clarify our implicit identification of those objects that are congruent.
For instance, two extrinsic bundles $\big(P, K^0, (M, \f), \omega\big)$ and~$(P', K'^0, (M', \f), \omega ')$ are called congruent if the extensions of structure groups to $L^0$ of $P$ and~$P'$ are isomorphic with $(M, \f)=(M', \f)$, the induced map on the base spaces being the identity map.

In Section~\ref{sec4} we show that these three categories are categorically isomorphic up to congruences and coverings.
Thus, roughly speaking, the extrinsic geometry in flag manifold and the geometry of involutive systems of linear differential equations are equivalent and their equivalence problem reduces to that of extrinsic bundles.

In Section~\ref{sec5} we study the equivalence problem for a class of submanifolds in $L/L^0$. The first basic invariant of an osculating map $\varphi$ is $\gr\varphi$ called the symbol of~$\varphi$. It assigns to $x \in M$ the space $\gr\varphi_x $ that has a structure of graded $\gr\f_x$-module. We say $\varphi$ is of constant symbol of type $(\g_-, \gr V, L)$ if $(\gr\f_x, \gr\varphi _x)$ is $L$-isomorphic to $ (\g_-, \gr V)$ for all $x \in M$, where $\g_-$ is a graded nilpotent Lie subalgebra of $\gr_- \lf = \bigoplus_{p <0} \gr_p \lf$.

We make the following assumption throughout:
\begin{enumerate}\itemsep=0pt
	\item[(C0)]
	There exists a filtration preserving identification of $V$ and $\gr V$ that identifies $\lf\subset \gl(V)$ and $\gr(\lf, \phi)\subset \gl(\gr V)$.
\end{enumerate}
We then work in the subcategory of constant symbol of type $(\g_-, \gr V, L)$. An algebraic aspect of this subcategory is represented by
the graded Lie algebra $\bar\g = \bigoplus \bar\g_p \subset \bigoplus \lf_p$ which is called the relative (or extrinsic) prolongation of $ \g_-$ in $\lf$ and is defined to be the maximal graded subalgebra of $\lf$ whose negative part coincides with $\g_-$.
The standard model in this category is given by
$\varphi_\mathrm{model}\colon \overline G/ \overline G^0 \to L/L^0$, where $\overline G$ and $\overline G^0$ are Lie subgroups corresponding to~$\bar \g$ and~$\bar \g^0$. The deviation of $\varphi$ from $ \varphi_\mathrm{model}$ is then measured by passing to
the corresponding extrinsic bundle $Q$ and by its structure function $\chi = \sum _{i \ge 1} \chi_i $ defined on~$Q$.

\smallskip
To go further we pose the following assumptions:
\begin{enumerate}\itemsep=0pt
\item[(C1)] There exists a $\bar G^0$-invariant graded subspace $\bar\g' \subset \lf$ such that $\lf = \bar\g \oplus \bar\g'$.
\item[(C2)] There exists a $\overline{G}^0$-invariant graded subspace $W = \bigoplus W_p \subset \phi^1\Hom(\g_{-}, \lf/\bar\g)$
such that $\phi^1\Hom(\g_{-}, \lf/\bar\g) = \partial \phi^1(\lf/\bar\g)\oplus W$.
\end{enumerate}

Here $\partial$ is the coboundary operator of the following differential complex equipped with the $\partial$-invariant induced filtration $\phi$:
\[
0 \to \lf/\bar\g \overset{\partial}\to \Hom ( \g_-, \lf/ \bar \g )\overset{\partial}\to \Hom \big( {\wedge} ^2\g_-, \lf/ \bar \g \big) \overset{\partial}\to \cdots,
\]
whose cohomology group $H(\g_-, \lf/\bar\g)$ then plays an important role.

Now our main theorem may be stated as follows:

\begin{thm0}\sloppy
 Under the assumptions $(\mathrm{C0})$, $(\mathrm{C1})$ and $(\mathrm{C2})$,
for every $L/L^0$ extrinsic bundle $(Q, \omega_Q)$
$(L/L^0$ extrinsic geometry or $L/L^0$ differential equation$)$
of constant symbol of type $(\g_-, \gr V, L)$, we can construct canonically an extrinsic bundle $(P, \omega_P)$ such that $(P, \omega_P)$ is congruent to~$(Q, \omega_Q)$ and is a reduction to the subgroup $\bar G^0$ and the structure function $\chi_P$ of $P$ takes its value in $W$.
\end{thm0}

The extrinsic bundle $(P, \omega)$ thus constructed is called $W$-normal extrinsic Cartan connection and is governed by the structure equation
\[
{\rm d}\omega +\frac12 [\omega, \omega] = 0,\qquad
\omega= \omega_I + \omega_{II},\qquad
\omega_{II} = \chi \omega_I ,
\]
where $\omega_I$, $\omega_{II}$ denote the $\bar \g$, $\bar \g '$ components of $\omega$ respectively.
It follows from the structure equation, in particular, that
\[
 \partial \chi _k = \Psi_k,
\]
where $\Psi_k$ is a differential polynomial of $\{\chi_i,\, i<k\}$ with $\Psi_1 = 0$. Thus the structure function $\chi= \sum _{i \ge 1} \chi_i$, and hence the structure of $(P, \omega)$ is uniquely determined up to the cohomology group $H^1_+ (\g_-, \lf/\bar \g) $.

Therefore as a corollary we have:

\begin{cor}[rigidity] Any $\varphi\colon (M, \f)\to L/L^0$ of constant symbol of type
$(\g_-, \gr V, L)$ is locally equivalent to
	$\varphi_\mathrm{model}$ if $(\mathrm{C0})$, $(\mathrm{C1})$ and $(\mathrm{C2})$ hold and if $H^1_+(\g_-, \lf/\bar\g ) = 0$.
\end{cor}

In Section~\ref{sec:var} we apply the above theorem to what may be called ``extrinsic parabolic geo\-metry''.
Take a simple graded Lie algebra $\g$ and an irreducible representation
$ \g \to \mathfrak{gl}(V)$, and further take a graded subalgebra $\lf = \bigoplus_{p\in \Z}\lf_p$ of $\mathfrak{gl}(V)$
containing $\g$ and let $L$ and $L^0$ be the
Lie subgroup of ${\rm GL}(V)$ corresponding to $\lf$ and $\lf^0=\bigoplus_{p\geq 0}\lf_p$.
The category of $L/L^0$ extrinsic geometries of constant symbol of type $(\g_-, \gr V, L)$ determined by the above data satisfies
the assumptions (C0), (C1) and (C2): The relative prolongation $\bar \g$ of $\g_-$ in $\lf$ proves to be $\g$ itself or~its central extension. Moreover there exist distinguished inner products on $\g$ and on $V$ such that if we define the adjoint operator $\partial^* $ on the differential complex
concerned by using the induced inner product then $\Ker \partial ^* $ will serve as $W$ required in (C2). The standard model $\varphi_\mathrm{model}$ is an~embedding of
a rational homogeneous variety into a flag variety.

Thus we have a large and rich class of $L/L^0$ categories associated with each such $(\g, \lf, \mathfrak{gl}(V))$ to which the theorem in Section~\ref{sec5} applies.

In this case the cohomology group $H^1_+(\g_-, \lf / \bar \g )$ can be computed by using root systems after Kostant's method. We have carried computation for every simple graded Lie algebra $\g$ over $\mathbb C$ and its irreducible representation $U$ to see when $ H^1_+(\g_-, U)$ is trivial, which gives a lot of rigidity theorems and a rough picture about the extrinsic parabolic geometries.

In Section~\ref{sec:general} we further study the equivalence problem of extrinsic geometry in more general setting
than in Section~\ref{sec5}.

Firstly, we treat the case where (C1) and (C2) are not necessarily satisfied. Then we can no more expect to have Cartan connections but we can still construct a series of bundles which give the complete invariants of a given osculating map $\varphi$ and the invariants are again controlled by the cohomology group $H^1_+(\g_-, \lf/ \bar \g)$. Hence we see, in particular, that
the corollary stated in~Section~\ref{sec5} holds without the assumptions (C1) and (C2).

Secondly, we generalize our theory to the case when $L$ is infinite dimensional. The main motivation is, on one hand, to study extrinsic geometries with respect to infinite-dimensional transformation group such as the general diffeomorphism group or the contact transformation group,
and, on the other hand, to study geometry of linear differential equations of infinite type.

Here we notice the remarkable similarity between intrinsic and extrinsic geometry which becomes visible and clear in this work:
Recall that the equivalence problem of intrinsic geometry reduces to that of towers.
A tower is a principal fibre bundle $P$ with structure group $G^0$ possibly infinite-dimensional
equipped with a vector valued 1-form which defines an absolute parallelism of $P$ satisfying certain natural conditions. While the equivalence problem of extrinsic geometry reduces to that of extrinsic bundles
defined in this paper. In~both intrinsic and extrinsic cases we are led to the equivalence problem of principal fibre bundles equipped with Pfaff systems.

The constructing Cartan connections make nice harmony in intrinsic and extrinsic geometry,
the condition (C) in intrinsic geometry~\cite{Mor1993} corresponding to (C2) in extrinsic geometry.

Moreover, the prolongation and reduction method of Singer--Sternberg and Tanaka for int\-ri\-nsic geometry corresponds to the reduction method in Section~\ref{sec:general} for extrinsic geometry.

It is the first cohomology group $H^1$ for extrinsic geometry
and the second $H^2$ for intrinsic geometry
that pays the key role.

Since extrinsic geometry has long history, there may be a lot of works which may relate to our work. We cite here those works that have influenced us or that we find related to the present paper: Lie~\cite{lie}, Wilczynski~\cite{wilch, wilch2}, Cartan~\cite{car1936, cartan1937}, Sulanke~\cite{sulanke}, Jensen~\cite{jensen}, Kol\'a\v{r}~\cite{kolar}, Se-Ashi~\cite{se1, se2}, Sasaki~\cite{ sasaki2, S,sasaki1}, Sasaki--Yamaguchi--Yoshida~\cite{syy},
Fels--Olver~\cite{fels-olver1999}, Hwang--Yamaguchi~\cite{hwang-yam}, Robles~\cite{robles}, Landsberg--Robles~\cite{landrob}, Doubrov--Komrakov~\cite{dk}, Doubrov--Zelenko~\cite{douzel}.

Among all we owe, in particular, to Wilczynski and Se-Ashi:
At the beginning of twentieth century after the influence of Lie, Wilczynski developed a general theory of projective curves and ruled surfaces, recognizing that these geometrics are equivalent to those of ordinary linear differential equations and certain systems of linear partial differential equations of 2nd order.

Se-Ashi reconstructed the differential invariants of ordinary linear differential equations by the method of bundles of moving frames, and found a prototype of extrinsic Cartan connections.
Then he extended this construction to systems of linear differential equations associated with a~simple graded Lie algebra of depth one and its irreducible representation.

Our preset work may be viewed as a full generalization of the works of Wilczynski and Se-Ashi to the framework of nilpotent geometry and analysis in which we find all well unified.

Even in extrinsic parabolic geometry, there are many interesting concrete examples. The~sim\-plest is the one associated with the simple Lie algebra $\mathfrak{sl}(2)$ and its irreducible representation on~$\Sym^n\big(\mathbb C^2\big)$, and this is exactly the geometry of curves in the projective space $P^n$ or that one of linear ordinary differential equations of order $n+1$ treated by Wilczynski~\cite{wilch} and Se-Ashi~\cite{se1,se2} (see also Ovsienko--Tabachnikov~\cite{O-T} for a modern exposition of Wilczynski results).

The next interesting example which is non-trivial as nilpotent geometry is the one associated with ($\mathfrak{sl}(3)$, Borel grading, adjoint representation) as shown by the works of Robles~\cite{robles} and Landsberg--Robles~\cite{landrob}. We have carried a detailed study in this case after our general method,
which will be an object of the second part of this paper.

\section{Filtered manifolds and flag varieties}\label{sec2}

\subsection{Filtered manifolds}

A \emph{filtered manifold} is a smooth manifold $M$ equipped with a tangential filtration $\f=\{\f^p\}_{p\in\Z}$ satisfying the following conditions:
\begin{enumerate}\itemsep=0pt
	\item[$(i)$] Each $\f^p$ is a subbundle of the tangent bundle $TM$ and $\f^p\supset \f^{p+1}$.
	\item[$(ii)$] $\f^0=0$ and $\f^{-\mu}=TM$ for a non-negative integer $\mu$.
	\item[$(iii)$] $[\uf^p,\uf^q]\subset\uf^{p+q}$ for $p,q\in\Z$, where $\uf^{\sbullet}$ denotes the sheaf of germs of sections of $\f^{\sbullet}$.
\end{enumerate}
The filtration $\{\f^p\}$ is occasionally written as $\{\f^pTM\}$ or as $\{T^pM\}$. The fibre of
$\f^p$ at $x\in M$ will be written as $\f^p_x$, $\f^pT_xM$, or $T^p_xM$. The filtered manifold will be denoted $(M,\f)$ or simply~$M$.

A (local) isomorphism $\varphi\colon (M,\f)\to (M',\f')$ of two filtered manifolds is a (local) diffeomorphism of the underlying smooth manifolds such that
$\varphi_*\colon (T_xM, \f_x) \to \big(T_{\varphi(x)}M', \f'_{\varphi(x)}\big)$ is an~iso\-mor\-phism of filtered vector spaces for all $x\in M$, where $\varphi$ is defined.

More generally, we can define a \emph{morphism of two filtered manifolds} $\varphi\colon (M,\f)\to (M',\f')$ as an arbitrary smooth map $\varphi\colon M\to M'$ such that $\varphi_*\colon T_x M\to T_{\varphi(x)} M'$ preserves the filtrations, that is $\varphi_*(\f_x^p)\subset \f^{\prime\,p}_{\varphi(x)}$.

Let $(M,\f)$ be a filtered manifold. For $x\in M$ we set
\[
\gr_p\f_x = \f^p_x/\f^{p+1}_x
\]
and
\[
\gr \f_x =\bigoplus_{p\in \Z} \gr_p\f_x.
\]
We see easily that $\gr\f_x$ inherits a Lie bracket induced from the bracket of vector fields, which satisfies
\[
\big[\gr_p\f_x,\gr_q\f_x\big]\subset \gr_{p+q}\f_x
\]
for all $p,q\in\Z$. Hence, $\gr\f_x$ turns to be a nilpotent graded Lie algebra, which is called \emph{the symbol algebra of $(M,\f)$ at $x$}. It should be noted that $\gr\f_x$ and $\gr\f_{x'}$ are not necessarily isomorphic as graded Lie algebras for different $x,x'\in M$. If, for all $x\in M$, the Lie algebra $\gr\f_x$ is isomorphic to a certain fixed graded Lie algebra $\m=\oplus_{p\in\Z}\m_p$, we say that $(M,\f)$ is \emph{of type $\m$}.

\begin{lem}
	Let $\varphi\colon (M,\f)\to (M',\f')$ be an immersion of two filtered manifolds. Then for each $x\in M$ the map $\varphi_*$ induces a homomorphism $\gr\f_x \to \gr\f'_{\varphi(x)}$ of graded Lie algebras. In~par\-ti\-cu\-lar, if $\varphi$ is a local diffeomorphism, then the Lie algebras $\gr\f_x$ and $\gr\f'_{\varphi(x)}$ are isomorphic.
\end{lem}
\begin{proof}
	Indeed, consider arbitrary $u\in \gr_p\f_x$, $v\in \gr_q\f_x$. Let $U\in \underline{\f^p}$, $V\in\underline {\f^q}$ such that $u=U_x\mod\f^{p+1}$ and $v=V_x\mod\f^{q+1}$.
Then $\varphi_*(U)$, $\varphi_*(V)$ are well-defined vector fields on $\varphi(M)$ in~some neighborhood of $\varphi(x)$, and, moreover, $\varphi_*([U,V])=[\varphi_*(U),\varphi_*(V)]$. Locally we can extend $\varphi_*(U)$, $\varphi_*(V)$ to vector fields $U'\in\underline{(\f')^p}$, $V'\in\underline{(\f')^q}$. Then we have $[U',V']|_{\varphi(M)} = \varphi_*([U,V])$. As the definition of the bracket operations on $\gr\f$ and $\gr\f'$ does not depend on~the choice of representatives,
we see that $\varphi_*$ indeed induces the Lie algebra homomorphisms
$\gr\f_x\to \gr\f'_{\varphi(x)}$ for all $x\in M$.
\end{proof}
\begin{rem}
	Note that the induced homomorphism of graded Lie algebras is not injective in~general, even if $\varphi_*$ is injective at each point.
\end{rem}

\begin{ex} [trivial filtered manifold] Any smooth manifold $M$ can be considered as a (trivial) filtered manifold $(M,\f)$ with $\f^0=0$ and $\f^{-1} = TM$. It is clear that
	$\gr\f_x$ is a commutative Lie algebra of dimension $\dim M$ for all $x\in M$. Hence $\gr\f_x$ is identified with $T_x M$ regarded as a~commutative Lie algebra.
\end{ex}

\begin{ex} [filtered manifold generated by a differential system $D$]
Let $M$ be an arbitrary smooth manifold, and let $D\subset TM$ be a vector distribution.
If
\[
\f^{-1}=D,\qquad
\underline{\f^{-k-1}}=\underline{\f^{-k}}+\big[\underline{\f^{-1}},\underline{\f^{-k}}\big]\qquad
\forall k>0,
\]
we say that $\f=\{ \f^p\}$ is generated by $D=\f^{-1}$,
or alternatively the vector distribution $D$ is (regularly) bracket generating.

The symbol algebra $\gr \f_x$ is also called the symbol of the distribution $D$ at a point $x\in M$. Note that in this case $\gr \f_x$ is necessarily generated by $\gr_{-1}\f_x$.
\end{ex}

\begin{ex} [homogeneous filtered manifold]
Let $G/G^0$ be a homogeneous space with a Lie group $G$ and its closed subgroup $G^0$.
Assume that the Lie algebra $\g$ of $G$ is endowed with a~filtration $\{ \g^p\}$
such that $(i)$ $[\g^p,\g^q]\subset \g^{p+q}$, $(ii)$ ${\rm Ad}(a)\g^p\subset \g^p$ for
$a\in G^0\ (p<0)$, $(iii)$ $\g^0$ is the Lie algebra of $G^0$.
Then the filtration $\{ \g^p\}$ induces a left $G$-invariant tangential filtration
$\f_{G/G^0}$ on $G/G^0$.
We see immediately that $\gr(\f_{G/G^0})_x\cong \gr_-\g \big({}=\bigoplus_{p<0}\gr_p\g\big)$
for any $x\in G/G^0$.

In particular, if $\m=\bigoplus_{p<0} \m_p$ is a nilpotent graded Lie algebra,
the simply connected Lie group $M$ corresponding to $\m$ endowed with the
left invariant tangential filtration is called the \emph{standard filtered manifold of type $\m$}.
\end{ex}

\subsection{Flag varieties}

Let $V$ be a vector space over $\R$ or $\C$, and be finite dimensional
unless otherwise stated.
By~a~fil\-t\-ration $\phi$ of $V$ we mean a series $\{ \phi ^p\} _{p \in \Z} $ of
subspaces $ \phi^p$ of $V$ satisfying $\phi^p \supset \phi^{p+1} $, that is, our filtrations are descending. A filtration $\phi$ is called \emph{saturated} if $\cup \phi^p = V$,
and \emph{fine} if $\cap \phi^p = 0$.
Unless otherwise stated, we shall always assume that all filtrations are saturated and fine.

Two filtrations $\phi_1, \phi_2$ are said to be isomorphic if
there is a linear isomorphism $f$ of $V$ such that
$f\phi_1^p = \phi_2^p $ for all $p$.

Let $\Flag(V) $ denote the flag space of all filtrations of $V$, and $\Flag(V, \phi)$ the flag variety of~type~$\phi$, that is, the set of all filtrations of V isomorphic to $\phi$. Note that the filtration $\phi$ of $V$ induces that of $\gl(V)$, denoted by the same letter $\phi$:
\[
\phi^p\gl(V) = \big\{ A \in \gl(V) \colon A\,\phi^iV \subset \phi^{i+p} V \big\},
\]
which satisfies
\[
\big[ \phi^p\gl(V), \phi^q\gl(V) \big] \subset \phi^{p+q}\gl(V).
\]
Denote also by ${\rm GL}(V)^0$ the subgroup of ${\rm GL}(V)$ stabilizing the flag $\phi$. Its Lie algebra is exactly $\phi^0\gl(V)$.

Note that the flag variety $\Flag(V,\phi)$ can be naturally identified with the homogeneous space ${\rm GL}(V)/{\rm GL}(V)^0$. Since the filtration $\{\phi^p\gl(V)\}$ is invariant with the adjoint action of ${\rm GL}(V)^0$, from Example~3 we see that $\Flag(V,\phi)$ is a homogeneous filtered manifold.

\subsection{Osculating maps}

Let $(M, \f)$ be a filtered manifold.
A smooth map $\varphi\colon (M, \f) \to \text{ Flag}(V, \phi) $ is called \emph{osculating} if
\[
\underline{\f^p} \, \underline{\varphi^q} \subset \underline{\varphi^{p+q} }\qquad
\text{for all}\quad p, q \in \Z,
\]
where $\varphi^q $ denotes the vector bundle $\cup_{x \in M} \varphi^q(x)$ over $M$ induced from $\varphi$ and $\underline{\varphi^q}$ the sheaf of sections of $\varphi^q$, and the multiplication on the left hand side of the above formula signifies the operation by differentiation along vector fields by viewing $\varphi^q\subset V\times M$. It is easy to see that this condition is equivalent to the fact that $\varphi$ is a morphism of filtered manifolds.

Let $\varphi\colon (M, \f) \to \Flag(V, \phi)$ be an osculating map. For $x \in M$, denote by $\gr\varphi(x)$ the graded vector space $\bigoplus\varphi^p(x)/\varphi^{p+1}(x)$.
Then we have the associated map
\[
\gr_p\f_x \times \gr_q\varphi(x) \to \gr_{p+q}\varphi(x).
\]
By this operation $\gr\varphi(x)$ becomes a $\gr\f_x $-module.

We say that an osculating map $\varphi\colon (M, \f) \to \Flag(V, \phi) $
is generated by $\{ \varphi ^p \}_{ p \ge p_0} $ if
$\gr\varphi(x) $ is generated by $ \bigoplus_{ p \ge p_0} \gr_p \varphi(x) $
for all $ x \in M$. In~other words, this means that the map $\varphi$ is completely determined by $\varphi^{p_0}$ as follows
\[
\underline{\varphi^p} = \sum_{i>0} \underline{\f^{-i}}\,\underline{\varphi^{p+i}}\qquad
\text{for all}\quad p<p_0.
\]
Actually, any map from a filtered manifold to a projective space or
to a Grassmann variety generically generates an osculating map to a flag variety.

We say that an osculating map $\varphi\colon (M, \f) \to \Flag(V, \phi) $ has constant symbol of type $(\g_-, \gr V) $ if there exist a nilpotent graded Lie algebra $\g_- = \bigoplus_{p < 0} \g_p$ and a graded $\g_-$-module structure on~$\gr(V, \phi)$ such that $(\gr\f_x, \gr \varphi(x))$ is isomorphic to $(\g_-, \gr V)$ as graded modules for all $x \in M$, that is, there exist an graded linear isomorphism $\alpha\colon \gr V\to \gr\varphi(x)$ and a graded Lie algebra isomorphism
$\beta\colon\geneg \to \gr\f_x$ which make the following diagram commutative:
\[
\begin{CD}
\g_- \times \gr V @>>> \gr V \\
@V{\beta \times \alpha}VV
@VV{\alpha}V \\
\gr\f_x \times \gr\varphi(x) @>>> \gr\varphi(x).
\end{CD}
\]

\section{Three categories}
\label{sec:cat}

In this section we define three categories associated to
$L/L^0 \subset \Flag(V, \phi)$, where $(V, \phi)$ is a~fil\-tered vector space, $L$ is a Lie subgroup of ${\rm GL}(V)$ and $L^0$ the subgroup of $L$ which fixes $\phi$. We denote the induced filtration of $\gl(V)$ by $\{\phi^p\gl(V)\}$ or simply by $\{\gl(V)^p\}$ and define the filtration $\{\lf^p\}$ of the Lie algebra $\lf$ of $L$ by
$\lf^p = \lf \cap \gl(V)^p$. Similarly we set $L^p = L \cap {\rm GL}(V)^p$ for~$p \ge 0$. We denote by $\gr\lf$ the associated graded Lie algebra and $\gr_-\lf=\bigoplus_{p<0}\gr_{p}\lf$.

\subsection[Category of L/L0 extrinsic geometries]
{Category of $\boldsymbol{L/L^0}$ extrinsic geometries}
\begin{df} Every object of the category of \emph{$L/L^0$ extrinsic geometries} is the osculating map
	\begin{gather*}
	\varphi \colon\ (M, \f)\to L/L^0 \subset \Flag(V,\phi).
	\end{gather*}
\end{df}

Every morphism is a pair $(f, \Lambda_a)$ of maps defined by the following commutative diagram:
\[
\begin{CD}
(M,\f) @>\varphi>> L/L^0\subset\Flag(V,\phi) \\
@VVf V @VV\Lambda_a V \\
(M',\f') @>\varphi'>> L/ L^0\subset\Flag(V,\phi),
\end{CD}
\]
where $f$ is a diffeomorphism of filtered manifolds and $\Lambda_a$ is a left translation by $a\in L$.

\begin{df}
	An $L/L^0$ extrinsic geometry is called of type $(\g_{-}, \gr V, L)$ at $x\in M$, if there exist
 a graded Lie subalgebra $\g_- \subset \gr_-\lf$,
 a graded Lie algebra isomorphism $\beta\colon \g_{-} \to \gr\f_x$ and~$a\in L$ satisfying $a\phi = \varphi_x$ which make the following diagram commutative:
\[
\begin{CD}
\g_{-}\times \gr V @>>> \gr V \\
@VV\beta\times\alpha V @VV\alpha V \\
\gr\f_x\times\gr \varphi(x) @>>> \gr\varphi(x),
\end{CD}
\]
where $\alpha=\gr a$.
It is called (of constant symbol) of type $(\g_{-}, \gr V, L)$, if it is so for all $x\in M$.
\end{df}

\subsection[Category of L/L0 differential equations]
{Category of $\boldsymbol{L/L^0}$ differential equations}
\begin{df}\label{dfn3}
	Every object of \emph{the category of $L/L^0$ $($weighted$)$ involutive systems of linear differential equations} (or simply \emph{the category of $L/L^0$ differential equations}) is an $L^0$-filtered vector bundle $(R,\{R^q\})$ with a typical fibre $(V,\{V^q\})$ over a filtered manifold $(M,\f)$ equipped with a flat $\lf$-connection $\nabla$ on $R$ satisfying
	\[
		\nabla_{\underline{\f^p}} \underline{R^q}\subset \underline{R^{p+q}}\qquad
\text{for all}\quad p,q.
	\]
	Every morphism of this category is an isomorphism $\big(R, \{R^q\}, (M,\f),\nabla\big)\! \to \! \big(R', \{R^{\prime q}\}, (M',\f),\nabla'\big)$ of $L^0$-filtered vector bundles, filtered manifolds and $\lf$-connections.
\end{df}

We mean by an $L^0$-filtered vector bundle with a typical fibre $\big(V,\{V^q\}\big)$ a vector bundle whose structure group is specified to be $L^0\subset \phi^0 {\rm GL}(V)$, namely a vector bundle endowed with a maximal class of admissible local trivializations $M\times V$ covering the bundle such that the transition function of any two admissible trivializations takes values in $L^0$.

If $\pi\colon R\to M$ is an $L^0$-vector bundle with a typical fibre $V$ and if $U\times V\ni (x,v)\mapsto \zeta(x)v\in\pi^{-1}(U)$ is an admissible local trivialization, we call $\zeta$ or $\zeta(x)$ an admissible frame of $R$. If we fix a~basis $\{e_1,e_2,\dots,e_m\}$ of $V$ and set $\zeta_i(x)= \zeta(x)e_i$, then $\{\zeta_1(x),\zeta_2(x),\dots,\zeta_m(x)\}$ is an~admissible basis of the frame $R_x$.

Recall that a connection on a vector bundle $E\to M$ is a 1-st order linear differential operator
\[
\nabla \colon\ \underline{E} \to \underline{E\otimes T^*M}
\]
satisfying
\[
\nabla(f\sigma) = f\nabla\sigma + \sigma\otimes {\rm d}f\qquad \text{for}\quad
\sigma\in\underline{E},\quad f\in \underline{M\times \R}.
\]

Now as in Definition~\ref{dfn3}, let $R$ be an $L^0$-vector bundle and $\nabla$ a connection on $R$ in the usual sense defined above, and let $L^0\subset L\subset {\rm GL}(V)$. We say that \emph{$\nabla$ is an $\lf$-connection} if
\[
	\zeta^{-1}\nabla_X\zeta \in\lf
\]
for any $L^0$-admissible frame $\zeta$ of $R$ and $X\in\underline{TM}$, where $\lf$ is a Lie algebra of $L$.

We say that the connection $\nabla $ is flat if its curvature $K$ vanishes identically, where
\[
	K(X,Y) = \nabla_X\nabla_Y - \nabla_Y\nabla_X-\nabla_{[X,Y]}.
\]

Let $\big((R, \{R^p\}), (M,\f), \nabla\big)$ be an $L/L^0$ differential equation. For each $x\in M$, we have
\[
\gr \f_x \times \gr R_x \to \gr R_x,
\]
which makes $\gr R_x$ a graded $\gr\f_x$-module called the symbol of $R$ at $x$.

\begin{df}
	An $L/L^0$ differential equation $\big(R, \{R^q\}, \nabla\big)$ over $(M,\f)$ is of type $(\g_{-}, \gr V, L)$ at $x\in M$, if there exist
a graded Lie subalgebra $\g_- \subset \gr_-\lf$, an isomorphism $\beta\colon \g_{-} \to \gr \f_x$ of~graded Lie algebras and
$z \in \Pi_x$
such that the following diagram is commutative:
	\[
	\begin{CD}
	\g_{-}\times \gr V @>>> \gr V \\
	@VV\beta\times\gr z V @VV\gr z V \\
	\gr\f_x\times\gr R_x @>>> \gr R_x.
	\end{CD}
	\]
	It is called (of constant symbol) of type $(\g_{-}, \gr V, L)$ if it is so for all $x\in M$.
\end{df}

\subsection[Category of L/L0 extrinsic bundles]{Category of $\boldsymbol{L/L^0}$ extrinsic bundles}

Let $K^0$ be an arbitrary subgroup of $L^0$, and let $\kf^0\subset \lf^0$ be the corresponding subalgebra. Note that it naturally inherits the filtration from $\lf$: $\kf^p = \lf^p \cap \kf^0$.

\begin{df}
	Every object in the category of \emph{$L/L^0$ extrinsic bundles} is a principal $K^0$-bundle $\pi\colon P\to M$ over a filtered manifold $(M, \f)$ endowed with an $\lf$-valued 1-form $\omega_P$ (or~simply written~$\omega$) such that
	\begin{enumerate}\itemsep=0pt
		\item[$(i)$] $R_a^*\om = {\rm Ad}(a)^{-1}\om$, where $R_a$, $a\in K^0$, denotes the right shift on the principal bundle,
		\item[$(ii)$] $\langle \tilde A, \omega\rangle = A$ for any fundamental vector field $\tilde A$ generated by $A\in \kf^0$,
		\item[$(iii)$] ${\rm d}\om + \tfrac{1}{2} [\om, \om] = 0$.
	\end{enumerate}
\end{df}

Every morphism of this category is an isomorphism of $L/L^0$ extrinsic bundles, that is, a~bundle isomorphism preserving $\lf$-valued 1-forms.

It should be remarked that if $(P,K^0,\om)$ is an $L/L^0$ extrinsic bundle over $(M,\f)$, then the image of
\[
\gr \bar\om_z \colon\ \gr \f_x \to \gr_{-}\lf
\]
is a graded Lie subalgebra of $\gr_{-} \lf$.

Indeed, let $u\in \gr_p\f_x$ and $v\in \gr_q\f_x$. Take $X\in \f^p_x$, $Y\in \f^q_x$ such that $X_x\equiv u$, $Y_x\equiv v$, and let~$\tilde X$, $\tilde Y$ be lifts of $X$, $Y$ to $P$ respectively. Then
\begin{align*}
\big[\gr \bar\om_z(u), \gr \bar\om_z(v)\big] &\equiv \big[\om_z\big(\tilde X\big), \om_z\big(\tilde Y\big)\big]
\!= -{\rm d}\om\big(\big[\tilde X, \tilde Y\big]\big) \!= -\tilde X \om\big(\tilde Y\big) \!+\tilde Y \om \big(\tilde X\big)
\!+\om\big(\big[\tilde X,\tilde Y\big]\big)
\\
&\equiv \om_z\big(\big[\tilde X,\tilde Y\big]\big) \equiv \om_z\big(\widetilde{[X,Y]}\big) \equiv \gr \bar\om_z\big([u,v]\big).
\end{align*}

We say that $\big(P,K^0,\om\big)$ is (\emph{of constant symbol}) \emph{of type
	$\big(\g_-, \, L/L^0\big)$} if for any $x\in M$ there exists $z\in\pi^{-1}(x)$ such that $\gr\bar\om_z (\gr\f_x)=\g_{-}$ for a graded Lie subalgebra $\g_{-}\subset \gr_{-}\lf$.

\subsection{Congruence classes}

Let us define the congruence classes in each three categories in order to state rigorously the functorial isomorphisms that we shall prove in the next section.

We say two $L/L^0$ extrinsic geometries $\varphi_1\colon (M_1, \f_1)\to L/L^0$ and $\varphi_2 \colon (M_2, \f_2)\to L/L^0 $
are \emph{congruent by a congruence} $\left(\id_M, \Lambda_a\right)$
if $M_1 = M_2 = M$ and if there exists an isomorphism $\left(\id_M, \Lambda_a\right)$
from $\varphi_1$ to $\varphi_2$ of $L/L^0$ extrinsic geometries.

We say also that two isomorphisms of extrinsic geometries
\begin{gather*}
	(f, \Lambda_a)\colon\ \ \,\big[\varphi_1\colon (M_1, \f_1)\to L/L^0 \big]	\to
	\big[\varphi_2 \colon (M_2, \f_2)\to L/L^0 \big],
\\
	(f', \Lambda_{a'})\colon\ \big[\varphi_1\colon (M_1, \f_1)\to L/L^0 \big]	\to
	\big[\varphi_2 \colon (M_2, \f_2)\to L/L^0 \big]
\end{gather*}
are congruent if there exist congruences
$\left(\id_{M_1}, \Lambda_{b_1}\right)$ and $\left(\id_{M_2}, \Lambda_{b_2}\right)$
such that
\[
 \left(\id_{M_2}, \Lambda_{b_2}\right) \circ (f, \Lambda_a) = \left(f', \Lambda_{a'})\circ (\id_{M_1}, \Lambda_{b_1}\right)\!.
\]

Then we have a category whose objects are the congruent classes of extrinsic geometries and the morphisms the congruence classes of isomorphisms.

In accordance with the above definition, we use the terminology: ``cong\-ruence'' and ``cong\-ruent'' also for (vector or fibre) bundles to mean by a congruence an isomorphism of bundles which covers the identity map of the base spaces. Then we have the category of congruence classes of $L/L^0$ differential equations and that of congruence classes of $L/L^0$ extrinsic bundles. But for $L/L^0$ extrinsic bundles we introduce another notion of congruence in a wider sense, called \emph{holo-congruent}.

Before that let us make some observation on a certain generalization of connection.

Let $G$ be a Lie group and $K$ a Lie subgroup with Lie algebra $\g$ and $\mathfrak k$ respectively. Let $Q \to M$ be a principal $K$-bundle over $M$. A $\g$-valued 1-form $\omega_Q$ (or simply $\omega$) on $Q$ will be called a~$\g$-con\-nection (form) if it satisfies:
\begin{enumerate}\itemsep=0pt
	\item[$(i)$] $R_a^* \omega = {\rm Ad}(a)^{-1} \omega$ for $a \in K$,
	\item[$(ii)$] $\langle \omega, \tilde A \rangle = A$ for $A \in \mathfrak k$.
\end{enumerate}

Cartan connections in intrinsic geometry and $L/L^0$ extrinsic bundles are examples of $\g$-connections.

If $\omega$ is a $\g$-connection on $Q$ and if we set
\[
{\rm d}\omega + \frac12 [\omega, \omega] = \gamma,
\]
then
\begin{gather*}
	i_{\tilde A}\gamma = 0, \\
	R_a^* \gamma = {\rm Ad}(a)^{-1} \gamma.
\end{gather*}
Therefore $\gamma$ is be viewed as a map
\[
\wedge ^2 ( TQ/VQ) \to \g ,
\]
where $VQ$ denotes the vertical tangent bundle of $Q$, and is called the curvature of $\omega$. If~$\gamma = 0$, we say that $\omega$ is flat.

\begin{prop}
	Let $ G \supset H \supset K \supset L$ be a descending sequence of Lie subgroups and $Q \to M$ a principal $K$-bundle equipped with a $\g$-connection $\omega_Q$. Then
\begin{enumerate}\itemsep=0pt
\item[$(1)$] If $R$ is principal $L$-bundle over $M$ and $\iota\colon R \to Q$ is an injection of principal bundles cove\-ring $\id_M$, then $ \omega_R( = \iota^*\omega_Q)$ is a $\g$-connection on $R$, and $\omega_Q$ is flat if and only if so is~$\omega_R$.
	
\item[$(2)$] There exist, uniquely up to congruences, a principal $H$-bundle $P \to M$, a $\g$-connection on $\omega_P$ on $P$, and an injection $\iota\colon Q \to P$ of principal bundles covering $\id_M$ such that $\iota^* \omega_P = \omega_Q$.
\end{enumerate}
\end{prop}

\begin{proof} Let us indicate how to construct $(P, \omega_P)$.
First we define $P$ to be the fibre product
$Q\times _K H $. We then define a $\g$-valued 1-form $\omega_{Q\times H}$
on $Q\times H$ by
\[
 ( \omega_{Q \times H} ) _{(q, h)} = {\rm Ad}(h) ^{-1} (\omega_Q)_q + (\Omega _H )_h \qquad \text{for}\quad (q, h) \in Q\times H,
\]
where $\Omega_H$ denotes the Maurer--Cartan form of $H$. It is not difficult to verify that there exists a~unique $\g$-valued 1 form
$\omega_P$ on $P$ such that its pull-back to $Q\times H$ coincides with
$\omega_{Q\times H}$. Then it is easy to prove the other statements of the proposition.

We remark that a simple computation yields
\[
{\rm d}\omega_{Q\times H} +\frac12 [ \omega_{Q\times H}, \omega_{Q\times H} ]
={\rm Ad}(h)^{-1} \bigg({\rm d}\omega_Q +\frac12[\omega_Q, \omega_Q]\bigg),
\]
which tells how to compute the curvature by taking a local trivialization of $Q$ and also gives an~alternative proof of the second half of (1).
\end{proof}

Now let us return to the category of $L/L^0$ extrinsic bundles.
By the above proposition we see that for an $L/L^0$ extrinsic bundle $\left(Q, K^0, M, \omega_Q\right)$ there is, uniquely up to congruences, an~extension of $Q$ to the group
$L^0$ which we denote $\left(\bar Q, L^0, M, \omega_{\bar Q}\right)$.

We say that two $L/ L^0$ extrinsic bundles $Q$ and $Q'$ are
\emph{holo-isomorphic} (resp.\ \emph{holo-congruent}) if their extensions to $L^0$, $\bar Q$ and $\bar Q'$ are isomorphic (resp.\ congruent).

Then the holo-conguruence classes of $L/L^0$ extrinsic bundles
forms a category, of which the morphisms are the congruence classes of
holo-isomorhisms.\vspace{-1.5mm}

\section{Categorical isomorphisms}
\label{sec4}

In this section we show that the categories of congruence classes of $L/L^0$ extrinsic geometries, congruence classes of $L/L^0$ differential equations, and holo-congruence classes of $L/L^0$ extrinsic bundles are categorically equivalent. The passage to the $L/L^0$ extrinsic geometries needs the topological assumption of simply connectedness. Once recognized we will not be strict for the distinction between $X$ and the (holo-) congruence class of $X$, if there is no fear of confusion.
\vspace{-1mm}

\subsection{From an extrinsic geometry to an extrinsic bundle}\vspace{-.5mm}

Let $\varphi \colon (M,\f)\to L/L^0\subset \Flag(V,\phi)$ be an $L/L^0$ extrinsic geometry. Let $\varphi^*L$ be the induced principal fibre bundle over $M$. Then its structure group is $L^0$ and there is a canonical embedding $\iota\colon \varphi^*L\to L$. Let $\om$ be the pull back $\iota^*\om_L$, where $\om_L$ is the Maurer--Cartan form of $L$. Then $(\varphi^*L,L^0,\om)$ is an $L/L^0$ extrinsic bundle over $(M,\f)$.
\vspace{-.5mm}

\subsection{From an extrinsic bundle to an extrinsic geometry}

Let $\left(P,K^0,\om_P\right)$ be an extrinsic bundle. We then consider the Pfaff system on $P\times L$ defined by\vspace{-.5mm}
\[
\Omega = \om_L-\om_P = 0.\tag{$\Sigma$}
\]
Since we have\vspace{-.5mm}
\[
{\rm d}\Omega = -\frac12\left[\Omega, \om_L + \om_P\right],
\]
the Pfaff system $(\Sigma)$ is completely integrable.

Therefore for each $(z,a)\in P\times L$ there is a unique maximal integral manifold, which defines a fibre preserving map from a neighborhood of $z$ to a neighbourhood of~$a$. It then defines an~embedding of a neighborhood of the projection $\pi(z) \in M$ to $L/L^0$. Note that any such two embeddings differ by a left translation of an element of $L$.
Therefore
if $M$ is simply connected, the maximal integral manifold passing $(z_0, a_0) $ determines an osculating map
$\varphi\colon M \to L/L^0$, which gives the functorial isomorphism.

\subsection{From a differential equation to an extrinsic bundle}
	Let $\big(R, \{R^p\}, (M,\f), \nabla\big)$ be an $L/L^0$ differential equation. Then define $P$ to be the set of admissible frames of the $L^0$-vector bundle $R$. We see that $P$ is an $L^0$-principal fibre bundle over $M$, and $R$ can be identified with $P\times_{L^0} V$.
	
	Let $\Pi$ be the extension of $P$ to $L$: $\Pi = P\times_{L^0} L$. Then the $\lf$-connection $\nabla$ of $R$ gives rise to a connection on $\Pi$ in the following way.
	
A connection on $\Pi$ is usually defined by a right invariant distribution $H$ of horizontal subspaces, or by a connection form $\omega$, which is an $\lf$-valued 1-form on $\Pi$ satisfying
	\begin{enumerate}\itemsep=0pt
		\item[$(i)$] $R_a^*\omega ={\rm Ad}(a)^{-1}\omega$, for all $a\in L$,
		\item[$(ii)$] $\langle \omega, \tilde{A}\rangle=0$, for all $A\in \lf$.
	\end{enumerate}

	The connection $H$ associated to $\nabla$ is defined as follows. Let $\mathring{z}\in \Pi$ and $\mathring{x}\in M$ its projection. For any smooth curve $\xi(t)\in M$ with $\xi(0)=\mathring{x}$ there is a locally unique horizontal lift $\zeta(v)(t)$ to~$R$ such that $\zeta(v)(0)=\mathring{z}(v)$ for $v\in V$. Then $\zeta(t)$ is a lift of $\xi(t)$ to $\Pi$. Tangent vectors to all such lifts define a horizontal subspace $H_{\mathring{z}}\subset T_{\mathring{z}}\Pi$. The connection form $\omega$ is uniquely defined by~$\ker \omega=H$.
	
Then the following properties are standard:
\begin{prop}\label{stProp}\quad
\begin{enumerate}\itemsep=0pt
\item[$(1)$] $\nabla \zeta = \zeta (\zeta^*\omega)$ for any local section $\zeta$ of $\Pi$,
\item[$(2)$] for local sections $\zeta$ and $\sigma$ of $\Pi$ and $R$ respectively if we define $\bar\eta$ by $\sigma = \zeta\bar\eta$, then
\[
\nabla \sigma = \zeta \left({\rm d}\bar\eta +\bar\omega\bar\eta\right)\!,\qquad\text{where} \quad\bar\omega=\zeta^*\omega,
\]
\item[$(3)$] $i_{\tilde A}\big({\rm d}\omega + \tfrac{1}{2}[\omega,\omega]\big)=0$ for all $A\in\lf$,
\item[$(4)$] $K = {\rm d}\omega + \tfrac{1}{2}[\omega,\omega]$.
\end{enumerate}
\end{prop}
\begin{proof}
	Since (1) and (2) are basic, we give a short proof. Let$\mathring{x}\in M$, $\mathring{z}=\zeta(\mathring{x})$, $\xi\in T_{\mathring{x}}M$ and~$x(t)$ be a smooth curve in $M$ such that $x(0)=\mathring{x}$, $\dot x(0)=\xi$. Let $z^H(t)$ be its horizontal lift satisfying $z^H(0)=\mathring{z}$. Then writing
\[
\zeta(t) = \zeta(x(t)) = z^H(t) a(t)\qquad\text{with}\quad a(t)\in L,\quad a(0)=e,
\]
we have
\begin{gather*}
\zeta^{-1}\nabla_{\xi}\zeta =\zeta^{-1}\nabla_{\xi}(z^Ha)
= \zeta^{-1}\big((\nabla_\xi z^H)a(0)+\zeta(0)a(0)^{-1}\dot a(0)\big)=\dot a(0),
\\
\langle \zeta^*\omega, \xi\rangle = \langle \omega, \dot\zeta(0)\rangle
= \bigg\langle \omega, \dot z^H(0)a(0)+\frac{\rm d}{{\rm d}t}|_{t=0}\big(z^H(0)a(t)\big)\bigg\rangle = \dot a(0).
\end{gather*}
Hence (1) is verified.
	
	Now (2) immediately follows from (1):
	\[
		\nabla \sigma = \nabla (\zeta\bar\eta) = (\nabla\zeta) \bar\eta +\zeta {\rm d}\bar\eta = \zeta({\rm d}\bar\eta+\bar\omega\bar\eta).
	\]
	
	Note that on account of (3) $K={\rm d}\omega+\tfrac{1}{2}[\omega,\omega]$ may be regarded as a section of $\Hom\big({\wedge}^2 TM,\lf\big)$. Then (4) follows easily from (2).
	\end{proof}
	
	On the bundle $\pi\colon \Pi\to (M,\f)$ there is an induced filtration $\f^p_{\Pi}$ defined by
	\begin{gather*}
		\f^p_{\Pi} = \pi^*\f^p,\qquad p<0, \\
		\f^p_{\Pi} = \widetilde{(\phi^p\lf)},\qquad p\ge 0.
	\end{gather*}
	
	The properties of Proposition~\ref{stProp} imply
	\begin{prop}\label{propEquiv}
		Let $\nabla$ be a flat $\lf$-connection on $R$ and let $(\Pi,\omega)$ be the associated principal $L$-bundle and connection. The following two conditions are equivalent:
		\begin{enumerate}\itemsep=0pt
		\item[$(1)$] $\nabla_{\underline{\f^p}} \underline{R^q}\subset \underline{R^{p+q}}$ for all $p,q$,
			\item[$(2)$] $\omega_z\colon T_z\Pi\to \lf$ is filtration preserving for any $z\in \Pi$.
		\end{enumerate}
	\end{prop}
	
	If $\iota\colon P\to\Pi$ is the canonical inclusion, then it is easy to see that $(P, \iota^*\omega)$ in an $L/L^0$-extrinsic bundle.

\subsection{From an extrinsic bundle to a differential equation}
Let $\big(P,K^0,\om\big)$ be an $L/L^0$ extrinsic bundle over a filtered manifold $(M,\f)$. Then define a filtered vector bundle $(R,\{R^q\})$ by $R^q=P\times_{K^0} V^q$, which can be regarded as $L^0$-filtered vector bundle by group extension. The $\lf$-valued 1-form $\omega$ then defines an $\lf$-connection on $R$. It is now clear that $\left(R,\{R^q\},\nabla\right)$ is an $L/L^0$-differential equation.

It should be noted that the $L/L^0$ differential equation corresponding to the $L/L^0$ extrinsic bundle $\big(P,K^0,\omega\big)$ is written succinctly as
\[
{\rm d}\eta + \omega\cdot\eta = 0.\tag{\amgiS}
\]
Here by $\cdot$ we mean the action of the Lie algebra $\lf$ on $V$.

This is a first order linear differential equation for unknown $V$-valued function $\eta$ on $P$. Since~$\omega$ is flat, this equation is integrable. In~fact, if we view $(\amgiS)$ as a Pfaff equation on $P\times V$ by regarding $\eta$ as the standard coordinates of $V$ and put
\[
\Theta = {\rm d}\eta + \omega\cdot\eta,
\]
then we have
\[
{\rm d}\Theta = {\rm d}\om\cdot \eta - \om\wedge {\rm d}\eta = -\frac12[\om,\om]\cdot\eta -\om\wedge (\Theta -\om\cdot\eta) = -\om\wedge \Theta.
\]
Hence $(\amgiS)$ is completely integrable. Therefore for any $(\mathring{z},\mathring{v})\in P\times V$ there exists locally a unique solution $\eta$ of $(\amgiS)$ such that $\eta(\mathring{z}) =\mathring{v}$.

Moreover, since
\[
\tilde A \eta = -A\cdot \eta\qquad\text{for}\quad A\in \kf^0,
\]
we have
\[
\eta\big(z\exp(tA)\big) = \left(\exp(tA)^{-1}\right)\eta(z).
\]
Therefore $\eta$ can be regarded as a local section $\sigma$ of the associated bundle $R=P\times_{K^0} V$. Then we see immediately from Proposition~\ref{stProp}(2) that a solution $\eta$ of $(\amgiS)$ corresponds bijectively to a section $\sigma$ of $\R$ satisfying $\nabla\sigma = 0$.

\subsection{From an extrinsic geometry to a differential equation}

Let $\varphi\colon (M,\f) \to L/L^0 \subset \mathrm{Flag}(V,\phi)$
be an $L/L^0$ extrinsic geometry.
Then, taking $R^q:=\cup_{x\in M}\varphi^q(x)$ $(\varphi^q(x)\subset V)$,
$R:=M\times V$
and $\nabla_{\f^p}\underline{R^q}:=\underline{\f^p} \underline{R^q}$, that is, $\nabla={\rm d}$ for $V$-valued functions on $M$, we have an $L/L^0$ differential equation $\big(R,\{ R^p\},\nabla\big)$ on $(M,\f)$.

\subsection{From a differential equation to an extrinsic geometry}

Let $\big(R, \{R^p\}, \nabla\big)$ be an $L/L^0$ differential equation on $(M, \f )$. One way to construct the corresponding
$L/L^0$ extrinsic geometry is to take an extrinsic bundle $\big(P, K^0, \omega_P\big)$ and integrate
$\omega_P $ to obtain a map $g\colon P \to L$.
Here we like to give another construction.
We assume the base space $M$ is simply connected.
Let $\Sol(\nabla)$ denote the space of all solutions of
\[
	\nabla s = 0 .
\]
Since there exists a unique solution $s$ such that $s(x_0) = v_0$ for arbitrary given $x_0 \in M, \ v_0 \in R_{x_0}$,
we note that $\Sol(\nabla ) \cong R_{x_0}$ by evaluation and
$R_{x_0} \cong V$ by any $z_0 \in \Pi$ over $x_0$.

Now define a map
\[
\varphi\colon\ (M, \f) \to \Flag(\Sol(\nabla))
\]
by
\[
	\varphi^p(x) = \big\{ s \in \Sol(\nabla) \colon s(x) \in R^p_x \big\}.
\]
Then this map $\varphi$ is exactly the corresponding extrinsic geometry.

To say in the language of extrinsic bundles, let $\Sol(\amgiS)$ denote the solution space of $(\amgiS)$, which is isomorphic to $V$ by an evaluation
\[
\Sol(\amgiS) \ni f \mapsto f(z_0) \in V \qquad z_0 \in P.
\]

Now we define
\[
	\Psi\colon\ P \to \Flag(\Sol( \amgiS))
\]
by
\[
	\Psi^p(z) = \big\{ f \in \Sol(\amgiS) \colon f(z) \in V^q \big\},
\]
which in turn induces a map
\[
	\psi\colon\ (M, \f) \to \Flag(\Sol(\amgiS)).
\]
Then we see easily that $\varphi$ and $\psi$ are congruent,
if we regard them as maps to $\Flag(V, \phi)$
by any admissible isomorphisms $\alpha\colon\Sol(\nabla) \to V$, $\beta\colon \Sol(\amgiS) \to V$ specified above.

We remark here that, if we obtain a map
$g\colon P \to L$ such that $g^* \omega_L = \omega_P$,
then we have immediately
\[
	\Sol(\amgiS) = \big\{ g^{-1}v \colon v \in V \big\},
\]
which implies that integration of $\omega_P$ solves $(\amgiS)$.

\subsection[L/L0 differential equations as linear differential equations in weighted jet bundles]
{$\boldsymbol{L/L^0}$ differential equations as linear differential equations \\in weighted jet bundles}

Let us show that an $L/L^0$ differential equation
$(R, \{R^p\}, \nabla)$ can be realized as an involutive system of linear differential equations of
finite type defined in a weighted jet bundle (see~\cite{Mor2002} for weighted jet bundles).

Let $R^{(\nu)}$ denote the quotient bundle $R/R^{\nu+1}$, which is a filtered vector bundle over the filtered manifold $(M,\f)$, and let $\wJ^{k}R^{(\nu)}$ be the weighted jet bundle of weighted jet order $k\ge 0$ associated to $R^{(\nu)}$.

\begin{prop}\label{rprop3}
	There are canonical morphisms
	\[
	j \colon\ R^{(k)} \to \wJ^{k} R^{(\nu)}, \qquad\text{for any}\quad k,\nu,
	\]
	therefore $ R^{(k)}$ may be regarded as a system of linear differential equations
	in $\wJ^{k} R^{(\nu)}$.
\end{prop}
\begin{proof}
	The map $j$ is defined as follows. Let $v^{(k)}\in R^{(k)}_x$. Take a section $\underline v\in \underline R_x$ such that
	\[
	\nabla \underline v = 0\qquad\text{and}\qquad \underline v^{(k)}(x) = v^{(k)},
	\]
	where $\underline{v}^{(k)}$ denotes the projection of $\underline{v}$ into $R^{(k)}$. We then define
	\[
	j\big(v^{(k)}\big) = j_x^{k}\big(\underline v^{(\nu)}\big),
	\]
	where $j_x^{k}$ denotes the weighted $k$-th jet at $x$.
	
	Let us show that this definition does not depend on the choice of $\underline v$ and is well-defined. For~that we are going to verify that if $v^{(k)}=0$, then $j_x^{(k)}v^{(\nu)}=0$.
	
	Take a local cross-section $\zeta$ of $\Pi$ in a neighborhood of $x$. Then to the section, $\underline v$ corresponds a $V$-valued function $\underline u$ in a neighborhood of $x$ determined by
	\[
	\underline v(y) = \big(\zeta(y), u(y)\big).
	\]
	Taking complementary subspaces $V_p$ to $V^{p+1}$ in $V^p$, we can write
	\[
	V = \oplus V_p\qquad\text{and}\qquad \underline u = \sum \underline u_p
	\]
	with $\underline u_p$ a $V_p$-valued function.

	Now recall that the coordinates of $j^k_x\underline v^{(\nu)}$ are represented by
\[
\left(X_{p_1}\cdots X_{p_\tau}\underline u_j\right)_x\qquad\text{for}\quad X_{p_i}\in \underline{\f^{p_i}},\quad j\le \nu
\]
with $j-\sum p_i\le k$.
	
We claim that all these values vanish, if $v^{(k)}=0$. Let us verify it in a special case $\tau=2$, the other cases being similar.
	
	Since $\nabla \underline v =0 $, we have, by Proposition~\ref{stProp}(2),
	\[
	X \underline u + \langle \zeta^*\omega, X\rangle \underline u = 0.
	\]
If we write $\tilde\omega = -\zeta^*\omega$, we have
	\[
	X\underline{u} = \tilde\omega(X)\underline u = \sum\tilde\omega^j_i(X)\underline u_j,
	\]
	where $\tilde\omega^j_i(X)$ denotes $\Hom(V_j,V_i)$ component.
	
Now we have
\begin{align*}
	X_{p_1}X_{p_2}\underline u &= X_{p_1} \big( \tilde\omega(X_{p_2})\underline u\big)
	= \big( X_{p_1} \tilde\omega(X_{p_2})\big)\underline u
	+ \omega(X_{p_2})X_{p_1}\underline u
\\
	&= \big( X_{p_1} \tilde\omega(X_{p_2})\big)\underline u
	+ \tilde\omega(X_{p_2})\tilde\omega(X_{p_1})\underline u.
\end{align*}
Hence we have
\[
	X_{p_1}X_{p_2}\underline u_i = \sum_j \big(X_{p_1}\tilde\omega^j_i(X_{p_2})\big)\underline u_j
	+ \sum_{j,l} \tilde\omega^j_i(X_{p_2})\tilde\omega^l_j(X_{p_1})\underline u_l.
\]
Note here that by Proposition~\ref{propEquiv} we have
\[
	\tilde\omega^j_i(X_p) = 0 \qquad\text{for}\quad X_p\in\underline{\f^p},\qquad\text{if}\quad i<j+p.
\]
Therefore
\[
X_{p_1}X_{p_2}\underline u_i = \sum_{j\le i-p_2}\big(X_{p_1}\tilde\omega^j_i(X_{p_2})\big)
\underline u_j
+ \sum_{j\le i-p_2} \tilde\omega^j_i(X_{p_2})\sum_{l\le j-p_1}\tilde\omega^l_j(X_{p_1})\underline u_l.
\]
But the values
\begin{gather*}
\underline u_j(x) \qquad\text{for}\quad j\le i-p_2\le k,
\\
\underline u_l(x) \qquad\text{for}\quad l\le j-p_1\le i-p_1-p_2\le k
\end{gather*}
vanish. Hence
\[
	(X_{p_1}X_{p_2}\underline u_i)_x \qquad\text{for}\quad i-p_1-p_2\le k
\]
vanish, which completes the proof of Proposition~\ref{rprop3}.
\end{proof}

It should be remarked that the morphism $j$ is actually defined without ``integration''.
Though we have used integration to find $\underline v$ such that $\nabla \underline v = 0$, it is not $\underline v$, but $v^{(k)}$ that plays the actual role to define $j$ as shown in the proof.

Let $H^k(\gr\f,\gr R)$ be the degree $k$ cohomology group of the $\gr\f$ module $\gr R$. As both $\gr\f$ and $\gr R$ are graded, the cohomology group also naturally inherits the grading
\[
	H^k(\gr\f,\gr R) = \sum_i H^k_i(\gr\f,\gr R).
\]

\begin{prop}
	If $H^0_i(\gr \f, \gr R)=0$ for $i>\nu$, then the map $j\colon R^{(k)}\to \wJ^{k}R^{(\nu)}$ is injective for $k\ge\nu$.
\end{prop}
\begin{proof}
	Consider the following commutative diagram (see~\cite{Mor2002}):
	\[
	\begin{CD}
	0 @. 0 \\
	@VVV @VVV \\
	\gr_{k+1}R @>>> \Hom(U(\gr \f),\gr R^{(\nu)})_{k+1}\\
	@VVV @VVV \\
	R^{(k+1)} @>>> \wJ^{k+1}R^{(\nu)}\\
	@VVV @VVV \\
	R^{(k)} @>>> \wJ^{k}R^{(\nu)}\\
	@VVV @VVV \\
	0 @. 0
	\end{CD}
	\]
	The third row is injective for $k=\nu$ and the first row is injective for $k\ge \nu$ by assumption. Then by induction we deduce that the second row is injective for $k\ge \nu$, since the columns are~exact.
\end{proof}

\begin{prop}\label{rprop5}
	Let $k_1$ and $k_2$ be integers such that $\nu\le k_1\le k_2$, $H^1_j(\gr \f, \gr R) = 0$ for $j>k_1$ and $\gr_l R=0$ for $l>k_2$.
	
	For $k\ge\nu$, consider $R^{(k)}$ as systems of differential equations embedded in $\wJ^{k}R^{(\nu)}$. Then they satisfy
	\begin{enumerate}\itemsep=0pt
		\item[$(1)$] $\Prol^{(l)} R^{(k)} \subset R^{(l)}$ for $l\ge k\ge \nu$,
		\item[$(2)$] $\Prol^{(l)} R^{(k)} = R^{(l)}$ for $l\ge k\ge k_1$,
		\item[$(3)$] $R\to \dots \to R^{(k+1)}\to R^{(k)}$ are all isomorphisms for $k\ge k_2$ and are involutive.
	\end{enumerate}
\end{prop}
\begin{proof}
	For the subbundle $R^{(k)}\subset \wJ^{k}R^{(\nu)}$ the prolongation $\Prol^{(l)}R^{(k)}\subset \wJ^{l}R^{(\nu)}$ is defined for~$l\ge k$ (for the definition see~\cite{Mor2002}). All assertions in Proposition~\ref{rprop5} can be easily verified by~standard arguments in the theory of weighted jet bundles.
\end{proof}

Existence of such integers $\nu$, $k_1$, $k_2$ is theoretically obvious in our finite-dimensional case. Such minimal integers are of interest for various concrete examples.

\subsection{Dual embeddings and differential equations}
Let $V$ be a vector space and $V^*$ its dual space. For a decreasing filtration $\phi=\{\phi^p\}$ we define the dual of $\phi$ to be the filtration $\phi^*$ of $V^*$ defined by
\[
\phi^{*p} = \big(\phi^{-p+1}\big)^\perp.
\]
Then the pairing
\[
\gr_p\phi \times \gr_q\phi^* \to \R
\]
is non-degenerate if $p+q=0$ and vanishes otherwise. Therefore
\[
(\gr\phi)^* \cong \gr\phi^*.
\]

For a map $\varphi\colon (M,\f) \to \Flag(V,\phi)$ we define its dual
$\varphi^*\colon (M,\f)\to \Flag(V^*,\phi^*)$ by $\varphi^*(x) = \varphi(x)^*$ for $x\in M$.

\begin{prop}\label{dual}\
	\begin{enumerate}\itemsep=0pt
		\item[$1.$] The map $\varphi$ is osculating if and only if so is $\varphi^*$.
	\end{enumerate}
 Assume that $\varphi$ is osculating. Then
 \begin{enumerate}\itemsep=0pt
 	\item[$2.$] $(\gr \varphi(x))^* \cong \gr\varphi^*(x)$ and
 	\[
 	 \langle X\alpha,v\rangle + \langle \alpha, Xv\rangle=0\qquad\text{for}\quad X\in\gr\f_x,\quad
 \alpha\in \gr\varphi^*(x),\quad v\in\gr\varphi(x).
 	\]
 	\item[$3.$] $H^0_i\left(\gr\f_x, \gr\varphi(x)\right)=0$ for $i>\lambda$ if and only if $\gr\varphi^*(x)$ is generated by $\oplus_{j\ge -\lambda} \gr_j\varphi^*(x)$.
 \end{enumerate}
\end{prop}
The proof is straightforward and is omitted.

It should be remarked that if $\varphi\colon (M,\f)\to L/L^0\subset \Flag(V,\phi)$ is a~$L/L^0$-extrinsic geometry, then $\varphi^*$ is also an $L/L^0$-extrinsic geometry with respect to the dual embedding $L/L^0\subset \Flag(V^*,\phi^*)$.

We also note that the duality in $L/ L^0$ extrinsic geometries naturally extends to that in $L/L^0$ differential equations.

\subsection{Examples}
In the examples below we illustrate the relation between differential equations, the corresponding osculating embeddings and their dual embeddings. Throughout these examples we denote by $V$ (or $V^{k+1}$) a vector space of dimension $k+1$, and by $V^*$ its dual. By taking a basis $\{e_0,e_1,\dots,e_k\}$ of $V$ and its dual basis $\{e^0,e^1,\dots,e^k\}$ of $V^*$ we make often the identifications
\begin{gather*}
V\ni v =\sum v^i e_i \leftrightarrow \begin{pmatrix} v^0 \\ v^1 \\ \vdots \\v^k \end{pmatrix} \in \R^{k+1},
\\
V^*\ni\alpha = \sum\alpha_i e^i \leftrightarrow (\alpha_0,\alpha_1,\dots,\alpha_k)\in \R^{k+1,*}.
\end{gather*}
We write $M$ or $M^n$ to denote a filtered manifold of dimension $n$ endowed with the trivial filtration unless otherwise mentioned.

\begin{ex}\label{ex4new} Consider a curve $[\theta]\colon M^1\to P\big(V^{k+1,*}\big)$ represented by a smooth curve $\theta\colon M^1\to V^{k+1,*}\backslash \{0\}$. Assume that $[\theta]$ is regularly generating $V^*$, that is there exists an osculating map $\varphi^*\colon M\to \Flag(V^*,\phi^*)$ such that $\varphi^{*0}(x) = \langle \theta(x) \rangle\supset \varphi^{*1}(x)=0$ for all $x\in M$ and that $\varphi^*$ is generated by $\varphi^{*0}$. Note that $\varphi^*$ is uniquely determined by $[\theta]$ up to a shift of filtration. In~this case this assumption is equivalent to assuming that $\big\{\theta(x), \theta'(x),\dots, \theta^{(k)}(x)\big\}$ are linearly independent and therefore span $V^*$ for every $x\in M$.

Then there exist smooth functions $p_0,p_1,\dots,p_k$ on $M$ such that
\[
\theta^{(k+1)}(x) = p_0(x)\theta^{(k)}(x) + \cdots + p_k(x)\theta(x),\qquad \text{for all}\quad x\in M.
\]
This means that if we write $\theta=(\theta_0,\theta_1,\dots,\theta_k)$ then $\{\theta_0,\theta_1,\dots,\theta_k\}$ forms a fundamental system of solutions of the following linear ordinary differential equation $(\cD)$ of order $k+1$:
\begin{equation*}
y^{(k+1)} = p_0 y^{(k)} + p_1y^{(k-1)} + \dots + p_k y.\tag{$\cD$}
\end{equation*}

Let $\varphi\colon M\to\Flag(V,\phi)$ be the dual of $\varphi^*$, i.e., $\varphi = (\varphi^*)^*$. Then by Proposition~\ref{dual} the map~$\varphi$ is osculating and $H^0_{+}(\gr \f_x,\gr\varphi(x))=0$. Hence, viewing $L={\rm GL}(V)$, $L^0=\phi^0{\rm GL}(V)$, we~see that the $L/L^0$-differential equation $\cR=\big(\{\cR^q\},\nabla\big)$ corresponding to $\varphi$ is defined in the jet bundle $J^{k+1}(\cR^{(0)})$ by Proposition~\ref{rprop5}.

Let us observe that the differential equation~$\cR$ is nothing but the ODE~$(\cD)$. Indeed, define $\psi \colon M\to \Flag(\Sol(\cD))$ by
\begin{gather*}
\Sol(\cD) = \left\{s_v = \sum_{i=0}^k v^i\theta_i\,\left|\,v=\begin{pmatrix} v^0 \\[-1mm] \vdots \\ v^k \end{pmatrix} \in \R^{k+1}\right.\right\},
\\[.5ex]
\psi^p(x) = \left\{s_v\in \Sol(\cD)\mid s^{(i)}_v(x)=0,\text{ that is }
\big\langle \theta^{(i)},v \big\rangle = 0 \text{ for all } i<p \right\}.
\end{gather*}
Clearly $\psi$ is isomorphic to $\varphi$. But the differential equation corresponding to $\psi$ is~$(\cD)$, so that the systems $(\cD)$ and $\cR$ are equivalent.

On the other hand, we know that the differential equation $\cR$ is equivalent to the Pfaff equation
\[
{\rm d}\eta +\omega\eta = 0,\tag{\amgiS}
\]
where $\omega$ is the pull-back of the Maurer--Cartan form $\Omega_{{\rm GL}(V)}$ to the induced bundle $\varphi^*{\rm GL}(V)$:
\[
\begin{CD}
	\varphi^* {\rm GL}(V) @>>> {\rm GL}(V) \\
	@VVV @VVV \\
	M @>>> {\rm GL}(V)/\phi^0{\rm GL}(V).
\end{CD}
\]
Let us see how the equation~$(\amgiS)$ looks like in our case.

Let $\{\Theta_0,\Theta_1,\dots,\Theta_k\}$ be a moving frame of $V^*$ defined by $\Theta_p=\theta^{(k-p)}$, $p=0,1,\dots,k$ and let $\big\{\Phi^0,\Phi^1,\dots,\Phi^k\big\}$ be the moving frame on of $V$ dual to $\{\Theta_0,\Theta_1,\dots,\Theta_k\}$.

Regarding $\Theta_q$ as a row vector and $\Phi^p$ as a column vector, we set
\[
\Theta = \begin{pmatrix} \Theta_0 \\ \Theta_1 \\ \vdots \\ \Theta_k \end{pmatrix}\!,\qquad
\Phi = \big(\Phi^0,\Phi^1,\dots,\Phi^k\big),
\]
and we get ${\rm GL}(V)$-valued functions $\Theta$, $\Phi$ on $M$ satisfying
\[
\Theta\Phi = E_{k+1} \qquad\text{(the identity matrix of degree $k+1$)}.
\]
It is clear that $\Phi$ gives a cross-section of the bundle $\varphi^*{\rm GL}(V)\to M$. Thus, the equation~$(\amgiS)$ reduces to
\[
{\rm d}\eta + \Phi^{-1}\,{\rm d}\Phi\,\eta = 0,
\]
which then becomes
\[
{\rm d}\eta - {\rm d}\Theta\,\Theta^{-1}\,\eta = 0.
\]
But by the construction of $\Theta$ we see that
\[
{\rm d}\Theta\,\Theta^{-1} =
\begin{pmatrix} p_0 & p_1 & p_2 & \dots & p_{k-1} & p_k \\
1 & 0 & 0 & \dots & 0 & 0\\
0 & 1 & 0 & \dots & 0 & 0\\
\vdots & \vdots & \ddots & \ddots & \vdots & \vdots \\
0 & 0 & 0 & \dots & 0 & 0 \\
0 & 0 & 0 & \dots & 1 & 0
\end{pmatrix}{\rm d}x.
\]
Thus, the Pfaff equation~$(\amgiS)$ reduces to
\[
\begin{pmatrix}
\eta_k' \\
\eta_{k-1}'\\
\vdots \\
\eta_0'
\end{pmatrix} =
\begin{pmatrix} p_0 & p_1 & p_2 & \dots & p_{k-1} & p_k \\
1 & 0 & 0 & \dots & 0 & 0\\
0 & 1 & 0 & \dots & 0 & 0\\
\vdots & \vdots & \ddots & \ddots & \vdots & \vdots \\
0 & 0 & 0 & \dots & 0 & 0 \\
0 & 0 & 0 & \dots & 1 & 0
\end{pmatrix}
\begin{pmatrix}
\eta_k \\
\eta_{k-1}\\
\vdots \\
\eta_0
\end{pmatrix}\!.
\]
This is exactly the system of ODEs of 1st order equivalent to $(\cD)$.

The above discussion gives our formulation of the well-known correspondence of the category of scalar linear differential equations of order $k+1$ and the category of non-degenerate curves in $P^k$ explored by Wilczynski~\cite{wilch}.

In the most symmetric case we arrive at a rational normal curve $C\subset P\big(\R^{k+1,*}\big)$ given by the (dual) Veronese embedding
\[
[\theta_\text{{Vero}}] \colon\quad P\big(\R^2\big)\to P\big(\R^{k+1,*}\big),\qquad
\left[\begin{pmatrix}z^0\\ z^1\end{pmatrix}\right]\mapsto \big[\big({-}z^1\big)^k, \big({-}z^1\big)^{k-1}z^0, \dots, \big(z^0\big)^k\big].
\]
In terms of representations, it is the ${\rm GL}(2,\R)$-orbit in $P\big(S^k\R^{2,*}\big)$ of the highest weight vector $\big(\varepsilon^1\big)^k\in S^k\R^{2,*}$, where $\big\{\varepsilon^0,\varepsilon^1\big\}$ is the standard basis of $\R^{2,*}$, and $S^k\R^{2,*}$ denotes the $k$-th symmetric power of $\R^{2,*}$, which has a basis $\big\{ \big(\varepsilon^1\big)^k, \big(\varepsilon^1\big)^{k-1}\varepsilon^0, \dots, \big(\varepsilon^0\big)^k \big\}$. The dual representation of~${\rm GL}(2,\R)$ on $\R^{2,*}$ extends naturally to $P\big(S^k\R^{2,*}\big)$. The stabilizer of~$\big[\big(\varepsilon^1\big)^k\big]$ is $B$, the subgroup of upper-triangular matrices in ${\rm GL}(2,\R)$. Thus we have the embedding
\[
P\big(\R^2\big)={\rm GL}(2,\R)/B \to {\rm GL}(2,\R)\cdot \big[\big(\varepsilon^1\big)^k\big],
\]
whose coordinate expression is the one given above.

In terms of the affine coordinate $x=z^1/z^0$ the curve $[\theta]$ is given by
\[
\theta_{\text{Vero}} \colon\ \R^1\ni x\mapsto \big( (-x)^k, (-x)^{k-1},\dots, (-x), 1\big) \in \R^{k+1,*}.
\]
The components of $\theta$ form a fundamental system of solutions of $y^{(k+1)}=0$.

Thus, starting from the dual Veronese embedding $[\theta]_{\text{Vero}}\colon P\big(\R^2\big) \to P\big(\R^{k+1,*}\big)$, we have an~oscu\-lating map $\varphi^*_{\text{Vero}}\colon P\big(\R^2\big)\to \Flag\big(\R^{k+1,*},\phi^*\big)$ generated by $[\theta]_{\text{Vero}}$ and its dual $\varphi_{\text{Vero}}$: $P\big(\R^2\big)\to \Flag\big(\R^{k+1},\phi\big)$. Then the differential equation corresponding to $\varphi_{\text{Vero}}$ is equivalent to the simplest one $y^{(k+1)}=0$.
\end{ex}

\begin{ex}
Consider a smooth surface in 3-dimensional projective space $[\theta]\colon M^2\to P\big(V^{4,*}\big)$ represented by a smooth map $\theta\colon M^2\to V^{4,*}\backslash \{0\}$. Assume that $[\theta]$ generates an osculating map
\[
\varphi^* \colon\ M\to \Flag(V^*,\phi^*)
\]
with $\varphi^{0*}(x)=\langle \theta(x)\rangle \supset \varphi^{1*}(x)=0$, $x\in M$.

There are two cases to distinguish for the type of the map $\varphi^*$, that is when the sequence $\big(\dim V,\dots,\dim \phi^{-1},\dim \phi^0\big)$ is equal to $(4,3,2,1)$ (case $(i)$) or to $(4,3,1)$ (case $(ii)$). Taking local coordinates $(x,y)$ on $M$ and identifying $V^{4,*}$ with $\R^{4,*}$ write
\[
\theta = \big(\theta_0(x,y),\dots, \theta_3(x,y)\big).
\]
In case $(i)$ we see that
\begin{gather*}
\dim \big\langle \theta(p) \big\rangle = 1,\\[.3ex]
\dim \big\langle \theta(p), \theta_x(p), \theta_y(p) \big\rangle = 2,\\[.3ex]
\dim \big\langle \theta(p), \theta_x(p), \theta_y(p), \theta_{xx}(p), \theta_{xy}(p), \theta_{yy}(p) \big\rangle = 3,\\[.3ex]
\dim \big\langle \theta(p), \theta_x(p), \dots, \theta_{yyy}(p) \big\rangle = 4.
\end{gather*}
By a choice of coordinates we can assume that $\theta_y=0$, from which it follows that $\theta_{xy}=\theta_{yy}=0$, so that $\theta$, $\theta_x$, $\theta_{xx}$, $\theta_{xxx}$ are independent and the other derivatives of $\theta$ are expressed as their linear combinations. Thus, we have
\[
\begin{cases}
\theta_y=0,\\
\theta_{xxxx} = p_0\theta_{xxx}+p_1\theta_{xx}+p_2\theta_{x}+p_3\theta.
\end{cases}
\]
Let $\varphi$ be the dual of $\varphi^*$. Then by the same argument as in Example~\ref{ex4new}, we see that the differential equation corresponding to $\varphi$ is equivalent to
\begin{equation}
\tag{$\cD_i$}
\begin{cases}
u_y=0,\\
u_{xxxx} = p_0u_{xxx}+p_1u_{xx}+p_2u_{x}+p_3u,
\end{cases}
\end{equation}
which reduces essentially to a linear ODE.

Next, let us consider case $(ii)$. For each point $a\in M$
\[
\gr \varphi^*(a) = \gr_{-2}\varphi^*(p)+\gr_{-1}\varphi^*(a)+\gr_0\varphi^*(a)
\]
is a $\gr\f_a$-graded module. Here we have $\dim \gr_{-2}\varphi^*=\dim\gr_{0}\varphi^*=1$, $\dim\gr_{-1}\varphi^*=2$ and $\gr\f_a=T_aM$ is a 2-dimensional abelian Lie algebra. If we take bases $\{\zeta_{-2}\}$ and $\{\zeta_{0}\}$ of $\gr_{-2}\varphi^*$ and $\gr_0\varphi^*$ respectively, then we have a symmetric bilinear form $\beta$ on $T_aM$ defined by $uv\zeta_0 = \beta(u,v)\zeta_{-2}$, which is equal to $uv\theta \mod \theta(a),\theta_x(a),\theta_y(a)$ up to scalar multiplication.

Note that if the surface $[\theta]\colon M\to P(V^{4,*})$ is defined in the affine coordinates by $z=f(x,y)$, that is $\theta(x,y)=(1,x,y,f(x,y))$, then the bilinear form $\beta$ coincides with the Hessian of $f$ at $a$.

Again, there are several cases to distinguish: ($a$) $\beta$ is indefinite, that is the signature of $\beta$ is~$(1,1)$; ($b$) $\beta$ is definite; ($c$) $\beta$ is degenerate.

Let us consider case (a). At each point $a\in M$ there exists a direct sum decomposition $T_aM=E_a\oplus F_a$ wtih $\dim E_a=\dim F_a=1$ such that $\beta(v,v)=0$ for $v\in E_a$ or $v\in F_a$ ($E$ and $F$ are called asymptotic directions of $\theta$). Then take local coordinates $x,y$ so that $x$ and $y$ are 1st integrals of $E$ and $F$ respectively. With this choice of coordinates we get
\[
\theta_{xy} \not\equiv 0, \qquad \theta_{xx}\equiv\theta_{yy}\equiv 0 \mod \theta,\theta_x,\theta_y.
\]
Therefore there exist functions $a_i$, $b_i$, $c_i$ of $(x,y)$ ($i=1,2$) such that $\{\theta_0,\theta_1,\theta_2,\theta_3\}$ is a fundamental system of solutions of
\begin{equation}
\tag{$\cD_{ii-a}$}
\begin{cases}
u_{xx} = a_1u_x+b_1u_y+c_1u,
\\
u_{yy} = a_2u_x+b_2u_y+c_2u.
\end{cases}
\end{equation}
Thus, we have shown that for a surface $[\theta]\colon M^2\to P\big(V^{4,*}\big)$ of signature $(1,1)$ the differential equation corresponding to $\varphi\colon M\to \Flag(V,\phi)$ is equivalent to $(\cD_{ii-a})$, where $\varphi^*$ is the osculating map generated by $[\theta]$ and $\varphi$ is dual to $\varphi^*$.

Conversely, suppose we are given a system of differential equations of the form $(\cD_{ii-a})$. Note first that since all higher order derivatives $u_{xxx}, u_{xxy}, \dots $ can be expressed via $u$, $u_x$, $u_y$, and~$u_{xy}$, this system is of finite type and the dimension of solution space is smaller or equal to $4$. The equality holds only when all compatibility conditions are satisfied. If this is the case, then the surface in $P\big(\R^{4,*}\big)$ defined by a fundamental system of solutions of $(\cD_{ii-a})$ has signature~$(1,1)$.

Further study of the remaining cases ($b$) and ($c$) is left to the reader. See also~\cite{wilch2} and \cite{sasaki2, sasaki1} for more detailed studies of case ($a$).

The most symmetric case of the non-degenerate surface in $P\big(V^{4,*}\big)$ is given by the Segre embedding
\[
P(U)\times P(U^*) \ni \big([u],[\alpha]\big)\mapsto [u\otimes\alpha]\in P(U\otimes U^*),
\]
where $U$ is a 2-dimensional vector space.

The determinant on $U\otimes U^*=\Hom(U,U)$ defines a quadratic form $\delta$ on $U\otimes U^*$ of signa\-ture~$(2,2)$, and $\delta(u\otimes \alpha)=0$, so that the surface is a quadric in $P(U\otimes U^*)$.

Note that the group $G={\rm SL}(U)\times {\rm SL}(U)$ acts on $U\otimes U^*$ by
\[
\rho(a,b)v\otimes \alpha = av\otimes b^{*,-1}\alpha\qquad\text{for}\quad
v\in U,\quad \alpha\in U^*,\quad a,b\in {\rm SL}(U),
\]
and therefore acts on $(U\otimes U^*) = U^*\otimes U$ by $\rho(a,b)^{*,-1}$.

Let $\{\varepsilon_0,\varepsilon_1\}$ and $\left\{\varepsilon^0,\varepsilon^1\right\}$ be bases of $U$ and $U^*$ dual to each other. Then the maps $\omega$ and $\theta$ defined by
\[
\begin{tikzcd}
& \rho(a,b)\left(\varepsilon_0\otimes \varepsilon^{1}\right)\subset V=U\otimes U^* \\
G={\rm SL}(U)\times {\rm SL}(U)\ni(a,b)
\arrow[ur, "\omega", start anchor={[xshift=6ex]}, end anchor={[xshift=-4ex]}]
\arrow[dr, "\theta", start anchor={[xshift=6ex]}, end anchor={[xshift=-4ex]}] & \\
& \!\!\!\!\rho^{*,-1}(a,b)\left(\varepsilon^1\otimes \varepsilon_0\right)\subset V^*
\end{tikzcd}
\]
give rise to the embeddings
\[
[\omega]_{\text{Seg}}\colon\ M\to P(V)\qquad\text{and}\qquad
[\theta]_{\text{Seg}}\colon\ M\to P(V^*),
\]
where $M=P(U)\times P(U^*)$ and $V=U\otimes U^*$. Then $[\omega]_{\text{Seg}}$ and $[\theta]_{\text{Seg}}$ generate osculating mappings respectively
\[
\varphi_{\text{Seg}}\colon\ M \to \Flag(V,\phi)\qquad\text{and}\qquad
\varphi_{\text{Seg}}^*\colon\ M\to\Flag(V^*,\phi^*),
\]
which are dual to each other.

We have
\[
\theta \left( \begin{pmatrix}\xi^0\\\xi^1\end{pmatrix}\otimes (\eta_0, \eta_1)\right) =
\left(\varepsilon^0,\varepsilon^1\right)
\begin{pmatrix}
\big({-}\xi^1\big)\eta_1 & \big({-}\xi^1\big)(-\eta_0) \\
\xi^0\eta_1 & \xi^0(-\eta_0)
\end{pmatrix}
\begin{pmatrix}
\varepsilon_0 \\
\varepsilon_1
\end{pmatrix}\!.
\]

In affine coordinates $x=\xi_1/\xi_0$, $y=\eta_1/\eta_0$, the components of $\theta$ are $\{-xy, x, y, -1\}$ and form a fundamental system of solutions of
\begin{gather*}
\begin{cases}
u_{xx}=0, \\
u_{yy}=0.
\end{cases}
\end{gather*}
This is precisely the differential equation that corresponds to the Segre embedding $\varphi_{\text{Seg}}\colon M^2\to \Flag\big(V^4,\phi\big)$.
\end{ex}

\begin{ex}\label{ex:A2}
	Consider a submanifold $[\theta]\colon (M,\f)\to P(V^*)$ represented by a smooth map $\theta\colon M\to V^*\backslash\{0\}$, where $M$ is a 3-dimensional contact manifold endowed with the filtration $\f$ induced from a contact distribution $D=\f^{-1}$. To fix the notation we assume $M=\R^3$ with the standard coordinates $(x,y,z)$ and the contact distribution defined by the contact form:
	\[
	\omega = {\rm d}z+\frac{1}{2}(x\,{\rm d}y-y\,{\rm d}x).
	\]
	Therefore the contact distribution is spanned by
	\[
	X = \frac{\partial}{\partial x} + \frac{1}{2}y\frac{\partial}{\partial z},\qquad
	Y = \frac{\partial}{\partial y} - \frac{1}{2}x\frac{\partial}{\partial z}
	\]
	and $\big\{X,Y,Z\big({=}\frac{\partial}{\partial z}\big)\big\}$ forms a moving frame on $M$. Note that $[X,Y]=-Z$ and we count $\word X=\word Y = 1$, $\word Z=2$ with respect to the weighted order associated with the contact filtration $\f$.
	
	Suppose that $\theta\colon M\to V^*\backslash\{0\}$ is regularly generating $V^*$, that is there exists an osculating map $\varphi^*\colon (M,\f)\to \Flag(V^*,\phi^*)$ generated by $\theta$ and $\varphi^{*0}=\langle \theta\rangle \supset \varphi^{*1}=0$.
	
	Recalling that $1,X,Y,Z,X^2,XY,Y^2,XZ,YZ,X^3,X^2Y,\dots$ form the basis of the ring $\cD$ of~dif\-ferential operators on $M$, we may study possible type of $\phi^*$ by $\left(\dim \phi^{*0}, \dim \phi^{*-1},\dots\right)$.
	
	If we assume that $\theta$, $X\theta$, $Y\theta$ are independent, then the first case to examine will be the case of type $(1,3,4,\dots)$. Suppose, for instance, that
	\[
		X^2\theta,\,XY\theta,\, Y^2\theta \equiv 0 \mod (\theta,X\theta,Y\theta).
	\]
	Then the components of $\theta$ satisfy a system of differential equations of the following form:
	\begin{equation}\label{eq-ex6-1}
	\begin{cases}
	X^2u = A_1 Xu+B_1 Yu + C_1u,\\
	Y^2u = A_2 Xu+B_2 Yu + C_2u,\\
	XYu = A_3 Xu+B_3 Yu + C_3u,
	\end{cases}
	\end{equation}
	where $A_i$, $B_i$, $C_i$ ($i=1,2,3$) are functions on $M$. It is easy to see that $\{1,X,Y,Z\}$ and $\cD\big\langle X^2,XY,Y^2\big\rangle$ are complementary and generate $\cD$. Therefore the system~\eqref{eq-ex6-1} is of finite type with the dimension of the solution space $\le 4$. Therefore in this case we have $\dim V=4$ and the type of $\phi$ is $(1,3,4)$. For example, if the right hand side of equation~\eqref{eq-ex6-1} vanishes identically, its solution space is spanned by functions $\{1,x,y,z+xy/2\}$.
	
{\samepage	Next, consider the case of type $(1,3,5,\dots)$. If we assume
	\[
		X^2\theta,\, Y^2\theta \equiv 0 \mod (\theta,X\theta,Y\theta),
	\]
 	we find an interesting class of differential equations
	\begin{gather}\label{eq-ex6-2}
	\begin{cases}
	X^2u = A_1 Xu+B_1 Yu + C_1u,\\
	Y^2u = A_2 Xu+B_2 Yu + C_2u,\\
	\end{cases}
	\end{gather}
	where $A_i$, $B_i$, $C_i$ are functions on $M$.

}
	
	It is easy to see that $\big\{1,X,Y,XY,Z,XZ,YZ,Z^2\big\}$ are complementary to $\cD\big\langle X^2,Y^2\big\rangle$ and generate the whole $\cD$ together with it. Therefore the solution space of this system is $\le 8$. If the equality holds, then $\dim V=8$ and the type of $\phi^*$ is $(1,3,5,7,8)$.
	
	Now assume that $\theta\colon M\to V^{*,8}$ satisfies the above conditions. Then in this case also by~the same argument as in Example~\ref{ex4new}, we see that the differential equation corresponding to~$\varphi\colon M\to \Flag(V,\phi)$ is equivalent to~\eqref{eq-ex6-2}, where $\varphi$ is the dual to $\varphi^*\colon M\to \Flag(V^*,\phi^*)$.
	
	The most symmetric model of $\varphi\colon M\to \Flag\big(V^8,\phi\big)$ is obtained by the adjoint representation of ${\rm SL}(3,\R)$ on the Lie algebra $\mathfrak{sl}(3,\R)$ as shown in Section~\ref{rank2ex}.
\end{ex}

\section[Equivalence problems, extrinsic normal Cartan connections and invariants]
{Equivalence problems, extrinsic normal Cartan connections\\ and invariants}\label{sec5}

In this section we study the equivalence problem of $L/L^0$ extrinsic geometries and $L/L^0$ dif\-fe\-ren\-tial equations via $L/L^0$ extrinsic bundles. We will consider exclusively those objects that are of constant symbol, say of type $(\g_- , \gr V, L)$. Otherwise, we have to deal with principal fibre bundles with varying structure groups as studied in~\cite{Mor1983}.

The main idea here for the equivalence problem is, starting from a $L/L^0$ extrinsic bundle, to construct, in a canonical manner,
a nice reduction (a subbundle ) whose structure function represents effectively the invariants of the original structure. We shall single out the conditions for the symbol $(\g_-, \gr V, L)$ under which we can construct what we call an extrinsic Cartan connection associated with the original structure. The general case in which it is no more possible to associate Cartan connection is treated in Section~\ref{sec:general}.

\subsection[Relative prolongations, standard models and extrinsic cohomology groups]
{Relative prolongations, standard models and extrinsic \\cohomology groups}

As in Section~\ref{sec:cat}, let $(V, \phi)$ be a filtered vector space,
$L \subset {\rm GL}(V) $ a Lie subgroup, $L^0$ the Lie subgroup of $L$
which fixes $\phi$, $\lf$ and $\lf^0$ the Lie algebras of $L$ and $L^0$
respectively. Denote again by $\phi$ the induced filtration on
$\mathfrak{gl}(V)$ and $\lf$. In~this section we assume that:

\begin{enumerate}\itemsep=0pt
	\item[(C0)]
	There exists a filtration preserving identification of $V$ and $\gr V$ that induces the isomorphism of $\lf\subset \gl(V)$ and $\gr(\lf, \phi)\subset \gl(\gr V)$.
\end{enumerate}

Using these identifications, we will write $V = \bigoplus V_p$, $\lf = \bigoplus \lf_p$, so that $\phi^pV = \bigoplus _{ i \ge p} V_i$ and $\phi^p\lf = \bigoplus _{ i \ge p} \lf_i$.

Let $\g_-$ be a graded Lie subalgebra of $\lf_-$. We define $\bar{\g} = \bigoplus \bar{\g}_p$ inductively as follows:
For $p < 0$, $\bar{\g}_p = \g_p $ and for $p \ge 0$,
\[
\bar{\g}_p = \big\{ A \in \lf _p \colon [A, \g_-] \subset \textstyle\bigoplus_{ i<p} \bar{\g}_i \big\}.
\]
Then as easily seen, $\bar{\g}$ is a graded Lie subalgebra of $\lf$.

\begin{df}
	The graded Lie algebra $\bar{\g}$ defined above is called
	\emph{the relative prolongation} of~$\g_-$ with respect to $\lf$
	and denoted $\Prol( \g_-, \lf )$.
\end{df}

Define the following subgroups of $L$ corresponding to subalgebras $\bar{\g}_0$, $\g_{-}$, $\bar{\g}^0$ respectively:
\begin{gather*}
\bar{G}_0 = \big\{ a \in L \colon {\rm Ad}(a) \g_- \subset \g_-, \text{ and } a V_p \subset V_p, \, \forall p \big\}, \\
\bar G_{-} = \exp(\g_{-}),\\
\bar G^0 = \bar G_0 \exp(\oplus_{p>0}\bar\g_p).
\end{gather*}
For simplicity we assume that there exists a Lie subgroup $\bar G\subset L$ corresponding to the subalgebra $\bar\g\subset \lf$ and containing the subgroup $\bar G_0$ (and thus $\bar G^0$). Note that our main results do not depend on the existence of such subgroup and are based only on the subgroup $\bar G^0$ and the subalgebra $\bar\g$.

Let $S$ be any subgroup of $L$ such that $G_-=\bar G_{-} \subset S \subset \bar G$ and $S^0 = S \cap L^0$. Then the homogeneous space $S/S^0$, endowed with the canonical invariant tangential filtration, is a~standard filtered manifold of type~$\g_-$.

\begin{df}
	The induced map
	\[
	\varphi_{{\rm model}\, S}\colon\ S/S^0 \to L/L^0 \subset \Flag(V, \phi)
	\]
	is called a standard model immersion of type $(\g_-, V, L)$.
\end{df}
Note that all $\varphi_{{\rm model} \, S}$ are locally isomorphic and contain identical open orbit of the subgroup~$G_{-}$. So that we will not specify $S$ and write $\varphi_{{\rm model}}$ unless needed.

According to the categorical isomorphisms in Section~\ref{sec:cat},
we may also speak of standard model of $L/L^0$ differential equation, or extrinsic bundle, of type $(\g_-, \gr V, L)$.

In particular the principal fibre bundle $\bar G \to \bar G/ \bar{G}^0$
equipped with the pull-back of the Maurer--Cartan form $\omega_L$ is a standard model of $L/L^0$ extrinsic bundle of type $(\g_-, \gr V, L)$.

Now we introduce the cohomology group which plays an important role in extrinsic geometry. The quotient space $\lf / \bar{\g}$ is a $\bar{\g}$-module and a fortiori $\g_-$-module. There is the chain complex and the cohomology group $H^p(\g_{-}, \lf/\bar\g)$ associated with the representation
of $\g_-$ on $\lf / \bar{\g}$:
\[
0 \longrightarrow \lf/\bar{\g}
\overset{\partial}\longrightarrow \Hom\left(\g_-, \lf/\bar{\g}\right)
\overset{\partial}\longrightarrow \Hom \left( \wedge^2 \g_-, \lf/\bar{\g}\right)
\overset{\partial}\longrightarrow \cdots.
\]
Noting that the coboundary operator preserves the filtration associated to $\phi$, we actually consider the subcomplex
\[
0 \longrightarrow \phi^1\lf/\bar{\g}
\overset{\partial}\longrightarrow \phi^1\Hom\left(\g_-, \lf/\bar{\g}\right)
\overset{\partial}\longrightarrow \phi^1 \Hom \left( \wedge^2 \g_-, \lf/\bar{\g}\right)
\overset{\partial}\longrightarrow \cdots.
\]

\subsection{Auxiliary groups, complementary subspaces}
For $k \ge 0 $ we define Lie subalgebras $\bar{\g}(k) $ of $\lf$ by
\[
\bar{\g}(k) = \bar{\g}_0 \oplus \cdots \oplus \bar{\g}_k \oplus {\lf}_{k+1} \oplus \cdots
\]
and the Lie subgroups by
\[
\bar{G}(k) = \bar{G}_0 \, {\rm exp}\left( \phi^1\bar{\g}(k)\right)\!,
\]
where
$\phi^1\bar{\g}(k) = \bar{\g}_1 \oplus \cdots \oplus \bar{\g}_k \oplus {\lf}_{k+1} \oplus \cdots
= \phi^1\lf \cap \bar{\g}(k)$. By construction $\bar G(k)\subset \bar G(k-1)$ and~$\bar G(k)=\bar G^0$ for $k$ bigger than the maximal $i$ such that $\lf_i\ne 0$. Clearly the Lie algebra $\bar{\g}^0$ of~$\bar{G}^0$ is~$\phi ^0\bar{\g}$.
We put
\[
E(k) = \g_- \oplus \bar{\g}(k),
\]
which are graded vector spaces converging to the Lie algebra $\bar{\g}$.

Taking an arbitrary complementary graded subspace $\bar{\g}' = \bigoplus \bar{\g}'_p$
of $\bar{\g} $ in $\lf$, we fix a direct sum decomposition
\[
\lf = E( k) \oplus E(k)' ,
\]
where $E(k) ' = \bigoplus_{i \le k} \bar{\g}'_i $.

\subsection[First reduction, bundle Q(0)]
{First reduction, bundle $\boldsymbol{Q(0)}$}\label{ss42}

Let $V$, $\lf$, $L$, $\geneg $ be as above. Let
\[
\varphi\colon\ (M, \f) \to L/L^0 \subset \Flag(V, \phi)
\]
be an osculating immersion of constant symbol of type $(\geneg, \gr V, L)$.

Let $\big(Q(-1), \omega (-1)\big)$ be the $L/L^0$ extrinsic bundle corresponding to
the $L/L^0$ extrinsic geo\-met\-ry $\varphi$, or more precisely, $Q(-1)$ is the induced bundle $\varphi^*L$ of the bundle $L \to L/L^0$
by the map $\varphi\colon M \to L/L^0$ and
$\omega(-1)= \Phi(-1)^*\omega_L$,
where $\omega_L$ denotes the Maurer--Cartan form of $L$ and
$\Phi(-1)\colon\varphi^*L \to L$ the canonical inclusion.

Since $\varphi$ is of constant symbol of type $(\geneg, \gr V, L)$, so is
$Q(-1)$. Therefore for each $x \in M$ there exists $z\in Q(-1)_x$ such that
\[
\gr \omega_z (\gr \f_x) \subset \g _-.
\]
Let $Q(0)$ be the totality of such $z $. Then $Q(0)$ is a principal fibre bundle over$M$ with structure group $\bar G(0) $, and the pair $(Q(0), \omega(0))$ proves to be a $L/L^0$ extrinsic bundle congruent to $Q(-1)$, where $\omega(0)$ is the pull-back of $\omega(-1)$.

We may repeat the construction of $ (Q(0), \omega(0))$
more directly from the map $\varphi $ as follows:

For $x \in M $ we define $Q(0)_x $ to be the set of all filtration preserving isomorphism $a\colon (V, \phi) \to (V, \varphi(x))$ such that $a \in L$ and that $\gr a $ induces a module isomorphism
\[
(\beta, \gr a)\colon\ (\geneg, \gr \phi) \to \big( \gr \f_x, \gr \varphi(x)\big).
\]
Note that $\beta$ is uniquely determined by $\gr a$
as $\beta (\xi) = (\gr a) \rho ( \xi) (\gr a) ^{-1}$ for $ \xi \in \geneg $, where $\rho$ denotes the representation of $\gr \f_x$ on $\gr \varphi(x)$. Put $Q(0) = \bigcup_{x \in M} Q(0)_x $, then we see that $Q(0)$ is a principal fibre bundle over $M$ with structure group $\bar{G}(0)$. Moreover the map $\Phi (0)\colon Q(0) \to L $ which sends $a \in Q(0)_x$ to $a$ is an immersion lifting $\varphi$.

\begin{prop} Let $\omega (0)$ $($or simply $\omega)$ be the pull-back $\Phi(0)^*\omega_L$. Then $\omega(0)$ is a 1-form on~$Q(0)$ taking values in $\lf$, and we have:
	\begin{enumerate}\itemsep=0pt
		\item[$(i)$] $R_a ^* \omega = \Ad \big(a^{-1} \big) \omega$, for $a \in \bar{G}(0)$.
		\item[$(ii)$] $\langle \tilde{A} , \omega \rangle = A$, for $A \in \bar{\g}(0) $.
		\item[$(iii)$] $ {\rm d}\omega + \frac12 [\omega, \omega] = 0$.
		\item[$(iv)$] 	Define the filtration of $Q(0)$ as follows: for $p\le 0$, $\f^p TQ$ being the inverse image of $\f^pTM$ and for $i>0$, $\f^i TQ$ the induced filtration from $\phi^i E(0)$. Then for each $z\in Q(0)$, the map $\om_z\colon T_zQ(0)\to\lf$ is filtration preserving, and $\gr\om_z\colon \gr(T_zQ(0),\f)\to\gr\lf$ maps $\gr(T_zQ(0),\f) $ isomorphically onto $E(0)$ for any $z\in Q$.
	\end{enumerate}
\end{prop}
\begin{proof} The assertions $(i)$, $(ii)$ and $(iii)$ are clear. Let us prove $(iv)$.

	Let $\pi\colon Q(0) \to M $ be the canonical projection, $z \in Q(0)$, and $ x = \pi (z)$.
	Let $g$ be a local cross-section of $\pi\colon Q(0) \to M $ around $x$ with $g(x) = z$.
	Denote by $\rho$ the representation of $\geneg$ on~$\gr V$
	and by $\theta$ that of $\gr\f_x$ on~$\gr\varphi(x)$.
	Then by the definition of $Q(0)$ there exists $\beta\colon \geneg \to \gr\f_x$
	such that it holds, for $A \in \g_p, v \in \gr_q\phi$,
	\[
	\rho (A) v = (\gr g(x))^{-1} \theta (\beta A) \gr g(x) v.
	\]
	If we take a local cross section $X$ of $\varphi^p$ representing $\xi = \beta (A)$, we have
	\begin{align*}
		(\gr g(x))^{-1} \theta (\beta A) \gr g(x) v
		&\equiv g(x) ^{-1} X _x (gv) \mod \phi^{p+q+1} \\
		& = \langle \omega ,g_* X_x\rangle v \\
		& \equiv \langle \gr \omega, \xi \rangle v \mod \phi^{p +q+1}.
	\end{align*}
	Hence $\gr \omega ( \xi ) = \rho \big( \beta ^{-1} \xi \big) $
	for $\xi \in \gr\f_x$, from which it follows
	that $\gr \omega $ maps $\gr T_zQ(0) $
	into~$E(0)$.
	But in view of the following commutative diagram
	and from our assumption that $\varphi$ is an~immersion,
	we see that
	$\gr \omega\colon \gr T_z Q(0) \to E(0) $ is an isomorphism:
	\[
	\begin{CD}
	T_xM @>g_*>>T_{g(x)}L @>\omega>L_{g(x)}^{-1}> T_eL @ = \lf
	\\
	@| @VV \pi_{L*} V @VV \pi_{L*} V @ VV\pi_{\lf} V
	\\
	T_xM @>\varphi_*>> T_{\varphi(x)}L/L^0 @>> \bar{L}_{g(x)}^{-1}> T_{\bar{e}}L /L^0 @ = \lf/ \lf^0.
	\end{CD}\tag*{\qed}
	\]
	\renewcommand{\qed}{}
\end{proof}

\subsection[Structure function chi]{Structure function $\boldsymbol{\chi}$}
According to the direct sum decomposition $\lf = E(0) \oplus E(0)'$, we write
\[
\omega (0) = \omega (0) _I + \omega (0) _{II}, \qquad \text{or simply}\quad
\omega = \omega _I + \omega _{II} ,
\]
where $\omega_I$, $\omega_{II}$ take values in $E(0)$, $E(0)'$ respectively. Then by $(iv)$ of the above proposition
\[
\omega _I\colon\ T_zQ(0) \to E(0)
\]
is an isomorphism and
\[
\omega _{II}\colon\ T_zQ(0) \to E(0)'
\]
is a map of order 1 and vanishes on the vertical vectors.
We define $ \chi = \chi (0) $ by
\[
\omega(0) _{II} = \chi (0)\, \omega(0) _{I},
\]
Then $\chi(0) $ is a function on $Q(0)$ taking values in $\phi^1 \Hom\left(E(0), E(0)'\right)$. Since $\chi(0) A = 0 $ for $A\in\bar\g(0)$, it
takes values in $\Hom\left( \geneg, E(0)'\right)$. Since we see
\[
\Hom\left( \geneg, E(0)'\right) \subset \Hom\left( \geneg, \bar {\g}'\right)
\cong \Hom\left( \geneg, \lf/ \gebar\right)\!,
\]
identifying $\gebar ' $ with $\lf/\gebar$, we may regard $\chi(0)$ as a function taking values in $\Hom( \geneg, \lf/ \gebar)$ and actually in
$\phi^1\Hom( \geneg, \lf/ \gebar)$.

\subsection{Condition (C)}
Starting from $Q(0)$, we carry a series of reductions by using the structure function $\chi$. We make the following assumptions (C1), (C2) which are almost indispensable to construct an extrinsic Cartan connection.

First, note that the adjoint action of $\bar{G}^0$ on $\lf$ leaves invariant $\gebar$ and $\phi^0\gebar$ and acts naturally on $\lf / \gebar$ and $\gebar / \phi^0\gebar$. We shall identify $\geneg$ with $\gebar / \phi^0\gebar$ and define the action of $\bar{G}^0$ on $\geneg$ via this identification. Therefore $\bar{G}^0$ acts on $\Hom\left( \geneg , \lf/ \gebar\right)$ and leaves invariant
$\phi^1\Hom\left(\geneg , \lf/ \gebar\right)$.

\begin{enumerate}\itemsep=0pt
	\item[(C1)] There exists a $\bar G^0$-invariant graded subspace $\gebar'\subset \lf$ such that $\lf = \gebar \oplus \gebar '$.
	\item[(C2)] There exists a $\bar{G}^0$-invariant graded subspace
	\[
	W = \bigoplus W_p \subset \phi^1\Hom\left(\geneg, \lf/\gebar\right)
	\]
	such that $\phi^1\Hom\left(\geneg, \lf/\gebar\right) = \partial \phi^1(\lf/\gebar ) \oplus W$.
\end{enumerate}

{\sloppy
Note that the representation $\rho$ of $\bar{G}^0$
on $\Hom\left(\geneg, \lf/\gebar \right)$ induces
the representation $\rho^{(k)}$ of~$\bar G^{(k)}=\bar{G}^0/\phi^{k+1}\bar{G}^0$
on $\Hom\left( \geneg, \lf/\gebar\right) /\phi^{k+1}\Hom\left( \geneg, \lf/\gebar\right)$. If $W$ is $\bar{G}^0$-invariant, then $W^{(k)}\left(= W/\phi^{k+1}W\right)$ is $\bar G^{(k)}$-invariant.

}

\subsection{Reductions}\label{ss43}
Now let us proceed to the construction of a series of $L/L^0$ extrinsic bundles $\{Q(k), \, k \ge 0\}$.
Conducted by the following diagram :
\[
\begin{CD}
@. L @= L @. \\
@. @AA\Phi(k)A @AA\Phi(k+1)A @. \\
\Hom(\geneg, \lf/\gebar) @<\chi(k)<< Q(k) @ <\iota_{k+1, k}<< Q(k+1) @>\chi(k+1)>>\!\!\!\! \Hom(\geneg, \lf/\gebar) \\
@VVV @ VVV @VVV @VVV \\
\Hom(\geneg, \lf/\gebar) ^{(k+1)} @<\chi(k) ^{(k+1)}<< Q(k)^{(k+1)}@ <\bar{\iota}_{k+1, k}<<
Q(k+1) ^{(k+1)} @>\chi(k+1)^{(k+1)}>> \!\!\!\! W ^{(k+1)}\\
@VVV @ V\exp \lf_{k+1}VV @V\exp \gebar _{k+1}VV @VVV \\
W ^{(k)} @<\chi(k) ^{(k)}<< Q(k)^{(k)} @ <\cong<< Q(k+1) ^{(k)} @>\chi(k+1)^{(k)}>>\!\!\!\! W^{(k)} \\
@. @ VVV @VV \bar{G}^0/\phi^{k+1}V @. \\
@. M @= M @.
\end{CD}
\]
we are going to construct bundles with 1-forms
$\big(Q(k), \bar{G} (k), M ; \omega(k)\big)$ inductively for $ k= 0, 1, 2, \dots$
in such a way that the following conditions are satisfied.

\begin{enumerate}\itemsep=0pt
\item[$(k)_1$] Canonically associated to an immersion
$\varphi\colon (M, \f ) \to L/L^0 $ of type $(\geneg,\gr V, L)$,
there is a principal fibre bundle $Q(k)$ over $(M, \f) $ with structure group $\bar{G} (k)$,
and an immersion $\Phi(k)\colon Q(k) \to L$ which gives bundle
homomorphism:
\[
\begin{CD}
Q(k) @>\Phi(k)>> L \\
@VV\bar{G}(k) V @VVL^0V \\
M @>\varphi>> L/ L^0.
\end{CD}
\]

\item[$(k)_2$]
There is an $\lf$-valued 1-form $\omega(k) $ defined by
$\omega (k) = \Phi(k)^* \Omega$ satisfying
\begin{gather*}
R_a ^* \omega(k) = {\rm Ad}\left(a^{-1}\right)\omega (k) \qquad \forall a \in \bar{G}(k),
\\
\big\langle \omega (k) , \tilde{A} \big\rangle = A \qquad \forall A \in \gebar(k),
\\
{\rm d}\omega(k) + \frac12 [ \omega(k) , \omega(k) ] = 0.
\end{gather*}

Write
\[
 \omega (k) = \omega (k)_I + \omega(k)_{II}
\]
with $\omega (k) _I $ and $\omega(k)_{II}$ being $E(k) $ and $E(k)'$-valued 1-forms
respectively. Then the li\-near map
$\left(\omega(k)_I \right)_z\colon T_z Q(k) \to E(k) $ is a filtration preserving isomorphism
and $\omega(k) _{II}$: $T_z Q(k) \to E(k)' $ is a map of order 1
for each $z \in Q(k)$.

\item[$(k)_3$]
Define a map $\chi(k)\colon Q(k) \to \phi^1\Hom(\geneg, \lf/\gebar)$ by
\[
\omega(k)_{II} = \chi (k) \omega(k)_I, \qquad \text{or symbolically}\qquad
\chi(k) = \frac {\omega(k)_{II}}{\omega(k)_I}.
\]
By definition $\chi(k)$ is a $\Hom(E(k), E(k)')$-valued function
on $Q(k)$, but by $(k)_2$ it actually takes values in
$\phi^1\Hom (\geneg, E(k)')$. Noting that $E(k)' \subset \gebar '$ and identifying $\gebar '$ with $\lf / \gebar $, we~regard $\chi(k)$ as a map from $Q(k)$ to $\Hom( \geneg, \lf/\gebar)$.

\item[$(k)_4$]
If $k \ge 1$, then
\[
Q(k) = \big\{ z \in Q(k-1)\colon \chi (k-1) ^{(k)} (z) \in W^{(k)} \big\}.
\]

\item[$(k)_5 $]
\qquad\ $\chi(k) ^{(k)} = \iota _{k, k-1} ^* \chi(k-1) ^{(k)}, $ \\
where $\iota_{k, k-1}$ denotes the canonical injection
$ Q(k) \to Q(k-1) $.
Therefore
$\chi (k) ^{(k)}$ takes values in $W^{(k)} = W / \phi^{k+1} W$
and
\[
R_g^* \chi(k) ^{(k)} = \rho \left(g^{-1}\right)\chi(k) ^{(k)}
\qquad \text{for}\quad g \in \bar{G}^0.
\]
Note here that the action
of $\bar{G}^0$ on $\phi^1\Hom(\geneg, \lf/\gebar) / \phi^{k+1}$ only depends on $\bar{G}^0/\phi^{k}\bar{G}^0$.

\item[$(k)_6$]
The map $\chi(k)^{(k+1)}\colon Q(k)^{(k+1)} \to \Hom( \geneg, \lf / \gebar ) ^{(k+1)}$ satisfies
\[
R_{g\exp A_{k+1}}^* \chi(k)^{(k+1)}
= \rho \left(g^{-1}\right) \chi(k)^{(k+1)} + \partial\circ\pi (A _{k+1})
\]
for $g\in \bar{G}^0$, $A_{k+1} \in \phi^{k+1} \g(k) $,
where $\partial \circ \pi$ is the composition of the following maps:
\[
\begin{CD}
\lf @ >\pi>> \lf / \gebar @>\partial >> \Hom(\geneg, \lf / \gebar ).
\end{CD}
\]

\item[$(k)_7$]
Let $\chi(k) _j $ denote the $\Hom(\geneg, \lf / \gebar)_j $-component
of $\chi(k)$ and $\chi(k) ^{(\ell)} = \sum _{j \le \ell} \chi(k)_j $. Then
\[
\partial \chi(k)_{k+1} = B_{ k+1} \big(\chi ^{(k)}, D \chi(k)^{(k)}\big),
\]
where $B_{k+1} $ is a $\Hom\left( \wedge ^2 \geneg , \lf/\gebar \right) _{k+1} $-valued function determined by $\chi(k)^{(k)}$ and its derivatives $D \chi(k)^{(k)}$. In~particular, if $\chi(k)^{(k)} \equiv 0 $, then $B^{(k)}\equiv 0$.
\end{enumerate}

Now let us prove the assertions above by induction on $k$. Supposing that for a non-negative integer $k$ the assertions $(j)_1, \ldots, (j)_7$ $(j< k)$
hold, we prove the assertions for $k$.

First for $(k)_1$, if $k= 0$, $Q(0)$ is already defined. If $k \ge 1$, we define $Q(k)$ by $(k)_4$. Then we see from $(k-1)_6$ that
for every $x \in M$ there exists $z$ in the fibre $Q(k-1){_x}$
over $x$ such that $\chi(k-1)^{(k)} (z) \in W^{(k)}$ and that for $a \in \bar{G} (k-1) $ we have $\chi(k-1)^{(k)} (za) \in W^{(k)}$, if and only if $a \in \bar{G}(k)$. Hence $Q(k) $ is a principal fibre bundle over $M$
with structure group $\bar{G}(k)$. Setting $\Phi(k)$ to be the restriction of $\Phi(k-1)$ to $Q(k)$, we finish the construction $(k)_1$.

To verify the assertion $(k)_2$ we have only to note that
\[
\omega (k) _I = \iota_{k, k-1} ^* \big( \omega (k-1) _I - \omega (k-1) _{ \gebar_k '} \big)
\]
and
\[
\iota_{k, k-1}^* \omega (k-1)_{ \gebar_k '} \equiv 0 \mod \omega (k)_I,
\]
from which it follows that $\omega(k)_ I $ gives an isomorphism
$T_zQ(k) \to E(k) $ at each $z \in Q(k)$.

Now $\chi(k)$ being defined by $(k)_3$, let us prove $(k)_5$.
Let $\iota\colon Q(k) \to Q(k-1)$ be the inclusion. We have
\begin{gather*}
	\iota^* \omega(k-1)_I = \omega(k) _I + \iota ^* \omega(k-1) _{\gebar_k '}, \\
	\iota^* \omega(k-1)_{II} = \omega(k) _{II} - \iota ^* \omega(k-1) _{\gebar_k '}.
\end{gather*}
Let $v \in \g_p \ ( p < 0 ) $, and take $X \in T_zQ(k)$ such that
$ \langle \omega_I , X \rangle = v $.
Then $\chi(k) (z) v = \langle \omega_{II} , X\rangle$.
But we have
\[
\big\langle \iota^* \omega(k-1)_I, X \big\rangle = v + \big\langle \iota ^* \omega(k-1) _{\gebar_k '} , X\big\rangle.
\]
If we put
$B_k = \langle \iota ^* \omega(k-1) _{\gebar_k '} , X\rangle \in \lf_k$
and $ Y= X - \tilde{B_k} \in T_zQ(k-1)$, then we have
$\langle \omega(k-1)_I, Y\rangle$ $= v $ and
\begin{align*}
\chi(k-1)(z) v &=\big\langle \omega(k-1) _{II} , Y \big\rangle
= \big\langle \omega(k-1) _{II} , X - \tilde{B_k} \big\rangle
= \big\langle \omega(k-1) _{II} , X \big\rangle
\\
& = \big\langle \omega(k) _{II} - \iota ^* \omega(k-1) _{\gebar_k '} , X\big\rangle
= \chi(k)(z)v - B_k \\
&\cong \chi(k)(z)v \mod \phi^{k+1+p },
\end{align*}
which implies: $ \chi(k) ^{(k)} = \chi(k-1)^{(k)}$.

The last formula in $(k)_5$ is a consequence of $(k-1)_6$.

Now let us prove the assertion $(k)_6$.
Let $ a = g\exp A $, where $g \in \bar{G} $ and
$A \in \phi^{k+1} \lf$ of which~$\lf_{k+1}$ component
is denoted $A_{k+1}$.

By definition
$\chi(k) = \frac {\omega(k)_{II}} { \omega(k)_I}$
and we have
\[
R_a^*\chi(k)
= \frac {R_a^* (\omega(k)_{II})}{R_a^*( \omega(k)_I) }
=\frac { (R_a^*\omega(k))_{II}} { (R_a^*\omega(k))_I}
=\frac { \big({\rm Ad}\big(a^{-1}\big)\omega(k)\big)_{II}} { \big({\rm Ad}\big(a^{-1}\big)\omega(k)\big)_I}.
\]
Now we have
\begin{align*}
\left({\rm Ad}\big(a^{-1}\omega\big)\right)_I
	&=\left(\Ad (\exp A)^{-1} {\rm Ad}\left(g^{-1}\right)\omega\right)_I \\
	&\equiv \left( {\rm Ad}\left(g^{-1}\right)\omega - \left[ A, {\rm Ad}\left(g^{-1}\right)\omega \right] \right)_I \mod \Order(2k+2)
\\
	&= \left( {\rm Ad}\left(g^{-1}\right) (\omega_I + \omega_{II} ) - \left[ A, {\rm Ad}\left(g^{-1}\right) (\omega_I + \omega_{II}) \right] \right)_I \\
	&\equiv {\rm Ad}\left(g^{-1}\right) \omega_I - \left( \left[ A, {\rm Ad}\left(g^{-1}\right) (\omega_I ) \right] \right)_I \mod \Order(k+2)
\\
	&\equiv {\rm Ad}\left(g^{-1}\right) \omega_I \mod \Order(k+1).
\end{align*}
We have also
\begin{align*}
	\left({\rm Ad}\big(a^{-1}\omega\big)\right)_{II}
	&=\left(\Ad (\exp A)^{-1} {\rm Ad}\left(g^{-1}\right)\omega\right)_{II} \\
	&\equiv \left( {\rm Ad}\left(g^{-1}\right)\omega - \left[ A, {\rm Ad}\left(g^{-1}\right)\omega \right] \right)_{II}
	\mod \Order(2k+2) \\
	&= \left( {\rm Ad}\left(g^{-1}\right) (\omega_I + \omega_{II} ) - \left[ A, {\rm Ad}\left(g^{-1}\right) (\omega_I + \omega_{II} ) \right] \right)_{I I} \\
	&\equiv {\rm Ad}\left(g^{-1}\right) \omega_{II} - \left[ A, {\rm Ad}\left(g^{-1}\right) (\omega_I ) \right]_{II} \mod \Order(k+2).
\end{align*}
Now let $ p < 0$ and $ v \in \g_p $. Take $ X, Y \in T_zQ(k)$ satisfying
\[
\left\langle {\rm Ad}\left(g^{-1}\right) (\omega_I) , X \right\rangle
= \left\langle \left({\rm Ad}\left(a^{-1}\right)\omega\right)_I , Y \right\rangle = v.
\]
Put $Z = Y-X$. Then from the above formula it follows that
the map $ X \mapsto Z$ is of order $\ge k+1$.
Then we have:
\begin{align*}
	\left(R_a^*\chi \right)v &= \left\langle \left({\rm Ad}\left(a^{-1}\right)\omega\right)_{II} , Y \right\rangle
= \left\langle \left({\rm Ad}\left(a^{-1}\right)\omega\right)_{II} , X+Z \right\rangle \\
	&\equiv \left\langle\left( {\rm Ad}\left(g^{-1}\right)\omega\right)_{II} - \left[ A, {\rm Ad}\left(g^{-1}\right) \omega_I \right]_{II} , X+ Z \right\rangle
\\
	&\equiv {\rm Ad}\left(g^{-1}\right) \chi ({\rm Ad}(g)v) - [ A, v] _{II} \mod \Order (k+2).
\end{align*}
Hence we have
\[
R_a^* \chi ^{(k+1)} v = \rho \big(g^{-1} \big) \chi ^{(k+1)}v - [A_{k+1} , v ] _{II},
\]
which proves $(k)_6$.

Finally let us verify $(k)_7$. From the structure equation we have
\begin{gather*}
{\rm d}\omega _I + \frac12 [\omega, \omega]_I = 0, \qquad
{\rm d}\omega _{II} + \frac12 [\omega, \omega]_{II} = 0, \qquad
\omega _{II} = \chi \omega _I.
\end{gather*}
Hence we have
\begin{gather}
{\rm d}\omega _{I} + \frac12 [\omega_{I}, \omega_{I}]_{I} + [\omega_{I}, \omega_{II}]_{I}
+\frac12 [\omega_{II}, \omega_{II}]_{I} = 0,
\\[1ex]
{\rm d}\omega _{II} + [\omega_{I}, \omega_{II}]_{II}
+\frac12 [\omega_{II}, \omega_{II}]_{II} =0.
\end{gather}
Let $v \in \g_p$, $w \in \g_q$ $(p, q < 0 )$ and $X$, $Y$ be the vector fields
on $Q(k)$ determined by $ \langle \omega_I, X\rangle = v$, $\langle \omega _I, Y \rangle = w$.
Evaluating the equations on $X$, $Y$, we have
\begin{gather}
-\langle \omega _I, [X, Y]\rangle + [v, w] + \mathcal A \big([v, \chi(w)]_I \big) + [\chi(v), \chi(w)]_I = 0, \label{eq7}
\\
X\chi(w) - Y \chi(v) - \langle \omega_{II} , [X,Y] \rangle
+[v, w]_{II}+ \mathcal A \big([v, \chi(w) ]_{II}\big) + [\chi (v) , \chi (w ) ] _{II} = 0,\label{eq8}
\end{gather}
where $\mathcal A$ denotes the alternating sum in $v$, $w$.
Substituting~\eqref{eq7} into~\eqref{eq8}, we get
\begin{gather}
X\chi(w) - Y \chi(v) -\chi \big( [v, w] + \mathcal A \big([v, \chi(w) ]_I \big) + [\chi(v), \chi(w)]_I\big)+ [v, w]_{II} \nonumber
\\ \hphantom{X\chi(w) }
{}+ \mathcal A \big([v, \chi(w) ]_{II}\big) +
[\chi (v), \chi (w )]_{II}= 0.\label{eq12}
\end{gather}
Let us compute~\eqref{eq12} under mod $\phi^{p+q+k+2}$.
Observing that
\[
X \chi _{k+1} (w) \in \phi^{k+1+q} \subset \phi^{p+q+k+2},
\]
we have
\begin{gather}
-\chi_{k+1} ( [v, w]) + \mathcal A \big([ v, \chi _{k+1} (w)] _{II} \big)
\nonumber
\\ \qquad
{}= -X\chi^{(k)} (w) + Y \chi ^{(k)}(v)+\chi ^{(k)} \big( [v, w] + \mathcal A \big(\big[v, \chi ^{(k)} (w) \big]_{I}\big) + \big[\chi ^{(k)} (v), \chi ^{(k)} (w)\big]_I\big) \nonumber
\\ \qquad\phantom{=}
{}- \mathcal A \big( \big[v, \chi^ {(k)}(w)\big]_{II} \big)- \big[\chi^{(k)}(v), \chi^{(k)}(w)\big]_{II}.\label{eq13}
\end{gather}
Thus we have
\begin{equation}\label{eq:chi}
\partial \chi _{k+1} = B_{k+1} \big( \chi ^{(k)}, D\chi ^{(k)} \big),
\end{equation}
where $B_{k+1} $ is a $\Hom\left(\Lambda ^2 \geneg , \lf / \gebar \right)$-valued
function defined by the right hand side of~\eqref{eq13}.

Thus we have proved the assertions $(k)_1, \ldots , (k)_7 $,
completing the inductive construction of~$\big(Q(k), \omega(k), \chi(k) \big)$.

Let $(P,\omega ,\chi)$ to be $\big(Q(k),\omega(k),\chi(k)\big)$ for $k$ bigger than the maximal $i$ such that $\lf_i\ne 0$.

Let us now formulate the main result of this section.
\begin{df}
	We say that an $L/L^0$ extrinsic bundle $\pi\colon P\to (M,\f)$ is
	\emph{an extrinsic $W$-normal Cartan connection}, if its structure group is $\bar G^0\subset L^0$ and the structure function $\chi$ defined by~$\om_{II}=\chi \om_I$ takes values in the subspace $W$.
\end{df}
Here we write $ \omega = \omega_I + \omega_{II} $ according to the direct sum decomposition $ \lf = \gebar + \gebar'$. Note that~$\omega_I$ is an absolute parallelism on $P$, which defines a Cartan connection on $M$ modelled by~$\bar G / \bar G^0$ in~the usual sense. We also note that the notion of extrinsic $W$-normalily does not depend on~the choice of the space $\gebar'$.

\begin{thm}\label{thm:reduction}
	Let $\{ V, \phi, L, \geneg \} $ be as above and satisfy
	${\rm (C0)}$, ${\rm (C1)}$ and ${\rm (C2)}$. Then for each immersion $\varphi\colon (M, \f) \to
	L/L^0 \subset \Flag (V, \phi)$ of type $(\geneg, V, L)$ we can construct canonically an~ext\-ri\-nsic $W$-normal Cartan connection $(P,\om)$. 	In particular, two immersions $(M, \varphi)$ and $(M', \varphi ')$ 	are $L$-equivalent if and only if the corresponding extrinsic Cartan connections $(P,\omega)$ and $(P',\omega')$ are isomorphic.
\end{thm}

\begin{cor}
	The structure function $\chi$ along with its derivatives fulfills completely the invariants of $(P,\omega)$. In~particular, the immersion $(M,\f)\to L/L^0\subset \Flag (V, \phi)$ of type $(\geneg, V, L)$ is locally equivalent ${(}$under the action of $L)$ to the model embedding if and only if $\chi=0$.
\end{cor}
\begin{rem}
	Of course the structure function
	\[
		\chi\colon\ P\to \Hom\left(\bar\g, \lf/\bar\g\right)
	\]
	depends on the choice of a complementary subspace $\bar\g'$ to $\bar\g$ in $\lf$. Let $\bar\g''$ be another $\bar G^0$-invariant complementary subspace and let $\chi'$ be the structure function determined by this choice. Then we see easily that
	\[
		\chi' = (1-\chi\lambda)^{-1}\chi = \bigg(1 + \chi\lambda + \frac{1}{2!}(\chi\lambda)^2+\cdots\bigg)\chi,
	\]
	where $\lambda\in \Hom(\lf/\bar\g,\bar\g)$ is given by $\lambda = \sigma' -\sigma$, $\sigma, \sigma'$ being the isomorphisms from $\lf/\bar\g$ to $\bar\g'$ and~$\bar\g''$ respectively. In~particular, if $\chi$ vanishes identically up to order $k$, then the same holds also for~$\chi'$.
	
	Similarly, the structure function $\chi$ also depends on the choice of the complementary sub\-space~$W$:
	\[
		\phi^1\Hom\left(\bar\g, \lf/\bar\g\right)=\phi^1\partial(\lf/\bar\g)\oplus W.
	\]
	Let $W'$ be another $\bar G^0$-invariant complementary subspace and let $\chi'$ be the new structure function. Again, it can be easily shown that $\chi$ vanishes up to order $k$ if and only if the same holds for $\chi'$.
\end{rem}

See Section~\ref{sec:general} for a generalization of the theorem and the corollary above.

\section{Rigidity of rational homogeneous varieties}
\label{sec:var}

\subsection{Extrinsic parabolic geometries}\label{parabolic}
Let $\g=\bigoplus_{i=-\mu}^\mu\g_i$ be a semi-simple graded Lie algebra. Denote by $\g_{-}$ the nilpotent subalgebra $\bigoplus_{i<0}\g_i$ and by $\pf$ the parabolic subalgebra $\bigoplus_{i\ge0}\g_i$ in $\g$. Let $\V$ be an arbitrary faithful finite-dimensional irreducible representation of $\g$.

\smallskip
Let $B$ be the Killing form of $\g$. The following properties are well-known~\cite{capslovak}:
\begin{enumerate}\itemsep=0pt
	\item There exists a grading element $E\in \g$, such that $\g_k=\{X\in\g\mid
	[E,X]=kX\}$.
	\item $B(\g_i,\g_j)=0$ for all $i+j\ne0$, and the restriction of $B$ to
	$\g_i\times\g_{-i}$ is non-degenerate.
	\item In the real case, there exists an involution $\theta$ of $\g$
	such that $\theta(\g_i)=\g_{-i}$ and the bilinear form $(X,Y)=-B(X,\theta(Y))$
	is positive definite. Moreover, for any $\g$-module $\V$ there exists a scalar
	product $(\,,\,)$ on $\V$ such that
	\[
	(Xu,v) + (u,\theta(X)v) = 0\qquad\text{for all}\quad u,v\in V,\quad X \in\g.
	\]
	\item In the complex case, there exists an anti-involution $\sigma$ of $\g$ such that
	$\sigma(\g_i)=\g_{-i}$ and the Hermitian form $(x,y)=-B(x,\sigma(y))$ is
	positive definite. Moreover, for any $\g$-module $\V$ there exists a positive
	definite Hermitian form $(\,,\,)$ on $\V$ such that
	\[
	(Xu,v) + (u,\sigma(X)v) = 0\qquad\text{for all}\quad u,v\in V,\quad X \in\g.
	\]
\end{enumerate}

Indeed, consider first the complex case. Then the existence of an anti-involution $\sigma$ such that $\sigma(\g_i)=\g_{-i}$ and the Hermitian form $(x,y)=-B(x,\sigma(y))$ is positive definite, is shown in~\cite[Lemma~1.5]{tan79}. In~particular, the corresponding real form $\g^{\sigma}$ is compact. Hence, it stabilizes a~certain positive-definite Hermitian form $(\,,\,)$ on~$\V$. Since any element $X\in \g$ can be represented as~$X=X_1+{\rm i}X_2$ and $\sigma(X)=X_1-{\rm i}X_2$, where $X_1,X_2\in \g^\sigma$, we immediately get
\begin{gather*}
(Xu,v) + (u,\sigma(X)v) = (X_1u,v) + (u,X_1v) + ({\rm i}X_2u,v)+(u,-{\rm i}X_2v)
\\ \hphantom{(Xu,v) + (u,\sigma(X)v)}
{}= (X_1u,v) + (u,X_1v) + {\rm i}(X_2u,v)+{\rm i}(u,X_2v) = 0.
\end{gather*}

Similarly, in the real case, the existence of an involution $\theta$ such that $\theta(\g_i)=\g_{-i}$ and the bilinear form $(x,y)=-B(x,\theta(y))$ is positive definite, is also demonstrated in~\cite{tan79}. Let $\g=\g^\theta(1)\oplus\g^\theta(-1)$ be the eigenvalue decomposition of $\theta$. Then it is well known that the Lie algebra $\g^\theta(1)+{\rm i}\g^\theta(-1)$ is compact in $\gl(\V^{\C})$. So, it preserves a positive definite Hermitian form on~$\V^{\C}$. Restricting the real part of this form on $\V$ gives us the requires positive definite scalar pro\-duct~on~$\V$.

For any $A\in\gl(\V)$, denote by $A^*$ the operator adjoint to $A$ with respect
to scalar (Hermitian) product $(\,,\,)$ on $\V$. Then the above conditions imply
that $\theta(X)=-X^*$ (resp., $\sigma(X)=-X^*$) for any $X\in\g$. In~particular, $\g$ is stable with respect to the (anti-)involution $A\mapsto -A^*$ of $\gl(\V)$.

Decomposing $\V$ according to the eigenvalues of the grading element $E$, we
can equip $\V$ with the grading $\V=\bigoplus \V_{j}$ such that $\g_i.\V_j\subset \V_{i+j}$ for all~$i$,~$j$. Note that the degrees in the sum $\V=\bigoplus \V_{j}$ may be rational, but the difference between any two degrees $j$, $j'$ for which both $\V_{j}$ and $\V_{j'}$ are non-trivial is necessarily an integer. Let $j_{\max }$ be the largest such degree. Shifting the grading of $\V$ by $-j_{\max }-1$, we can always normalize the grading of $\V$ such that it is concentrated in the negative degree with $\V_{-1}\ne 0$. Note that independently of such shift of degrees the induced grading of $\gl(\V)$ is always defined via eigenvalues of the grading element~$E$.

Similarly, if $\V$ is an arbitrary finite-dimensional representation of the Lie algebra $\g$, we can decompose it into the sum of irreducible representations, introduce the above grading on each of the summands, and combine them into the grading of~$\V$. So, from now on we assume that~$\V$ is an arbitrary (not necessarily irreducible) finite-dimensional representation of~$\g$, equipped with the above grading. Let $\phi$ be the induced filtration of $\V$, which is by definition stabilized by $\pf$.

Let $G$ be a connected semisimple Lie group with a Lie algebra $\g$, and let $P$ be the parabolic subgroup corresponding to the subalgebra $\pf$. The homogeneous flag variety $G/P$ is a rational homogeneous variety.
Assume that the representation $\g\to\gl(\V)$ can be extended to the representation $G\to {\rm GL}(\V)$ of the Lie group $G$. Let
\[
\varphi_{{\rm model}}\colon\ G/P \to \Flag(\V,\phi)
\]
be the standard embedding of type $(\g_{-},\V,G)$.

\begin{lem}\label{symalg}
	Suppose that $\V$ is isotypic, that is all highest weights of $\V$ coincide. Then the relative prolongation $\Prol(\g_{-},\V)$ is equal to $\g+Z(\g)\subset \gl(\V)$, where $Z(\g)$ is the centralizer of~$\g$ in $\gl(\V)$.
\end{lem}
\begin{proof}
	It is sufficient to prove lemma in the complex case, as the real case immediately follows via the complexification argument. Consider $\g$-module $\gl(\V)$. Assume first that $\V$ is irreducible. As $\Prol(\g_{-},\V)$ contains $\g$, it is a submodule of the $\g$-module $\gl(\V)$. Let $U$ be any graded submodule of $\Prol(\g_{-},\V)$ such that $\g\cap U = 0$. By definition $U$ is concentrated in the non-negative degree. Hence, it generates a $\g$-invariant subalgebra $\n$ that is also concentrated in a~non-negative degree. In~particular, $\g+\n$ is also a subalgebra. Since $\g$-module $\V$ is irreducible, $(\g+\n)$ also acts irreducibly on $\V$. Hence (see \cite[{Chapitre~1, Section~6, Th\'eor\`eme~4}]{bou}), $\g+\n$ is a reductive subalgebra in $\gl(\V)$. From the general theory of graded semisimple Lie algebras it follows that the dimensions of subspaces of opposite degrees should coincide. This is only possible if $\n$, and thus $U$ is concentrated in~degree~$0$. As $[\g_{-},U]\subset U$, this immediately implies that $[\g_{-}, U] =0$, and $U\subset Z(\g)$.
	
	Let $\V$ be now an arbitrary isotypic $\g$-module, that is $\V$ is isomorphic to $\V'\otimes \C^k$, where~$\V'$ is irreducible and $\C^k$ is a trivial $\g$-module. We note that $\g$-module $\gl(\V)$ is isomorphic to~$\gl(\V')\otimes\gl(k,\C)$ and any its irreducible submodule has the form $U\otimes L$, where $U$ is an irreducible submodule of $\gl(\V')$ and $L$ is a $1$-dimensional subspace in $\gl(k,\C)$. If such submodule lies in~$\Prol(\g_{-},\V)$ and has zero intersection with $\g$, then it is concentrated in the non-negative degree. It~is easy to see that $U$ itself lies in $\Prol(\g_{-},\V')$. By above this means that $U\subset Z(\g)$. This completes the proof.
\end{proof}

\begin{rem}
	The condition of $\V$ to be isotypic can not be dropped. The simplest counter-example is given by $\g=\sll(2,\C)$ and $\V=V_1\oplus V_2$,	where $V_1$ and $V_2$ are two irreducible representations of $\sll(2,\C)$ of different dimension.
\end{rem}
In the following we always assume that the $\g$-module $\V$ is isotypic.
We also say that $\varphi_{{\rm model}}$ is \emph{a standard embedding} of $G/P$ corresponding to the $\g$-module~$\V$.

\subsection{Harmonic theory on semisimple Lie algebras}\label{s:harmonic}

Let us introduce an invariant symmetric bilinear form on $\gl(\V)$ as follows
\[
\Tr(A,B) = \tr AB,\qquad\text{for all}\quad A,B\in \gl(\V).
\]
It is clear that the restriction of $\Tr$ to $\gl_i(\V)\times\gl_j(\V)$ is
identically zero for $i+j\ne0$ and is non-degenerate for $i+j=0$.

Note that the restriction of $\Tr$ to any semisimple subalgebra of
$\gl(\V)$ is non-degenerate. In~particular, the restriction of $\Tr$ to $\g$ is non-degenerate, and we have $\gl(\V)=\g+\g^\perp$, where $\g^{\perp}$ is an orthogonal complement of $\g$ in $\gl(\V)$ with respect to $\Tr$. Similarly, let $\bar\g$ be the prolongation of $(\g_{-},\V)$. Then according to Lemma~\ref{symalg}, we have $\bar\g=\g+Z(\g)$.

\begin{lem}
	We have the following orthogonal decomposition with respect to $\Tr$:
	\[
		\gl(\V) = \g \oplus Z(\g) \oplus \bar\g^\perp.
	\]
\end{lem}
\begin{proof}
	It is sufficient to consider only the complex case. As above, decompose $\V$ as $\V'\otimes \C^k$, where the module $\V'$ is irreducible. Then from Schur's lemma, we know that
$Z(\g)=\gl(k,\C)$. Thus, it is easy to see that the restriction of $\Tr$ on $Z(\g)$ is non-degenerate. Moreover, for any $X\in \g$ and $A\in Z(\g)$, we have $\tr(XA) = \tr(X)\tr(A) = 0$.
\end{proof}

Note that all three components of this decomposition are $\g$- and $Z(\g)$-invariant. Moreover, since the Lie algebra $\g$ is graded and $Z(\g)\subset \gl_0(\V)$, we see that this decomposition (namely, the subspace $\bar\g^\perp$) is also compatible with the grading of $\gl(\V)$.

Let $\big(C=\sum_qC^q,\p\big)$ be the cochain complex associated with the $\g_{-}$-module $\gl(\V)$. It inherits the grading from the gradings of
$\g_{-}$ and $\gl(\V)$:
\[
C^q = \sum_p C_p^q,\qquad
C_p^q = \big\{c\in C^q \colon c(\g_{-i_1}\wedge\dots\wedge\g_{-i_q})\subset \gl_{p-i_1-\dots-i_q}(\V)\big\}.
\]
It is clear that the coboundary operator $\partial$ preserves grading and, hence, maps $C_p^q$ to $C_p^{q+1}$.

Similarly, for any $\g_{-}$-invariant graded subspace $W$ in $\gl(\V)$, denote
by $C(W)=\sum C^q_p(W)$ the cochain complex associated with the
$\g_{-}$-module $W$. We shall be particularly interested in~$C\big(\bar\g^\perp\big)$.

Any cochain $c\in C^q$ can be uniquely decomposed into the sum $c=c_I + c_{II}$ according to the decomposition $\gl(\V)=\bar\g\oplus
\bar\g^\perp$.

The cohomology $H^q_p\big(\geneg, \bar\g^\perp\big)$ can be effectively computed using Kostant theorem~\cite{kost} and the notion of normality formalized by
Tanaka~\cite{tan79}. Namely, we can extend the scalar products on $\g$ and $\V$ introduced in Section~\ref{parabolic} to the spaces $C^q$ and introduce the adjoint coboundary operator $\p^*\colon C^{q+1}\to C^{q}$ and the Laplacian $\Delta=\p\p^*+\p^*\p\colon C^q\to C^q$. The standard argument shows that $\ker \Delta = \ker \p \cap \ker \p^*$ and we have the direct product decomposition $C^q = \im \p \oplus \im \p^* \oplus \ker \Delta$. We shall call an element $c\in C^q$ \emph{harmonic}, if $\Delta c=0$, which is equivalent to $\p c=\p^*c=0$. Additionally, for any $\chi\in C^q$ we denote by $H\chi$ the projection of $\chi$ to $\ker \Delta$ in the above direct product decomposition. We shall also call $H\chi$ the \emph{harmonic part of $\chi$}.

\subsection{Normal reductions}

Let now $(M,\f)$ be an arbitrary filtered manifold of constant type $\m=\g_{-}$. Let
$\varphi\colon (M,\f) \to \Flag(\V,\phi)$ be arbitrary osculating map. Define $\bar\g'=\bar\g^\perp$ in the condition (C1) and $W=\phi^1\ker\partial^*$ in the condition (C2) of Section~\ref{ss43}.

In particular, an extrinsic Cartan connection associated with the embedding $\varphi$ is $W$-normal, if and only if we have $\partial^*\chi = 0$. We shall call such extrinsic Cartan connections just normal.

Applying Theorem~\ref{thm:reduction}, we immediately get
\begin{thm} For any embedding $\varphi\colon (M,\f) \to \Flag(\V,\phi)$ of type $(\g_{-},\V)$, there is a unique normal extrinsic Cartan connection $\om\colon TP\to \gl(\V)$ on the principal $\overline G^0$-bundle $P\to M$.
\end{thm}

\subsection{Vanishing of cohomology groups and rigidity}

According to corollary of Theorem~\ref{thm:reduction}, the structure function $\chi=\sum_{k\geq 1}\chi_k\colon P\to
\text{Hom}\left(\g_-,\bar\g^{\perp}\right)_k$ is a \emph{complete invariant} of the embedding $\varphi$.

Using~Section~\ref{s:harmonic}, we identify $H^1\left(\g_{-}, \bar\g^{\perp}\right)$ with $\ker \Delta = \ker \p\cap \ker\p^*$ and consider $H\chi$ as an element of $H^1\left(\g_{-}, \bar\g^{\perp}\right)$. It is the \emph{fundamental invariant} of the embedding $\varphi$ in the following sense.
\begin{prop}
We have $\chi=0$ if and only if $H\chi = 0$. And in this case the embedding $\varphi$ is locally equivalent to the standard embedding of type $(\g_-,\V){:}$
\[
\varphi_{{\rm model}}\colon\ G/G^0 \to L/L^0 \subset
\Flag(\V,\phi).
\]
\end{prop}
\begin{proof}
Indeed, assume that for some $k\ge 0$ we have $\chi_i = 0$ for all $i\le k$.
Then equation~(\ref{eq:chi}) from the proof of Theorem~\ref{thm:reduction} implies also that $\partial \chi_{k+1}=0$. Since $\chi$ is normal, we also have $\partial^*\chi_{k+1}=0$. So, $\chi_{k+1}=H\chi_{k+1}$. By induction, we get that $\chi=0$ if and only if $H\chi=0$.
\end{proof}

In particular, if the cohomology spaces $H^1_{r}\big(\g_{-}, \bar\g^{\perp}\big)$ vanish identically for all $r\ge 1$, then any embedding $\varphi$ of type $(\g_-,\V)$ is locally equivalent to the standard one. We say in this case that the standard embedding $\varphi_{{\rm model}}$ is \emph{rigid}.

So, we arrive at a natural question, for which irreducible representations the positive part of the cohomology $H^1_+\left(\g_{-}, \bar\g^{\perp}\right)$ does not vanish. By the following proposition, the only rational homogeneous varieties of simple Lie groups, which might be non-rigid are limited to $(i)$ projective spaces,
$(ii)$ quadrics, $(iii)$ $A$-type (Lagrange) contact spaces,
$(iv)$ $C$-type (projective) contact spaces.

\begin{prop}\label{thm3} Let $\g = \bigoplus \g_p $ be a simple graded Lie algebra and $U$ an irreducible submodule of the $\g$-module $\gl(\V)$. The cohomology group $H^1_r(\g_-, U)$ vanishes for $r \ge 1$ except for the following cases$:$
\begin{enumerate}\itemsep=0pt
\item[$1.$] $(A_3,\Sigma)$ with $\Sigma=\{\alpha_2\}$,

$(A_l, \Sigma) \ (l\ge 1 )$ with $\Sigma = \{\alpha_1\}$, $\{\alpha_l\}$, or $\{\alpha_1, \alpha_l\}$,

\item[$2.$] $(B_2, \Sigma)$ with $\Sigma = \{\alpha_1\}$, or $\{\alpha_2 \}$,

$(B_l, \Sigma) \ (l\ge 3 )$ with $\Sigma = \{\alpha_1\}$,

\item[$3.$] $(C_l, \Sigma) \ (l\ge 3 )$ with $\Sigma = \{\alpha_1\}$,

\item[$4.$] $(D_4, \Sigma) $ with $\Sigma = \{\alpha_1\}$, $\{\alpha_3 \}$, or $\{\alpha_4 \}$,

$(D_l, \Sigma) \ (l\ge 5 )$ with $\Sigma = \{\alpha_1\}$,
\end{enumerate}
	where $\Sigma$ denotes the subset of the simple roots $\{\alpha_1, \dots, \alpha_l\}$
	which defines the gradation of $\g$, that~is, $\g_{\alpha} \subset \g_p$ for
	$\alpha = a_1\alpha_1 + \cdots + a_l\alpha_l$
	if
	$ \sum	_{\alpha_i \in \Sigma} a_i	= p	$.
\end{prop}
\begin{proof}
	Let $\mu$ be the lowest weight of $U$.	It is written as a linear combination of the fundamental weights: $\mu = \sum_{i=1}^l \mu_i\omega_i$ with non-positive integers $\mu_i$, where $\omega_i$ is the $i$-th fundamental weight.
	By the theorem of Kostant~\cite{kost}, we have
	\[
	H^1(\g_- , U) = \bigoplus _{\alpha \in \Sigma}L^{\g_0}(-\sigma_{\alpha}.(-\mu)),
	\]
	where $L^{\g_0}(-\sigma_{\alpha}.(-\mu))$ denotes a $\g_0$-irreducible module of lowest weight:
	\[
	-\sigma_{\alpha}.(-\mu) = -(\sigma_{\alpha}(-\mu + \rho ) -\rho)
	= \sigma_{\alpha}(\mu - \rho ) +\rho,
	\]
	$\sigma_{\alpha}$ being the reflection with respect to $\alpha $ and
	$\rho = \omega_1 + \cdots + \omega_l $.
	Since $\sigma_{\alpha_i}(\omega_j) = \omega_j - \delta_{ij}\alpha_i$,
	we~have
	\[
	\sigma_{\alpha_i}(\mu - \rho ) +\rho = \sigma_{\alpha_i}(\mu ) +\alpha_i.
	\]
	
	Now let us compute the degree of this weight. Since the grading of $U$ is determined by the eigenvalues of the grading element $E\in\g$, the degree of $\lambda = \sum \lambda_j \alpha_j$ is defined as
	\[
	\deg(\lambda) = \sum _{\alpha_i \in \Sigma}\lambda_i.
	\]
	For $\alpha_i \in \Sigma$ we have
	\[
	\deg(\sigma_{\alpha_i}(\mu ) +\alpha_i) = \deg(\sigma_{\alpha_i}(\mu )) + 1 .
	\]
	Let $C^{-1}=(p_{ij})$ be the inverse of the Cartan matrix $C$. Note that its components $p_{ij}$ are all positive. We have
	\[
	\sigma_{\alpha_i}(\mu ) = \sigma_{\alpha_i}\bigg(\sum_{j=1}^l \mu_j \omega_j \bigg) = \sum_{j=1}^l \mu_j \omega_j - \mu_i\alpha_i
	= \sum_{j,k=1}^l \mu_j p_{jk} \alpha_k - \mu_i\alpha_i.
	\]
	Therefore we have
	\[
		\deg (\sigma_{\alpha_i}(\mu )) = \sum _{j=1}^l\mu_j\sum_{k'} p_{jk'} - \mu_i = \sum _{j\neq i}\mu_j \sum_{k'}p_{jk'} + \mu_i\bigg(\sum_{k'} p_{ik'}-1\bigg),
	\]
	where the summation $\sum_{k'}$ is taken over all $k'$ such that $\alpha_{k'} \in \Sigma$.
	We then deduce from this that if $(\g, \Sigma)$ has the following property:
	\[
	\sum _{\alpha_{k} \in \Sigma} p_{ik}-1 > 0 \qquad \text {for all}\quad i\quad
	\text{such that}\quad \alpha _i \in \Sigma,
	\]
	then
	$H^1_r(\g_-, V) = 0 $ for $r \ge 1$.
	Now the theorem follows immediately from the table of $(p_{ij})$
	as found, for instance, in Humphreys' book~\cite[p.~69]{humph}.
\end{proof}

Consider now the cases from Proposition~\ref{thm3} in more detail. Let $\pf\subset\g$ be the parabolic subalgebra in the complex simple Lie algebra that corresponds to one of the cases in Proposition~\ref{thm3}. Considering $\pf$ up to $\Aut(\g)$, is equivalent to treating $\Sigma$ up to the automorphism of the Dynkin diagram of the corresponding root system. Taking this into account, we may reduce the number of cases to one of the following:
\begin{enumerate}\itemsep=0pt
\item $(A_l,\{\alpha_1\})$,\quad $l\ge 1$,
\item $(B_l,\{\alpha_1\})$, \quad$l\ge 2$,
\item $(C_l,\{\alpha_1\})$, \quad$l\ge 2$,
\item $(D_l,\{\alpha_1\})$, \quad$l\ge 4$,
\item $(A_3,\{\alpha_2\})$,
\item $(A_l,\{\alpha_1,\alpha_l\})$, \quad$l\ge 2$.
\end{enumerate}
We use here the identification of $(B_2,\{\alpha_2\})$ with $(C_2,\{\alpha_1\})$.

In the case $(A_l,\{\alpha_1\})$, we get
\[
\deg(\sigma_{\alpha_1}(\mu))=
 \sum_{j=2}^l p_{j1}\mu_j + \mu_1(p_{11}-1) =
 \frac{1}{l+1}\bigg(\sum_{j=2}^l (l+1-j)\mu_j - \mu_1\bigg) \ge 0.
\]
This gives an explicit condition on the lowest weight $\mu$ of the module $U$, such that $H^1(\g_{-},U)$ is concentrated in the positive degree.

In $B_l$, $C_l$, and $D_l$ cases, we note that $p_{11}=1$. So, we get
\[
	\deg(\sigma_{\alpha_1}(\mu)) = \sum_{j=2}^l p_{j1}\mu_j + \mu_1(p_{11}-1)
	= \sum_{j=2}^l p_{j1}\mu_j \ge 0.
\]
Since $\mu_j\le 0$, we immediately get that necessarily $\mu_2=\dots=\mu_l=0$, and under this condition, $H^1(\g_{-1},U)$ is concentrated in degree $1$.

Similarly, in the case $(A_3,\{\alpha_2\})$, we have $p_{22}=1$. So, we get that $H^1(\g_{-1},U)$ is concentrated in positive degree if any only if $\mu_1=\mu_3=0$, and this degree is equal to 1 in this case.

Finally, in the case $(A_l,\{\alpha_1,\alpha_l\})$, we get
\[
	\deg(\sigma_{\alpha_1}(\mu)) =
	\sum_{j=2}^{l} (p_{j1}+p_{jl})\mu_j + \mu_1(p_{11}+p_{1l}-1)
	= \sum_{j=2}^{l} \mu_j \ge 0.
\]
So, again we get $\mu_2=\dots=\mu_l=0$ as the necessary condition for the module $L^{\g_0}(-\sigma_{\alpha_1}.(-\mu))\subset H^1(\g_{-},U)$ to be concentrated in the positive degree, which is equal to 1 in this case.

Similarly, $L^{\g_0}(-\sigma_{\alpha_l}.(-\mu))\subset H^1(\g_{-},U)$ is concentrated in positive degree (equal to $1$) only for modules with the lowest weight $\mu_l\om_l$, $\mu_l\le 0$.

Summing up, we get
\begin{prop}\label{prop7}
	Let $\mu=\mu_1\om_1+\dots+\mu_l\om_l$, $\mu_i\le 0$, be the lowest weight of an irreducible module~$U$ over one of the simple Lie algebras of rank~$l$ from Proposition~$\ref{thm3}$. Suppose $H^1(\g_{-},U)$ has non-trivial components in positive degree. Then
	\begin{enumerate}\itemsep=0pt
		\item[$1.$] In the case $(A_l,\{\alpha_1\})$, we have
		\[
		N = \sum_{j=2}^l (l+1-j)\mu_j - \mu_1 \ge 0,
		\]
 where $N$ is a multiple of $l+1$. Then $H^1(\g_{-},U)$ is an irreducible $\g_0$-module concentrated in degree $N/(l+1)+1$.

\item[$2.$] In the cases $B_l$ $(l\ge 2)$, $C_l$ $(l\ge 2)$, $D_l$ $(l\ge 4)$ and $\Sigma=\{\alpha_1\}$, and for $A_3$, $\Sigma=\{\alpha_2\}$, we~have $\mu_i=0$ unless $\alpha_i\in\Sigma$. Then $H^1(\g_{-},U)$ is an irreducible $\g_0$-module concentrated in degree $1$.

\item[$3.$] In the case $(A_l,\{\alpha_1,\alpha_l\})$, we have $\mu_2=\dots=\mu_l=0$ or $\mu_1=\dots=\mu_{l-1}=0$. Each of these conditions leads to an irreducible $\g_0$-module in $H^1(\g_{-},U)$ concentrated in degree $1$.
	\end{enumerate}
 All other cases from Proposition~{\rm \ref{thm3}} are equivalent to the above via the automorphisms of the Dynkin diagram of the corresponding root system.
\end{prop}

\subsection{Simple Lie algebras of rank 2 and their adjoint representations}
\label{rank2ex}

In view of Propositions~\ref{thm3} and~\ref{prop7} let us consider, as first examples of extrinsic parabolic geometry on non-trivial filtered manifolds,
the following simple Lie algebras of rank $2$ with contact type gradings and their adjoint representations:
\begin{enumerate}\itemsep=0pt
	\item[$(i)$] $(A_2, \{ \alpha_1, \alpha_2\})$, $\mu=\omega_1+\omega_2$,
	\item[$(ii)$] $(C_2, \{ \alpha_1\})$, $\mu=2\omega_1$,
	\item[$(iii)$] $(G_2,\{ \alpha_2\})$, $\mu=\omega_2$,
\end{enumerate}
where $\{ \alpha _i \}$ indicates the grading, and $\mu$ is the highest weight of the adjoint representation.

Throughout this subsection $\g$ ($= \oplus\g_p$) will denote one of the above simple Lie algebras with a contact type grading, that is, the negative part $\g_-$ of $\g$ is a Heisenberg Lie algebra, or more precisely, $\g_- = \g_{-2} \oplus \g_{-1}$, $\dim \g_{-2} = 1 $, and the bracket $\g_{-1} \times \g _{-1} \to \g_{-2}$
is non-degenerate. In~the cases $(i)$, $(ii)$ $\dim \g_- = 3$ and $5$ in the case $(iii)$.

The representation $\rho\colon \g \to \mathfrak{gl}(V)$ that we consider here is the adjoint representation, so that $V = \g$, while we give a grading $\oplus V_i $ on $V$ by shifting: $V_i = \g_{i-2}$, to adjust to the usual orders of related differential equations. We fix a filtration $\phi$ of $V$ by setting $\phi^i = \oplus_{p \ge i }V_p$.

In each case we are going to inspect an osculating map
\[
\varphi \colon\ (M, \mathfrak f ) \to L/ L^0 \subset \Flag(V,\phi)
\]
and the associated differential equation of symbol type $(\g_-, V, L)$ with a suitable choice of a Lie subgroup $L$ of ${\rm GL}(V)$. Note that the filtered manifold $(M, \mathfrak f)$ has $\g_-$ as symbol, and therefore is a contact manifold of dimension $3$ or $5$.

Consider in detail case $(ii)$. We have $\g=\mathfrak{sp}(U)$, where $U$ is a 4-dimensional vector space equipped with a non-degenerate symplectic form. We endow $\g$ with a contact type grading so that
\[
	\dim \g_{\pm 2} = 1,\qquad \dim \g_{\pm1}=2,\qquad \dim\g_0=4.
\]
As already said, $V = \g$ with $V_i = \mathfrak g_{i-2}$ as a graded $\g$ module. We assume also that $G = {\rm Sp}(U)$.

Consider an osculating map
\[
	\varphi\colon\ (M, \mathfrak f) \to L/L^0 \subset \Flag(V, \phi)
\]
for a Lie subgroup $L$ of ${\rm GL}(V)$ containing $G$. Then it is generated by the projective embedding
\[
	\varphi^4 \colon\ (M, \mathfrak f ) \to P(V).
\]
In particular the standard model embedding
\[
	\varphi_{\mathrm{model}} \colon\ G/G^0 \to \Flag(V, \phi)
\]
is generated by the projective embedding
\[
	\varphi^4_{\mathrm{model}} \colon\ G/G^0 \to P(V),
\]
which may be called the contact Veronese embedding: It is isomorphic to the following projective embedding
\[
	P(U) \to P\big(S^2(U)\big).
\]
Recall here that $\mathfrak{sp}(U)$ may be identified with $S^2(U)$.

As a choice of subgroup $L\subset G$ we can take $O(V)$ ($=O(V,\kappa)$), where $\kappa$ is the Killing form of~$\g$. So, we have natural embeddings
\[
G\subset O(V)\subset {\rm GL}(V), \qquad \g\subset \mathfrak{o}(V,\kappa)\subset \mathfrak{gl}(V).
\]
Note that in the real case we have $\mathfrak{o}(V,\kappa)=\mathfrak{o}(6,4)$.

It is easy to check that the relative prolongations of $\g_{-}$ in $\mathfrak{o}(V)$ and $\gl(V)$ coincide with $\g$ and $\bar\g=\g+\mathfrak{z}$ respectively.

Let us compute cohomology groups $H^1_+(\g_-, \mathfrak{o}(V)/\g)$ and $H^1_+\left(\g_-, \gl(V)/\bar\g\right)$. In~terms of $\g$-modules, we have the irreducible decomposition
\[
\mathfrak{gl}(V)=\mathfrak{gl}(10)=\Gamma_{0,0}\oplus \Gamma_{0,1}
\oplus \Gamma_{2,0}\oplus \Gamma_{0,2}\oplus \Gamma_{2,1}\oplus \Gamma_{4,0}.
\]
Here $\mathfrak{l}=\mathfrak{o}(6,4)=\Gamma_{2,0}\oplus \Gamma_{2,1}$.
For the relative prolongation we have $\Prol(\mathfrak{g}_-,\mathfrak{o}(6,4))=\mathfrak{g}$, so
\[
\bar\g=\g=\Gamma_{2,0}=\mathfrak{sp}(4), \qquad
\bar\g^{\perp}=\Gamma_{2,1}.
\]
Then by Proposition~\ref{prop7}, we have no fundamental invariants of $\varphi$:
\[
H^1_+\big(\mathfrak{g}_-, \mathfrak{g}^{\perp}\big)=0.
\]

On the other hand, if we take $\mathfrak{l}$ to be the whole $\mathfrak{gl}(10)$, the relative prolongation of $\g_{-}$ becomes $\Prol(\g_-,\mathfrak{gl}(10))=\g\oplus\mathfrak{z}=\bar\g$.
Then
\[
H^1_+\big(\mathfrak{g}_-,\bar\g^{\perp}\big)=H^1_1(\mathfrak{g}_-,\Gamma_{4,0})
=S^5\mathbb{R}^2\qquad (\dim=6).
\]
It corresponds to the highest weight $-\alpha_1+2\alpha_2$ (and the lowest weight $\alpha_1-2\alpha_2$).

The model equation can be written in this case as
\[
\begin{cases}
X^3u =0, \\
XYXu =0, \\
YXYu =0, \\
Y^3u =0
\end{cases}
\]
with
\[
X=\frac{\partial}{\partial x}+\frac12y\frac{\partial}{\partial z}, \qquad
Y=\frac{\partial}{\partial y}-\frac12x\frac{\partial}{\partial z}
\]
for a contact coordinate system $(x,y,z)$ of $G/G^0$: The contact distribution $D$ is defined by a~contact form
\[
\theta={\rm d}z-\frac12(y{\rm d}x-x{\rm d}y),
\]
therefore $X$, $Y$ are sections of $D$.

Now consider its deformed version
\begin{equation*}\tag{$\E_a$}
\begin{cases}
X^3u=aY^2u,\\
XYXu=0,\\
YXYu=0,\\
Y^3u=0,
\end{cases}	
\end{equation*}
where $a$ is an arbitrary constant. We can verify that it still has a $10$-dimensional solution space spanned by
\[
\begin{cases}
1,\ x,\ y,\ z,\ x^2,\ xy,\ xz, \\[1ex]
y^2+ \dfrac{a}{3}x^3, \\[1.5ex]
yz-\dfrac{a}{24}x^4, \\[1.5ex]
z^2+\dfrac{a}{120}x^5,
\end{cases}
\]
which defines an explicit (local) deformation of the standard embedding.

Let us prove that this deformation is non-trivial.
\begin{prop}\label{deformCase2}
	If $a\ne 0$, then $(\E_a)$ is not locally equivalent to $(\E_0)$ as linear differential equations.
\end{prop}	
\begin{proof}
	Let us show that the structure function $\chi=\chi_1$ of an extrinsic bundle $Q$ associated to the equation~$(\E_a)$ does not vanish if $a\ne 0$, which will then prove by Theorem~\ref{thm:reduction} that $(\E_a)$ is not isomorphic to $(\E_0)$ if $a\ne 0$.
	
	We first note that since the principal parts of ($\E_a$) are the same for all $a$, the symbol of ($\E_a$) is isomorphic to that of the model equation and is equal to $(\g_{-},V)$ as graded module, where $\g=\fsp(4,\R)$ and $V=\g$ as graded module (modulo the shift of the grading).
	
	Next, let us note that ($\E_a$) satisfies all compatibility conditions for any $a$, as suggested by the computer computation mentioned above.
	
	Let $X$, $Y$ be as above and let $Z=-[X,Y]=\frac{\partial}{\partial z}$. By simple computation, we have
	\[
	\begin{cases}
	X^2Y=-XZ,\\
	Y^2X = -YZ,
	\end{cases}
	\]
	and for the 4-th order
	\[
	\begin{cases}
	X^4 = -aYZ,\\
	X^3Y = 0,\\[1ex]
	X^2Y^2 =\dfrac12 Z^2,\\[1ex]
	XY^3=0,\\
	Y^4=0,\\
	X^2Z = 0,\\[1ex]
	XYZ = -\dfrac12 Z^2,\\[1ex]
	Y^2Z = 0.
	\end{cases}
	\]
	The value of $Z^2$ remains free.
	
	For the 5-th order we have
	\[
	X^5 = \frac{a}{2} Z^2,
	\]
	and all other 5-th order derivatives
	\begin{gather*}
	X^4Y,\quad X^3Y^2,\quad \dots \\
	X^3Z,\quad X^2YZ,\quad \dots \\
	XZ^2,\quad YZ^2
	\end{gather*}
	vanish.
	
	We then see easily that all derivatives of order higher than~$5$ vanish. Thus, the equation~($\E_a$) is prolonged to a weighted involutive system of order $5$ and has a 10-dimensional solution space.
	
	Now note that since the symbol of~($\E_a$) is isomorphic to $(\g_{-}, \gr V)$, this means that
	\[
		(\gr V)^* \cong U(\g_{-})\otimes (\gr_{-2} V)^*/I,
	\]
	where $U(\g_{-})$ denotes the universal enveloping algebra of $\g_{-}$ and $I$ denotes the left $U(\g_{-})$-module generated by the generators of model equation ($\E_0$), that is
	\[
		X^3\otimes Z^*,\qquad Y^3\otimes Z^*,\qquad XYX\otimes Z^*,\qquad YXY\otimes Z^*,
	\]
	where $Z^*$ denotes the basis of $(\gr_{-2} V)^*$ dual to $Z$.
	
	Recalling our convention of grading, e.g., $\gr_q V^*=(\gr_{-q} V)^*$, we note that the isomorphism above is an isomorphism of $\g_{-}$-modules, which can be given explicitly as follows.
	
	Let $\{A_1,\dots, A_{10}\}$ be a basis of $\fsp(4,\R)$ and $\{A_1^*,\dots,A_{10}^*\}$ its dual basis of $\fsp(4,\R)^*$ given by	
\begin{gather*}
A_1 = E_{14},\ \quad\qquad\qquad\quad\, A_4 =E_{23},\ \quad\qquad\qquad\quad\,
A_7 =E_{32} ,\ \quad\qquad\qquad\quad\, A_{10} = E_{41}.
\\
A_2 = \frac{1}{\sqrt{2}}(E_{13}+E_{24}),\quad \ A_5 =\frac{1}{\sqrt{2}}(E_{11}-E_{44}),\quad \
A_8 = \frac{1}{\sqrt{2}}(E_{21}-E_{43}) ,
\\
A_3 = \frac{1}{\sqrt{2}}(E_{12}-E_{34}),\quad\ A_6 = \frac{1}{\sqrt{2}}(E_{22}-E_{33}), \quad \
A_9 = \frac{1}{\sqrt{2}}(E_{31}+E_{42}),
\end{gather*}
	We identify $X=A_8$, $Y=A_9$, $Z=A_{10}$.
	
	Then we can easily verify that the following map gives a $\g_{-}$-module isomorphism
	\[
		\begin{pmatrix} A_1^* \\ A_2^* \\ \vdots \\ A_{10}^*
		\end{pmatrix}
		\mapsto
	 \begin{pmatrix}
		-\frac{1}{2}Z^2 \\
		YZ \\
		XZ \\
		-Y^2 \\
		-\frac{1}{\sqrt{2}}Z\\
		-\frac{1}{\sqrt{2}}(XY+YX)\\
		X^2\\
		Y \\
		X\\
		1
		\end{pmatrix} \cdot Z^* .
	\]
	
	We denote by $P$ the 10 by 1 matrix of differential operators that appear in the right hand side of the above correspondence. Let $\{\theta_1, \dots,\theta_{10}\}$ be a fundamental system of solutions of ($\E_a$) and let
	\[
	\theta = (\theta_1, \dots,\theta_{10}),\qquad \Theta =P\theta, \qquad \Phi=\Theta^{-1}.
	\]

	Note that $\Theta$ and $\Phi$ are ${\rm GL}(V)\,(={\rm GL}(10,\R)$)-valued local functions on $M$. Similarly as in~Example~\ref{ex4new} we may regard $\Phi$ as giving a cross-section
	\[
	\begin{CD}
	Q(0) @= Q(0) @>\iota >> {\rm GL}(V) \\
	@A\Phi AA @AA\Theta^{-1} A @VVV\\
	M @= M @>\varphi >> \Flag(V,\phi),
	\end{CD}
	\]
	where $\varphi$ is an embedding corresponding to ($\E_a$) and $Q(0)$ the associated extrinsic bundle (recall Section~\ref{ss42}).
	
	Therefore, the $\gl(V)$-valued form $\Omega$ given by
	\[
	 \Omega = \Phi^{-1}\,{\rm d}\Phi = -{\rm d}\Theta\,\Theta^{-1}
	\]
	is the pull-back to $M$ of the canonical form $\om(0)$ on $Q(0)$. Write
	\[
		\Omega = F_0\om_0 + F_1\om_1 + F_2\om_2,
	\]
	where $\{\om_0,\om_1,\om_2\}$ is the coframe dual to $\{Z,X,Y\}$. By using the relations of ($\E_a$) we can easily write down $\{F_i\}$ and see that $F_0$ and $F_2$ do not depend on the parameter $a$ and
	\[
	F_1 = aE_{74} + \overset{\circ}{F}_1,
	\]
	where $E_{74}$ is the matrix element of $(7,4)$ and $\overset{\circ}{F}_1$ does not depend on~$a$.
	
	Recalling the discussion of Section~\ref{sec5}, we see that the 1-st order component $\chi(0)_1$ of the structure function $\chi(0)$ of $Q(0)$ is then represented by $\gamma= E_{74}\otimes A^*_8$.
	
	Now take a closer look at the complex
	\[
		0\to \bar\g_1 \xrightarrow{\partial} \Hom\left(\g_{-}, \gl(V)/\bar \g\right)_1 \xrightarrow{\partial} \Hom\left(\wedge^2 \g_{-}, \gl(V)/\bar \g\right)_1.
	\]
	
	If we write down explicitly the adjoint representation of $\fsp(4,\R)$ with respect to the basis $\{A_1,A_2,\dots,A_{10}\}$, it is not difficult to verify that~$\gamma$ is a cocycle but not a coboundary. Hence it~defines a non-trivial cohomology in $H^1_1\left(\g_{-},\gl(V)/\bar\g\right)$. It then proves, by virtue of corollary to~Theorem~\ref{thm:reduction}, that ($\E_a$) is not isomorphic to ($\E_0$) for any $a\ne 0$.
\end{proof}

It should be remarked that $\gamma=E_{74}\otimes A_8^*$ has weight $\alpha_1-2\alpha_2$ and in turn gives a lowest weight vector of $H^1\left(\g_{-},\gl(V)/\bar\g\right)$. Note also that systems ($\E_{a_1}$) and ($\E_{a_2}$) for different non-vani\-shing~$a_1$ and $a_2$ are equivalent to each other.

Cases $(i)$ of the contact $A_2$ geometry was already treated in Example~\ref{ex:A2} of Section~\ref{sec4}. In~particular, we have the standard model given by the Segre embedding
\[
P(U) \breve{\times} P(U^*) \to \ P\left( U \breve{\otimes} U^*\right)\!,
\]
where $\dim U =3$. The corresponding model equation can be locally expressed in the following form of a system of (weighted) second order differential equations
\[
\begin{cases}
	X^2u =0, \\
	Y^2u =0.
\end{cases}
\]

Similar to Proposition~\ref{deformCase2} we can show that the solution space of the system
\[
\begin{cases}
	X^2u =aYu, \\
	Y^2u =0,
	\end{cases}
\]
is 8-dimensional and provides a non-trivial deformation of the Segre embedding for any non-zero constant $a$.

Finally, let us describe the standard model for case $(iii)$ for the contact $G_2$ geometry. Let $U$ be the $7$-dimensional standard representation of $\g$. It is easy to check that the $\g$-module $\Lambda^2U$ decomposes as $V\oplus U$, where, as above, $V=\g$ is the adjoint representation. The projection to the first summand gives us $\g$-invariant map $\Lambda^2U\to V$, while the projection to the second summand defines an anticommutative multiplication on $U$.

The standard embedding $\varphi_{\mathrm{model}}$ is generated in this case by the following projective em\-bedding
\[
\varphi^4_{\mathrm{model}}\colon\ G/G^0=N_0\Gr_2(U)\longrightarrow P^{13}=P(V),
\]
where $N_0\Gr_2(U)$ is the Grassmann manifold of all $2$-dimensional null subalgebras in $U$. We~may call this map \emph{the contact sub-Pl\"ucker embedding}.

Using Proposition~\ref{prop7} we see that the cohomology describing the fundamental invariants of~osculating maps $\varphi\colon (M,\f)\to L/L^0\subset \Flag(V,\phi)$ is trivial in this case. So the standard model embedding $\varphi_{\mathrm{model}}$ is rigid.

We remark that the first cohomology $H^1(\g_{-},V) =\bigoplus H^1_{r}(\g_{-},V)$ represents a system of~differential equations, and $V=\oplus V_{q}$ represents its solution space, see~\cite{Mor2006}. In~our case the computation shows that
\[
H^1(\g_{-},V)=H^1_{0}(\g_{-},V)\cong \Hom_{0}\big(S^2\g_{-1},V_0\big).
\]
Here $\Hom_{0}\big(S^2\g_{-1},V_0\big)$ is the 7-dimensional irreducible $\g_{0}$-submodule of $\Hom\big(S^2\g_{-1},V_{0}\big)$. Exp\-li\-citly, we get the system of (weighted) linear differential equations on $M$ which is $G_2$-invariant and has $14$-dimensional solution space defining the standard embedding
\[
\begin{cases}
Z^2u = W^2u = YZu = XWu =0, \\
\big(2ZX+Y^2\big)u = \big(2YW+X^2\big)u =0, \\
(XY+YX+WZ+ZW)u =0,
\end{cases}
\]
where we take the local coordinates $(x,y,z,w,v)$ in $X$ such that
the contact distribution is represented by $\{ {\rm d}v-2y{\rm d}x+x{\rm d}y+z{\rm d}w=0 \}$, \
$Z=\frac{\partial}{\partial z},\ W=\frac{\partial}{\partial w}-z\frac{\partial}{\partial v},\
X=\frac{\partial}{\partial x}+2y\frac{\partial}{\partial v},\
Y=\frac{\partial}{\partial y}-x\frac{\partial}{\partial v}$.
See~\cite{IMT} for more details. Here
$u$ is an unknown function $u=u(x,y,z,w,v)$ on $X$.

\subsection{Realizability of deformations}
Proposition~\ref{thm3} lists all possible types of non-rigid rational homogeneous varieties of simple Lie groups. However, it does not provide explicit examples when they are really rigid.

For quadrics one can take the standard representation of $B_l$, $D_l$ or $A_3\cong D_3$, as a non-degenerate quadric in $P^n$ is clearly non-rigid.

In case of projective spaces one can consider, for example, the osculating map $\varphi\colon P^n\to \Flag\left(\R^{(n+1)(n+2)/2},\phi\right)$ generated by the Veronese embedding $\nu_2\colon P^n\to P^{(n+1)(n+2)/2-1}$, which corresponds to the $S^2(V)$ representation of ${\rm SL}(V)$. The filtration $\phi$ has type $((n+1)(n+2)/2$, $n+1,1)$ in this case.

Let $(x_1,\dots,x_n)$ be affine coordinates on $P^n$. In~these coordinates the Veronese embedding is specified by the space of all polynomials of degree $\le 2$ in $x_1,\dots,x_n$. The corresponding ${\rm PSL}(n+1,\R)$ invariant system of PDEs
\[
\frac{\partial^3u}{\partial x_i\partial x_j\partial x_k} = 0.
\]
We can provide an explicit deformation of the Veronese embedding by replacing the equation $\frac{\partial^3 u}{(\partial x_1)^3}=0$ with
\begin{gather*}
\frac{\partial^3 u}{(\partial x_1)^3} = \frac{\partial^2 u}{(\partial x_2)^2}.
\end{gather*}
It is easy to see that this deformed system is still compatible and has the solution space of~dimen\-sion $(n+1)(n+2)/2$, thus providing an explicit deformation of the Veronese embedding. Computer algebra calculations show that the dimension of the symmetry algebra of this system is less than the one of the trivial system implying that this deformation is non-trivial.

Consider now the adjoint variety of $\mathfrak{sp}(2n,\R)$. It was explicitly demonstrated in the previous subsection that it is non-rigid for the case of $\mathfrak{sp}(4,\R)$.

Assume now that $n$ is arbitrary. Let $V=\mathfrak{sp}(2n,\R)$ be the adjoint representation. Its~hig\-hest weight is $2\om_1$, and $\bar\g= \mathfrak{csp}(2n,\R)$. We have the following decompositions of $\gl(V)$ and $\bar\g^\perp\subset \gl(V)$:
\begin{gather*}
\gl(V) = V_{4\om_1}\oplus V_{2\om_1+\om_2} \oplus V_{2\om_2} \oplus V_{2\om_1} \oplus V_{\om_2} \oplus V_0,
\\
\bar\g^\perp = V_{4\om_1} \oplus V_{2\om_1+\om_2} \oplus V_{2\om_2} \oplus V_{\om_2}.
\end{gather*}
In particular, $H^1_+(\g_{-},V_{4\om_1})$ does not vanish, which suggests that the adjoint variety of $\mathfrak{sp}(2n,\R)$ is non-rigid for any $n\ge 3$.

Consider now $A$-type contact case in more detail. It corresponds to $(A_l, \{\alpha_1,\alpha_l\})$. Its minimal embedding is given by the highest root orbit of the adjoint representation of ${\rm SL}(l+1,\R)$. However, it can be shown that this orbit is non-rigid only for $l=2$. Indeed, we have the following decomposition of $\bar\g^{\perp}$, $V=\mathfrak{sl}(l+1,\R)$ depending on $l$:
\begin{gather*}
\bar\g^{\perp} = V_{2\om_1+2\om_l}\oplus V_{2\om_1+\om_{l-1}}\oplus V_{\om_2+2\om_{l}}\oplus V_{\om_2+\om_{l-1}} \oplus V_{\om_1+\om_l}, \qquad l\ge 4,
\\
\bar\g^{\perp} = V_{2\om_1+2\om_3}\oplus V_{2\om_1+\om_2}\oplus V_{\om_2+2\om_3}\oplus V_{2\om_2} \oplus V_{\om_1+\om_3}, \qquad l=3,
\\
\bar\g^{\perp} = V_{2\om_1+2\om_2}\oplus V_{3\om_1}\oplus V_{3\om_2}\oplus V_{\om_1+\om_2}, \qquad l=2.
\end{gather*}
In particular, we see that for $l\ge 3$ this decomposition has no components $U$, for which the cohomology $H^1_+(\g_{-},U)$ is non-trivial (see Proposition~\ref{prop7}).

\begin{lem}
	Consider the injective map $\mathfrak{sl}(l+1,\R)\to \gl(V)$, $l\ge2$, where $V$ is the irreducible representation with the highest weight $\rho=\omega_1+\omega_2+\dots+\omega_l$. Then the decomposition of $\gl(V)$ contains the components $V_{(l+1)\omega_1}$ and $V_{(l+1)\omega_l}$ with multiplicity $1$.
\end{lem}
\begin{proof}
	Indeed, the multiplicity of $V_{(l+1)\omega_1}$ and $V_{(l+1)\omega_l}$ in the decomposition of $\gl(V_{\rho})$ can be computed using the Littlewood--Richardson rule~\cite{fulton}. To do this, we note that the representation~$V_\rho$ is self-dual, and, hence $\gl(V_\rho)$ is isomorphic to $V_{\rho}\otimes V_{\rho}$. In~terms of Young tableaux the representation $V_{\rho}$ corresponds to the tableau $\lambda$ of shape $(l,l-1,\dots,1)$ ($l$ boxes in the first row, $l-1$ boxes in the second row, \dots, $1$ box in the last row).
	
	As $V_{(l+1)\omega_1}$ and $V_{(l+1)\omega_l}$ are dual to each other, it is clear that their multiplicities in the decomposition of $V_{\rho}\otimes V_{\rho}$ are the same. So, we can concentrate on computing the multiplicity of $V_{(l+1)\omega_1}$, which corresponds to any tableau of shape $(k+l+1,k,\dots,k)$ ($l+1$ rows in total). To fix $k$, we have to make sure that the tableau has exactly $l(l+1)$ (twice the number of boxes in the tableau of $V_{\rho}$). So, $k=l-1$ and we should represent $V_{(l+1)\omega_1}$ by the tableau $\nu$ of shape $(2l,l-1,\dots,l-1)$ ($l+1$ rows).
	
	Finally, applying the Littlewood--Richardson rule, the multiplicity of $V_{(l+1)\omega_1}$ is equal to the number of Littlewood--Richardson tableaux of shape $\nu/\lambda$ and of weight $\lambda$. It is easy to see that the first row of $\nu/\lambda$ should be filled by $1$. The next empty box appears in the 3rd row, and it should be filled by $2$ (otherwise we would not be able to fill in the 3rd column by the strictly increasing sequence). In~the same way, it is easy to see that the 4th row should be filled by~$2$ and $3$ and so on. So, we have a unique Littlewood--Richardson tableaux of shape $\nu/\lambda$ and of~weight~$\lambda$. See, for example, the picture below for $l=4$:
	\begin{gather*}
	\young(\bullet\bullet\bullet\bullet 1111,\bullet\bullet\bullet,\bullet\bullet 2,\bullet 23,234)\tag*{\qed}
\end{gather*}
\renewcommand{\qed}{}
\end{proof}

So, according to Proposition~\ref{prop7} the cohomology $H^1_+\left(\g_{-},\bar\g^{\perp}\right)$ does not vanish in this case. This leads to the conjecture that a rational homogeneous variety corresponding to the representation~$V_{\rho}$ is non-rigid.

\section{Equivalence problem of extrinsic geometries~-- general case}\label{sec:general}

\subsection{Generalizations to the non-principal and to the infinite-dimensional cases}

In this section we extend the discussions in Section~\ref{sec5} to obtain the invariants of extrinsic geo\-metries in more general settings.

First we drop the assumption that $L/L^0$ is embedded in a flag variety $\Flag(V, \phi)$ since it is not the representation itself of $L$ on $V$ but the filtered structure on $L/L^0$ that matters.

Second we will not assume the conditions of $\overline G^0$-invariance (C1) and (C2) which were needed for ``principal reduction'' to obtain extrinsic Cartan connections. Here we perform rather ``step reduction'' which does not necessarily lead to Cartan connections, but still gives a general algorithm to obtain the complete invariants of an extrinsic geometry.

Third we extend our discussions to the cases when $L$ may be infinite dimensional. This generalization will have interesting applications to extrinsic geometries of non-linear differential equations with respect to infinite-dimensional Lie groups such as the contact transformation groups.

To generalize to infinite-dimensional transformation groups, we rather treat infinitesimal transformations.

\begin{df}[\cite{Mor1988,sin-st65}] \emph{A transitive Lie algebra sheaf $($of infinitesimal transformations$)$ on a~filtered manifold $(N,\f)$} is a subsheaf $\cL$ (of Lie algebras) of the Lie algebra sheaf $\underline{TN}$ satisfying:
	\begin{enumerate}\itemsep=0pt
		\item[$(i)$] $\cL$ leaves invariant $\f$, that is, $[\cL, \underline{\f}^p] \subset \underline{\f}^p$ for all $p\in\Z$,
		\item[$(ii)$] $\cL$ is transitive, that is, the evaluation map $\cL_x \ni [X]_x \mapsto X_x \in T_xN$ is surjective for all $x\in N$,
		\item[$(iii)$] $\cL$ is defined by an involutive system of differential equations.
	\end{enumerate}
\end{df}

Note that the condition $(iii)$ means that $\cL$ is ``continuous'' in the sense of Lie and Cartan. For~the notion of involutive differential equations on filtered manifolds, see~\cite{Mor2002}.

Let $\cP(\cL)$ denote the pseudo-group of local transformations of $(N,\f)$ generated by $\cL$, namely by all local flows $f_t$ of all sections $X$ of $\cL$.

Let $\varphi_1\colon (M_1,\f_{M_1})\to (N,\f_N)$ and $\varphi_2\colon (M_2,\f_{M_2})\to (N,\f_N)$ be two morphisms of filtered manifolds, and let $x_1\in M_1$,\ $x_2\in M_2$.
We say that the germ of $\varphi_1$ at $x_1$ is $\cL$-equivalent to the germ of $\varphi_2$ at $x_2$, if there exist a local isomorphism $h\colon (M_1,\f_{M_1})\to (M_2,\f_{M_2})$ with $h(x_1)=x_2$ and $a\in \cP(\cL)$ such that $a\circ \varphi_1=\varphi_2\circ h$.

It is under this $\cL$-equivalence that we are going to study the equivalence problem. Before going further, let us recall some basic facts about our $\cL$.

First of all we note that each stalk $\cL_x$ has a natural filtration $\{\cL^p_x\}_{p\in\Z}$ induced from that of~$(N, \f_N)$ as defined by
\[
\cL_x^p =\begin{cases}
\big\{ [X]_x\in \cL_x \mid X_x \in\f^p_x \big\} & (p\le 0) ,
\\[1mm]
\big\{ \xi \in \cL^{p-1}_x \mid [\xi , \cL^i_x ] \subset \cL^{i+p}_x \text{ for all } i<0\big\} & (p\ge 1).
\end{cases}
\]
Then it is easy to verify that
\[
[\cL^p_x, \cL^q_x] \subset \cL^{p+q}_x.
\]
Thus, $\big(\cL_x, \{\cL^p_x\}\big)$ is a filtered Lie algebra, and passing to the quotient we get a graded Lie algebra~$\gr \cL_x$. As shown later, $\gr\f_x$ is anti-isomorphic to $\gr_{-} \cL_x \big({}=\sum_{p<0} \gr_p \cL_x\big)$.

With the notion of weighted jet bundle the filtration $\{\cL^p_x\}$ can be defined alternatively as follows. Regarding $TN$ as a filtered vector bundle over a filtered manifold $(N,\f_N)$, let $\wJ^{(k)}TN$ be the weighted $k$-th jet bundle. Denote by $\wj^{(k)}_x$ the canonical map $\underline{TN}_{\,x} \to \underline{\wJ^{(k)}TN}_{\,x}$ for $x\in N$. Then we have
\[
\cL^{p+1}_x = \big\{ [X]_x\in\cL_x \mid \wj^{(p)}_x X = 0 \big\}.
\]

Now let
\begin{gather*}
\cL^{(k)}_x =\big \{ \wj^{(k)}_xX \mid [X]_x\in \cL_x\big\} = \cL_x / \cL_x^{k+1},
\\
\cL^{(k)} = \bigcup_{x\in N} \cL^{(k)}_x.
\end{gather*}
Then $\cL^{(k)}\subset \wJ^{(k)}TN$ may be regarded as a system of differential equations defining $\cL$.

The condition $(iii)$ of the definition of Lie algebra sheaf may be understood to mean that there exists $k_0$ such that $\cL^{(k_0)}$ is weighted involutive and $\cL^{(l)}$ is the prolongation of
$\cL^{(k_0)}$ for~$l>k_0$. From the formal theory of differential equations, it is equivalent
to saying that $\cL^{(k)}$ are vector bundles for all $k$.

Taking the projective limit, we set
\[
\cL^{(\infty)}_x = \lim_{\longleftarrow} \cL^{(k)}_x,
\]
which carries the structure of filtered Lie algebra and is called the formal algebra of $\cL$ at~$x$. Note that from the transitivity of $\cL$ it follows that $\cL^{(\infty)}_x$ and $\cL^{(\infty)}_y$ are isomorphic for all~$x$,~$y$ in~a~connected component of $N$. Assuming $N$ connected if necessary, we fix a filtered Lie algebra~$\lf$, which is anti-isomorphic to all $\cL^{(\infty)}_x$, $x\in N$.

A proper generalization of the bundle $L\to L/L^0$ to infinite dimension may be to introduce the following principal bundle equipped with a Pfaff system:

\begin{thm}\label{thm:n3}
There is associated to a transitive Lie algebra sheaf $\cL$ on a filtered manifold $(N,\f)$ a principal fibre bundle $\fL$ over $N$ with structure group $L^0$ and $\lf$-valued $1$-form $\om_{\fL}$ on $\fL$ such that
	\begin{enumerate}\itemsep=0pt
		\item[$(i)$] $L^0$ is a possibly infinite-dimensional Lie group equipped with a filtration $\{L^p\}_{p\ge 0}$ consisting of closed normal subgroups of $L^0$ with the Lie algebra of $L^p$ being $\lf^p$ and
		\begin{gather*}
		L^0 = L^{(\infty)} = \lim_{\longleftarrow} L^{(k)},\\
		\fL = \fL^{(\infty)} = \lim_{\longleftarrow} \fL^{(k)},
		\end{gather*}
		where $L^{(k)} = L^0/L^{k+1}$, $\fL^{(k)} = \fL/L^{k+1}$.
		\item[$(ii)$] The map $\big(\om_{\fL}\big)_z\colon T_z\fL\to\lf$ is a filtration preserving isomorphism for all $z\in \fL$.
		\item[$(iii)$] $\langle \om_{\fL}, \tilde A\rangle = A$ for all $A\in \lf^0$.
		\item[$(iv)$] The adjoint action of $L^0$ on $\lf^0$ is extended to $\lf$ and
		\[
		R^*_a\om_{\fL} = {\rm Ad}(a)^{-1}\om_{\fL} \qquad\text{for all}\quad a\in L^0.
		\]
		\item[$(v)$] ${\rm d}\om_{\fL} + 1/2 [\om_{\fL},\om_{\fL}] = 0$.
		\item[$(vi)$] If $X$ is a section of $\cL$, then there is a unique lift $\widehat X$ of $X$ to $\fL$ such that $L_{\widehat X}\om_{\fL} = 0$. Conversely, if a local vector field $Z$ on $\fL$ satisfies $L_Z\om_{\fL} = 0$, then there exists a section $X$ of $\cL$ such that $Z= \widehat X$.
	\end{enumerate}
\end{thm}
Note that $\fL$ and $L^0$ can be infinite dimensional. But since they are projective limits of finite-dimensional objects, we can treat those infinite-dimensional objects similarly to the finite-dimensional case.

\begin{proof}
	To prove the theorem we employ the scheme developed in~\cite{Mor1993}.
	
	First we note that since $\cL$ is transitive on $N$, $(N, \f_N)$ has a constant symbol, say $\m$. Let $\cR^{(k)}(N,\f_N,\m)$ or simply $\cR^{(k)}$ be the reduced frame bundle of order $(k+1)$ and set $\cR = \cR^{(\infty)}={\displaystyle\lim_{\longleftarrow} \cR^{(k)}}$ (see~\cite[p.~316]{Mor1993}). Recall that every automorphism $h$ of $(N,\f_N)$ is uniquely lifted to an~automorphism $\hat h^{(k)}$ of $\cR^{(k)}$. Therefore we have a lift $\widehat{\cP(\cL)}^{(k)}$ of $\cP(\cL)$ to $\cR^{(k)}$. Passing to the infinitesimal, every infinitesimal automorphism $X$ of $(N,\f_N)$ is uniquely lifted to the one $\widehat X^{(k)}$ on $\cR^{(k)}$ and the Lie algebra sheaf $\cL$ to the one $\widehat\cL^{(k)}$ on $\cR^{(k)}$. Clearly, we have $\cP\big(\widehat\cL^{(k)}\big)=\widehat{\cP(\cL)}^{(k)}$.
	
	By evaluation, $\widehat \cL^{(k)}$ defines a subbundle $D^{(k)}\subset T\cR^{(k)}$, which is, as easily seen, of constant rank and completely integrable. Now choose $\mathring{z}^{(k)}\in\cR^{(k)}$ ($k=0,1,\dots$) such that $\pi^{k+1}_k\big(\mathring{z}^{(k+1)}\big)$ $=\mathring{z}^{(k)}$, where $\pi^{k+1}_k$ denotes the projection $\cR^{(k+1)}\to \cR^{(k)}$. Let $\fL^{(k)}$ be the maximal integral manifold of $D^{(k)}$ through $\mathring{z}^{(k)}$, and set
	\[
	\fL = \lim_{\longleftarrow} \fL^{(k)},\qquad
\mathring{z} = \lim_{\longleftarrow} \mathring{z}^{(k)}.
	\]
	Then we see immediately that $\fL^{(k+1)}\to \fL^{(k)}$ is a surjective submersion for all $\infty \ge k \ge -1$, where we set $\fL^{-1}=N$. Not only this, we are now going to see that $\fL^{(k)}\to \fL^{(j)}$ are all principal fibre bundles for $k\ge j$.
	
	Let $\om_{\cR}$ be the canonical 1-form on $\cR$, which is a 1-form on $\cR$ taking values in $E(\m)=\m\oplus\g^0(\m)$, where $G^0(\m)$ is the structure group of $\cR\to N$ and $\g^0(\m)$ is the Lie algebra of $G^0(\m)$.
	
	Let $\om_{\fL}$ be the restriction of $\om_{\cR}$ to $\fL$: $\om_{\cL} = \iota^*\om_{\cR}$, where $\iota$ denotes the inclusion $\fL\to \cR$. If~$\Phi\in \widehat{\cP(\cL)}$, then $\Phi^*\om_{\cR}=\om_{\cR}$ and $\Phi(\fL)=\fL$. Therefore $\Phi^*\om_{\fL}=\om_{\fL}$. It then follows that there is a subspace $E_{\fL}\subset E(\m)$ such that we have an isomorphism
	\[
	\big(\om_{\fL})_z \colon\quad T_z\fL \to E_{\fL}\qquad\text{for all}\quad z\in\fL,
	\]
	and we have
	\[
	{\rm d}\om_{\fL}+\frac12\gamma_{\fL}(\om_{\fL},\om_{\fL})=0
	\]
	with a constant structure function $\gamma_{\fL}\in \Hom\big({\wedge}^2 E_{\fL}, E_{\fL}\big)$, because $\widehat{\cP(\cL)}$ is transitive on $\fL$.
	
{\samepage
We see therefore that $\gamma_{\fL}\colon E_{\fL}\times E_{\fL}\to E_{\fL}$ defines a Lie algebra structure on $E_{\fL}$. Moreover, if~we~set
	\[
	\g^0(\fL) = \g^0(\m)\cap E_{\fL},
	\]
	the Lie algebra structure defined by $\gamma_{\fL}|_{\g^0(\fL)\times\g^0(\fL)}$ coincides with the one induced from $\g^0(\m)$.

}
	
	Let $X,Y\in\cL$. Substituting $\wX$, $\wY$ into the structure equation, we have
	\[
	\wX \om_{\fL}\big(\wY\big) - \wY \om_{\fL}\big(\wX\big) - \om_{\fL}\big(\big[\wX,\wY\big]\big)
	+\gamma_{\fL}\big(\om_{\fL}\big(\wX\big), \om_{\fL}\big(\wY\big)\big)=0.
	\]
	But since $L_{\wX}\om_{\fL}=0$, we have
	\[
	\wX \om\big(\wY\big) = \big(L_{\wX}\om_{\fL}\big)\big(\wY\big)+\om_{\fL}(L_{\wX}\wY)=\om_{\fL}\big(\big[\wX,\wY\big]\big).
	\]
	Therefore we have
	\[
	\om_{\fL}\big(\big[\wX,\wY\big]\big)+ \gamma_{\fL}\big(\om_{\fL}\big(\wX\big), \om_{\fL}\big(\wY\big)\big)=0,
	\]
	which implies that the map
	\[
	\fL_x \xrightarrow{\widehat{\phantom{X}}} \widehat\cL_z \xrightarrow{\left(\om_{\fL}\right)_z} E_{\fL} \qquad (z\in\fL,\ x=\pi_N(z)\in N)
	\]
	gives an anti-isomorphism $\cL^{(\infty)}_x \xrightarrow{\om^{\wedge}_x} (E_{\fL},\gamma_{\fL})$ of Lie algebras, where $\cL_x^{(\infty)}$ denotes the formal algebra of $\cL$ at $x$.
	
	Recalling that $E_{\fL}$ and $\cL_x^{(\infty)}$ have natural filtrations $\big\{\phi^p E_{\fL}\big\}$ and $\big\{ \phi^p \cL_x^{(\infty)} \big\}$, we remark that
	$\om^\wedge_x\big(\phi^p\cL_x^{(\infty)}\big) = \phi^p E_{\fL}$ for all $p\in\Z$. In~fact, it is easy to see that it holds for $p\le 0$. But in both filtrations the space $\phi^p$ ($p\ge 0$) is determined by
	\[
	\phi^p = \big\{ X \in \phi^{p-1} \mid [X,\phi^a]\subset \phi^{p+a} \text{ for all } a < 0\big\}.
	\]
	Therefore the assertion is valid for all $p\in\Z$.
	
\smallskip	
In particular, we have
	\begin{enumerate}\itemsep=0pt
		\item $\gr \f_{N,x}$ is anti-isomorphic to $\gr_{-}\cL_x$ for $x\in N$.
		\item For $X\in \cL_x$, $z\in \fL$ with $\pi_N(z)=x$, and $k\in\Z$, $j^k_xX = 0$ if and only if $\wX_z\equiv 0 \mod \phi^{k+1}T_z\fL$.
	\end{enumerate}
	In relation to (2), in general, we can prove that if $X$ is an infinitesimal automorphism of a~filtered manifold $(M,\f)$, then, for $k<0$, $j^k_x X=0$ if and only if $X_k\in\f^{k+1}_x$.
	
	Let us show that if $za, zb\in\fL$ for $z\in\fL$, $a,b\in G^0(\m)$, then $zab\in\fL$. In~fact, since $z,za\in\fL$ and $\widehat \cP(\cL)$ is transitive on $\fL$, there exist neighborhoods $U$, $U'$ of $x=\pi^{\cR}(z)$ in $N$ and a diffeomorphism $\Phi\colon \big(\pi^{\cR}\big)^{-1}(U)\to \big(\pi^{\cR}\big)^{-1}(U')$ such that $\Phi^*\om_{\cR} = \om_{\cR}$, $\Phi(wg)=\Phi(w)g$, $w\in \big(\pi^{\cR}\big)^{-1}(U)$, $g\in G^0(\m)$, $\Phi\big(\big(\pi^{\cR}\big)^{-1}(U)\cap \fL\big) \subset \big(\pi^{\cR}\big)^{-1}(U')\cap \fL$ and $\Phi(z)=za$, where $\pi^{\cR}$ denotes the projection $\cR\to N$. Then since $zb\in\fL \cap \big(\pi^{\cR}\big)^{-1}(U)$, we have $zab=\Phi(z)b=\Phi(zb)\in\fL$, which proves the assertion.
	
	Now it is easy to see that $\fL$ is a principal fibre bundle over $N$ with structure group $G^0(\fL)$, where $G^0(\fL)$ is a Lie subgroup of $G^0(\m)$ with Lie algebra $\g^0(\fL)$, which can be identified with $\lf^0$.
	
	It is not difficult to verify that all the statements in Theorem~\ref{thm:n3} hold.
\end{proof}

\begin{rem}
	The statement $(vi)$ of Theorem~\ref{thm:n3} means that $(\fL,\om_{\fL})$ can be regarded as a~defining equation of $\cL$, in other words, an infinitesimal automorphism $X$ of $(N,\f)$ is a section of $\cL$ if and only if the lift $\wX$ of $X$ to $\cR$ is tangent to $\fL$, which is, in turn, equivalent to saying that the lift $\wX^{(k_0)}$ of $X$ to $\cR^{(k_0)}$ is tangent to $\fL^{(k_0)}$ for certain integer $k_0$.
	
	This integer $k_0$ is given by the condition
	\[
	H^1_r(\gr_{-}\lf,\gr \lf) = 0 \qquad \text{for}\quad r\ge k_0.
	\]
	The existence of such $k_0$ is assured by the finiteness of this cohomology group (see~\cite{Mor1988}).
\end{rem}

It being prepared, let $\varphi\colon (M,\f_M)\to (N,\f_N)$ be a morphism. Then via pull-back we obtain the principal fibre bundle $\varphi^*\fL$ over $M$ with the structure group $L^0$ endowed with an $\lf$-valued 1-form $\tilde\varphi^*\om$, where $\tilde\varphi\colon \varphi^*\fL\to \fL$ is a canonical inclusion. The bundle with Pfaff forms $(\varphi^*\cL,\tilde\varphi^*\om)$, denoted also by $(Q_\varphi, \om_\varphi)$, is called an extrinsic bundle associated with $\varphi$.

From now on we assume $\gr\varphi_{*,x}\colon \gr (\f_M)_x\to \gr (\f_N)_{\varphi(x)}$ is injective for all $x\in M$, so that $\gr \varphi_{*,x}$ is an injective graded Lie algebra homomorphism.

Let $\g_{-}=\bigoplus_{p<0}\g_p$ be a graded subalgebra of $\gr\lf_{-} = \bigoplus_{p<0}\lf$.

\begin{df}
	We say that $\varphi\colon (M,\f_M)\to (N,\f_N)$ has constant symbol of type $(\g_{-}, \cL)$,
if for any $x,y\in M$ there exist $a\in \cP(\cL)$ and isomorphisms of graded Lie algebras $\beta_x$ and $\beta_y$ which make the following diagram commutative:
	\[
	\begin{tikzcd}[row sep=small]
	& \gr (\f_M)_x \arrow[r, "\varphi_{*,x}"] & \gr (\f_N)_{\varphi(x)} \arrow[dd,"\gr a"] \\
	\g_{-} \arrow[ur, "\beta_x"] \arrow[dr, "\beta_y"]& & \\
	& \gr (\f_M)_y \arrow[r, "\varphi_{*,x}"] & \gr (\f_N)_{\varphi(y)}.
	\end{tikzcd}
	\]	
\end{df}

Then we have
\begin{prop}
	The map $\varphi\colon (M,\f_M)\to (N,\f_N)$ is of type $(\g_{-},\cL)$,
if and only if $(Q_{\varphi},\om_{\varphi})$ satisfies the following criterion:
	
	For any $x\in M$ there exists $z\in Q_\varphi$ over $x$ such that $\gr \om_z (\gr (\f_M)_x ) = \g_{-}$.
\end{prop}

Now we are going to construct the invariants of extrinsic geometries $\varphi\colon (M,\f_M)\to (N,\f_N)$ under $\cL$. To avoid surplus complexity, we assume:
\begin{enumerate}\itemsep=0pt
	\item[(C0)] $\lf=\widehat\bigoplus\,\lf_p$, that is, $\lf$ is isomorphic to the completion of $\gr\lf$ obtained as a projective limit of~partial sums $\bigoplus_{p\le k}\lf_p$ for $k\ge 0$.
\end{enumerate}
We also restrict ourselves to morphisms $\varphi$ of constant symbol, say of type $(\g_{-},\cL)$.

We follow the discussions in Section~\ref{sec5} by making necessary modifications to adapt to this general case.

Let $\bar\g = \Prol(\g_{-},\lf)$ be the relative prolongation of $\g_{-}$ in $\lf$. Similarly as before we define
\[
\overline G_0 = \big\{ a\in L^{(0)} \mid {\rm Ad}(a) \g_{-} = \g_{-} \big\}
\]
and $\overline G(k)$, $\bar\g(k)$ as well as their filtrations
$\big\{\overline G(k)^p\big\}$, $\big\{\bar\g(k)^p\big\}$. Recall thus for instance,
\begin{gather*}
\bar\g(k) = \bar\g(k)^0 =
\bigoplus_{0 \le p \le k} \bar \g _p \oplus \bigoplus_{q>k} \lf_q,
\qquad \bar\g(k)^k =\bar \g _k \oplus \bigoplus_{q>k} \lf_q.
\end{gather*}

We choose complementary subspaces $\bar\g' = \bigoplus \g'_p$
to $\bar\g$ in $\lf$ so that
\[
\lf_p = \g_p \oplus \g'_p,
\]
and we have a direct sum decomposition
\[
\lf = E(k) \oplus E'(k)
\]
with
\[
E(k) = \g_- \oplus \bar \g(k), \qquad E'(k) = \bigoplus _{q \le k} \g'_q.
\]

We choose also complementary subspace $W = \bigoplus _{p \ge 1}W_p$ such that we have for all $p \ge 1$
\[
\Hom\left(\g_{-}, \lf/\bar\g\right)_p = W_p \oplus \partial ( \lf/\bar\g)_p.
\]

Let $(Q_\varphi,\om_\varphi)$ be the original extrinsic bundle induced from
$\varphi\colon (M,\f_M)\to (N,\f_N)$.

First we define $(Q(0),\om(0))$ by
\[
Q(0) = \big\{ z\in Q_\varphi \mid \gr (\om_\varphi)_z \big(\gr (\f_M)_x \big) = \g_{-} \big\}
\]
and $\om(0)=\iota^*\om_{\varphi}$, where $\iota\colon Q(0)\to Q_\varphi$ is the canonical inclusion.

Now we construct inductively
\[
Q(k)\to B(k-1), \quad \om(k), \quad \chi(k),
\]
in such a way that:
\begin{enumerate}\itemsep=0pt
	\item[$(i)$] $Q(k)$ is a principal fibre bundle over the base space
	\[
	B(k-1) = Q(k-1)^{(k-1)}= Q(k-1)/\bar G(k)^k
	\]
	with structure group $\overline G(k)^k$. Furthermore there is a canonical inclusion $\iota\colon Q(k) \to Q(k-1)$ as principal fibre bundles over $B(k-1)$,
	\item[$(ii)$] $\om(k)=\iota^*\om(k-1)$ and
	\[
	R_a ^* \omega(k) = {\rm Ad}(a)^{-1}
	\omega (k) \qquad \text{for} \quad a \in \bar G(k)^k,
	\]
	\item[$(iii)$] If we write $\om(k)_{II} = \chi(k)\om_{I}$
	according to the direct sum decomposition
	$\lf = E(k) \oplus E'(k)$, then
	\[
		\omega(k)_I \colon\ T_zQ(k) \to E(k)
	\]
	is a filtration preserving linear isomorphism for all $z \in Q(k)$,
	\item[$(iv)$] Define a map $\chi (k)\colon Q(k) \to \Hom (E(k), E'(k))$ by $\omega(k)_{II}=\chi(k) \omega(k)_I$ and decompose it~as
	\[
		\chi (k) = \chi ^-(k) + \chi ^+(k)\qquad \text{and} \qquad 	
	\chi (k)_p = \chi ^-(k)_p + \chi ^+(k)_p
	\]	
	by requiring that
	\begin{gather*}
		\chi ^-(k)_p\colon\ Q(k) \to \Hom(\g_-, E'(k))_p, \\
		\chi ^+(k)_p\colon\ Q(k) \to \Hom\bigg(\bigoplus_{0 \le i \le k}\bar \g_i, E'(k)\bigg)_p,\\
		\chi(k) = \sum \chi(k)_p, \qquad
		\chi(k)^- = \sum \chi(k)^-_p, \qquad
		\chi(k)^+ = \sum \chi(k)^+_p.
	\end{gather*}
	Then
		\begin{enumerate}\itemsep=0pt
		\item[$(a)$] $\chi(k)_i = \iota^*\chi(k-1)_i$ for $i < k$,
		\item[$(b)$] $\chi(k)_p = 0 $ for $p \le 0$,
		\item[$(c)$] $\chi^-(k)_k \in W_k$.
		\end{enumerate}
	\end{enumerate}
$$
\begin{tikzcd}[column sep=tiny,row sep = small]
\cL \ar[d] \ar[r, hookleftarrow] & Q_\varphi \ar[d] \ar[r, hookleftarrow] & Q(0) \ar[d] \ar[r, hookleftarrow] & Q(1) \ar[d] \ar[r, hookleftarrow] & Q(2) \ar[d] \ar[r, hookleftarrow] & \cdots \ar[r, hookleftarrow] & Q(k) \ar[d] \ar[r, hookleftarrow] & Q(k\!+\!1) \ar[d] &
\\
\vdots & \vdots & \vdots & \vdots & \vdots & & Q(k)^{(k+1)} \ar[d] \ar[r, hookleftarrow] & Q(k+1)^{(k\!+\!1)} \ar[d] \ar[r, equal] & B(k\!+\!1) \ar[dl, "\overline G_{k+1}"]
\\[8mm]
& & & & & & Q(k)^{(k)} \ar[d] \ar[r, equal] & B(k) \ar[dl, "\overline G_{k}"] &
\\[6mm]
& & & \vdots \ar[d] & \vdots \ar[d] & & B(k-1) \ar[dl, dotted] & &
\\[2mm]
& & \vdots \ar[d] & Q(1)^{(2)} \ar[d] \ar[r, hookleftarrow] & Q(2)^{(2)} \ar[d] \ar[r, equal] & B(2) \ar[dl, "\overline G_{2}"] & & &
\\
\vdots\ar[d] & \vdots\ar[d] & Q(0)^{(1)} \ar[d]\ar[r, hookleftarrow] & Q(1)^{(1)} \ar[d]\ar[r, equal] & B(1) \ar[dl, "\overline G_{1}"] & & & &
 \\[3mm]
L^{(0)} \ar[d] \ar[r, hookleftarrow] & Q_\varphi^{(0)} \ar[d] \ar[r, hookleftarrow] & Q(0)^{(0)} \ar[d] \ar[r, equal] & B(0) \ar[dl, "\overline G_{0}"] & & & & & &
\\[4mm]
N & M \ar[l,"\varphi"] \ar[r, equal] & M &&&&&
\end{tikzcd}
$$

The normalization procedure is illustrated by the above diagram.  Note the main differences from the previous case are first of all that
$Q(k)$ is no more a~prin\-cipal bundle over $M$ but only over $Q(k-1)^{(k-1)}$ and second that $\chi^+(k)$ here does not vanish, in general.

The construction is based on the following formula:
\begin{prop}
\[
 R_{\exp A_{k+1}}^*\chi^-(k)_{k+1}= \chi^-(k)_{k+1} - \partial A_{k+1}
	\qquad \text{for} \quad A_{k+1} \in \bar\g(k)_{k+1} (= \lf_{k+1}).
\]
\end{prop}
\begin{proof}
We have already shown this formula in Section~\ref{sec5} under the assumptions of $\bar G^0$ invariance conditions (C1), (C2), but to prove it in our general case we should argue more carefully.

In the formula $\partial$ denotes the coboundary operator in the complex
\[
0 \rightarrow \lf/\bar\g \overset{\partial}
\rightarrow \Hom(\g_-,\lf/\bar\g) \rightarrow \cdots,
\]
identifying $\bar\g' $ with $\lf/\bar\g$, and writing simply $\partial$
for $\partial\circ \pi$, where $\pi$ is the projection: $\lf \to \lf/\bar\g$. During the proof, we simply write
\[
\omega (k) = \omega, \qquad \chi(k) = \chi = \sum _{i \ge 1} \left(\chi^+_i + \chi^-_i\right)\!, \qquad \alpha = \exp A_{k+1}.
\]
Recall that
\begin{gather*}
	R_\alpha^*\omega ={\rm Ad}\big(\alpha^{-1}\big)\omega, \\
	\omega = \omega_I + \omega_{II},\\
	\omega_{II} = \chi \omega_I.
\end{gather*}
Then we have
\begin{align*}
	R_{\alpha} ^*(\omega_I + \omega_{II})
	&= {\rm Ad}\big(\alpha^{-1}\big)(\omega_I + \omega_{II} )\\
	&={\rm Ad}\big(\alpha^{-1}\big)^{I }_{I}\omega_I + {\rm Ad}\big(\alpha^{-1}\big)^{II }_{I}\omega _{II}
	+{\rm Ad}\big(\alpha^{-1}\big)^{I }_{II}\omega_I
	+{\rm Ad}\big(\alpha^{-1}\big)^{II }_{II}\omega_{II} \\
	&={\rm Ad}\big(\alpha^{-1}\big)^{I }_{I}\omega_I + {\rm Ad}\big(\alpha^{-1}\big)^{II }_{I}\chi\omega _{I}
	+{\rm Ad}\big(\alpha^{-1}\big)^{I }_{II}\omega_I
	+{\rm Ad}\big(\alpha^{-1}\big)^{II }_{II}\chi\omega_{I} \\
	&=\left({\rm Ad}\big(\alpha^{-1}\big)^{I }_{I} + {\rm Ad}\big(\alpha^{-1}\big)^{II }_{I}\chi\right)\omega _{I}
	+\left({\rm Ad}\big(\alpha^{-1}\big)^{I }_{II}+{\rm Ad}\big(\alpha^{-1}\big)^{II }_{II}\chi\right)\omega_{I} ,
\end{align*}
where we use the notation ${\rm Ad}\big(\alpha^{-1}\big)^{I }_{II}$ to denote the $\Hom(E(k), E'(k))$-component of ${\rm Ad}\big(\alpha^{-1}\big)$ and similar notations.

Therefore we have
\begin{gather*}
	R_{\alpha} ^*\omega_I =\left({\rm Ad}\big(\alpha^{-1}\big)^{I }_{I} + {\rm Ad}\big(\alpha^{-1}\big)^{II }_{I}\chi\right)\omega _{I},\\
	R_{\alpha} ^*\omega_{II} =\left({\rm Ad}\big(\alpha^{-1}\big)^{I}_{II}+{\rm Ad}\big(\alpha^{-1}\big)^{II }_{II}\chi\right)\omega_{I}.
\end{gather*}
On the other hand, we have
\[
R_{\alpha}^*\omega_{II} = (R_{\alpha}^*\chi )(R_{\alpha}^*\omega_I ).
\]
Therefore we have
\[
\left({\rm Ad}\big(\alpha^{-1}\big)^{I }_{II}+{\rm Ad}\big(\alpha^{-1}\big)^{II }_{II}\chi\right)\omega_{I}
=(R_{\alpha}^*\chi ) \left({\rm Ad}\big(\alpha^{-1}\big)^{I }_{I} + {\rm Ad}\big(\alpha^{-1}\big)^{II }_{I}\chi\right)\omega _{I}.
\]
Letting $\omega_I$ take values in $\g_-$, we have
\[
{\rm Ad}\big(\alpha^{-1}\big)^{I^- }_{II}+{\rm Ad}\big(\alpha^{-1}\big)^{II }_{II}\chi^-
=(R_{\alpha}^*\chi )\left({\rm Ad}\big(\alpha^{-1}\big)^{I^- }_{I} + {\rm Ad}\big(\alpha^{-1}\big)^{II }_{I}\chi^-\right)\!,
\]
and then
\begin{gather*}
{\rm Ad}\big(\alpha^{-1}\big)^{I^- }_{II}+{\rm Ad}\big(\alpha^{-1}\big)^{II }_{II}\chi^-
=(R_{\alpha}^*\chi ^-) \left({\rm Ad}\big(\alpha^{-1}\big)^{I^- }_{I^-} + {\rm Ad}\big(\alpha^{-1}\big)^{II }_{I^-}\chi^-\right)
\\ \hphantom{{\rm Ad}\big(\alpha^{-1}\big)^{I^- }_{II}+{\rm Ad}\big(\alpha^{-1}\big)^{II }_{II}\chi^-	=}
{}	+(R_{\alpha}^*\chi^+) \left({\rm Ad}\big(\alpha^{-1}\big)^{I^- }_{I^+} + {\rm Ad}\big(\alpha^{-1}\big)^{II }_{I^+}\chi^-\right)\!,
\end{gather*}
where ${\rm Ad}\big(\alpha^{-1}\big)^{I^- }_{I^-}$ denotes the $\Hom(\g_-,\g_-)$-component of ${\rm Ad}\big(\alpha^{-1}\big)$ and the same convention for~the similar notation.

Now enter $v_p \in \g_p \ (p < 0)$ into the above formula and catch
the $(\lf/\bar\g)_{i+p}$ -component for $i \le k+1$. Taking into account that
\begin{gather*}
	{\rm Ad}\big(\alpha^{-1}\big)^{I^- }_{II} = - \ad(A_{k+1})^{I^- }_{II}
	+\ \text{higher order terms} ,
\\
	{\rm Ad}\big(\alpha^{-1}\big)^{II }_{II} = \id_{II} - \ad (A_{k+1})^{II}_{II }
	+\ \text{higher order terms}
\end{gather*}
and that
\[
\left({\rm Ad}\big(\alpha^{-1}\big)^{I^- }_{I^+} + {\rm Ad}\big(\alpha^{-1}\big)^{II }_{I^+}\chi^-\right)(v_p) \in \lf_{k+1+p} + \lf_{k+2+p} + \cdots
\]
and therefore
\[
\big(R_{\alpha}^*\chi^+_{j+p}\big) \left({\rm Ad}\big(\alpha^{-1}\big)^{I^- }_{I^+} + {\rm Ad}\big(\alpha^{-1}\big)^{II }_{I^+}\chi^-\right)(v_p) = 0 \qquad \text{for}\quad j \le k+1,
\]
we have
\begin{gather*}
	R_\alpha^*\chi^-_i = \chi^-_i \qquad \text{for\quad }i \le k
\\
	R _\alpha^*\chi^-_{k+1}= \chi^-_{k+1} - \partial A_{k+1},
\end{gather*}
which proves the proposition.
\end{proof}

We can then easily carry out the inductive construction of $Q(k)$, and
by passing to limit we obtain
\[
Q = Q(\infty),\qquad \omega = \omega(\infty), \qquad \chi = \chi(\infty).
\]
The structure function $\chi\colon Q \to \Hom(\bar\g, \lf/\bar\g)$
is an invariant of $Q$ and that of the initial morphism
$\varphi\colon (M, \f) \to (N, \f_N)$. To examine it more closely, we write
\begin{gather*}
	\chi = \sum \left( \chi^-_k + \chi^+_k\right)\!,\\
	\chi^+_k = \sum_{0 \le i < k }\chi^{+i}_k, \qquad \text{where}\quad
	\chi^{+i}_k\quad \text{takes its value in}\ \Hom(\bar\g_i, \bar\g_k).
\end{gather*}
Then we have
\begin{prop}
	\begin{gather*}
		\partial \chi ^-_{k+1} =
		\Psi^-_{k+1} \left(\chi_\ell,\, D\chi_\ell;\, \ell \le k \right)\!,
\\
		\chi ^{+}_{k+1} = \Psi^+_{k+1} \left( \chi^-_{\ell}, \, \chi^{+}_{m},\,
		D\chi^-_{\ell}, \, D\chi^{+}_{m};\, \ell \le k+1, \, m\le k \right)\!,
	\end{gather*}
	where $\Psi^-_{k+1} $ is a vector valued function determined explicitly only by
$\chi_\ell$ $(\ell \le k)$ and their derivatives,
	and $\Psi^+_{k+1}$ is determined only by $\chi^-_{\ell}$ $(\ell\le k+1)$, $\chi^{+}_{m}$ $(m\le k)$ and their derivatives.
	
	Moreover $\Psi^-_{k+1} = 0$ if $\chi _\ell = 0$ $(\ell \le k)$,
	and $\Psi^+_{k+1} = 0 $ if $\chi_\ell = 0$ $(\ell \le k)$ and $\chi^-_{k+1}=0$.
\end{prop}

\begin{proof}
From the structure equation
\begin{gather*}
	{\rm d}\omega + \frac12 [\omega, \omega]= 0, \qquad
	\omega= \omega_I + \omega_{II}, \qquad
	\omega_{II} = \chi \omega_{I},
\end{gather*}
it follows that
\begin{gather*}
	{\rm d}\omega_I + \frac12[\omega_I, \omega_I] +[\omega_I, \omega_{II}]_I
	+\frac12[\omega_{II}, \omega_{II}]_I =0,
\\
	{\rm d}\omega_{II} + [\omega_I, \omega_{II}]_{II} +
	\frac12[\omega_{II}, \omega_{II}]_{II} =0,
\end{gather*}
and then
\begin{gather*}
	{\rm d}\omega_I + \frac12[\omega_I, \omega_I] +[\omega_I, \chi\omega_{I}]_I
	+\frac12[\chi\omega_{I}, \chi\omega_{I}]_I =0,
\\
	{\rm d}(\chi\omega_{I}) + [\omega_I, \chi\omega_{I}]_{II} +
	\frac12[\chi\omega_{I}, \chi\omega_{I}]_{II} =0,
\end{gather*}
which yields
\begin{gather*}
[\omega_I, \chi\omega_I]_{II}- \chi\bigg(\frac12[\omega_I, \omega_I]\bigg)- \chi ([\omega_I, \chi\omega_{I}]_I)
+\frac12[\chi\omega_I, \chi\omega_I]_{II}-\chi\bigg(\frac12[\chi\omega_{I}, \chi\omega_I]_I\bigg)
\\ \hphantom{[\omega_I, \chi\omega_I]_{II}}
{}+{\rm d}\chi\wedge\omega_I	= 0.
\end{gather*}

Now evaluating the above equation at $\tilde v_p\wedge \tilde w_q$, where
$\tilde v$, $\tilde w$ denote the tangent vectors such that
$\omega_I(\tilde v_p) = v_p \in \g_p$, $\omega_I(\tilde w_q) = w_q \in \g_q$ and $p, q < 0$, and looking at the $(\lf/\bar\g)_{k+1+p+q}$-component, we have
\begin{gather*}
	(\partial \chi^-_{k+1} )(v_p, w_q)
	=\mathcal{A}\sum\limits_{\substack{a+b=k+1,\\ a, b >0}}
	\left(\chi^-_a[v_p, \chi^-_b(w_q)]_I
	+[\chi^-_a(v_p), \chi^-_b(w_q)]_{II}\right)
\\ \hphantom{(\partial \chi^-_{k+1} )(v_p, w_q)	=}
	{}+\mathcal A\sum\limits_{\substack{a+b+c=k+1,\\ a,b,c>0}} \chi^-_a[\chi^-_b(v_p), \chi^-_c(w_q)]_I
	-\tilde v_p \chi^-_{k+1+p}(w_q)
	+\tilde w_q \chi^-_{k+1+q}(v_p)
\\ \hphantom{(\partial \chi^-_{k+1} )(v_p, w_q)	=}
	{}+\mathcal A\!\!\!\sum_{0<b<k+1}\!\!\!\chi^+_{k+1+p+q}[v_p, \chi^-_b(w_q)]_I
	+\mathcal A\!\!\!\sum\limits_{\substack{ b+c <k+1,\\ b, c>0}}\!\!\! \chi^+_{k+1+p+q} [\chi^-_b(v_p), \chi^-_c(w_q)]_I,
\end{gather*}
where $\mathcal A$ denotes the alternating sum in $v_p$, $w_q$.
The last formula defines explicitly the func\-tion~$\Psi^-_{k+1}$.

Next we evaluate the previous formula at $\tilde A_i \wedge \tilde v_p$, where $A_i \in\bar\g_i$ ($i \ge 0$), $v_p \in \g_p$ ($p <0$) and~$\tilde A_i$,~$\tilde v_p$ are the corresponding tangent vectors. Looking at $(\lf/\bar\g)_{k+1+p}$-component, we have
\begin{gather*}
	\left[v_p, \chi^+_{k+1}A_i\right]_{II}
	=\left[A_i, \chi^-_{k+1-i}v_p\right]_{II}
	-\chi^-_{k+1-i}\left[A_i, v_p\right] - \chi^+_{k+1+p}\left[A_i, v_p\right]
\\ \hphantom{\left[v_p, \chi^+_{k+1}A_i\right]_{II}=}
	{}-\sum_{a+b+i =k+1}\chi^-_a\left[A_i, \chi^-_bv_p\right]_I
	-\sum_{b+i <k+1}\chi^+_{k+1+p}\left[A_i, \chi^-_bv_p\right]_I
\\ \hphantom{\left[v_p, \chi^+_{k+1}A_i\right]_{II}=}
	{}+\sum_{a+b=k+1}\chi^-_a\left[v_p, \chi^+_bA_i\right]_I
	+\sum_{b <k+1} \chi^+_{k+1+p}\left[v_p, \chi^+_bA_i\right]_I
\\ \hphantom{\left[v_p, \chi^+_{k+1}A_i\right]_{II}=}
	{}+\sum_{a+b= k+1}\left[\chi^+_aA_i, \chi^-_bv_p\right]_{II}
	-\sum_{a+b+c = k+1}\chi^-_a\left[\chi^+_bA_i, \chi^-_cv_p\right]_{I}
\\ \hphantom{\left[v_p, \chi^+_{k+1}A_i\right]_{II}=}
	{}-\sum_{b+c<k+1}\left[\chi^+_bA_i, \chi^-_cv_p\right]_I
	+\tilde A_i\chi^-_{k+1}(v_p)- \tilde v_p \chi^+_{k+1+p}(A_i),
\end{gather*}
where $a,b,c>0$. Since $\partial\colon \lf/\bar\g \to \Hom(\g_-, \lf /\bar\g)$ is injective, the above formula determi-\linebreak nes~$\Psi^+_{k+1}$.
\end{proof}

Note that we know the finitness of the cohomology group:
\begin{prop}
	There exists an integer $r_0$ such that
	\[
	H^i_r(\g_{-},\lf/\bar\g) = 0 \qquad\text{for all}\quad r>r_0\quad \text{and all}\quad i.
	\]
\end{prop}
\begin{proof}
	The proof of Theorem 2.1 in~\cite{Mor1988} applies also to this case.
\end{proof}

Now we have
\begin{thm}\label{thm:inf}
	Let $\varphi\colon(M,\f_M)\to(N,\f_N)$ be a morphism of type $(\g_{-},\cL)$. Then we can construct the series of principal fibre bundles $Q(k)\to B(k-1)$ with $\lf$-valued $1$-forms $\om(k)$ and the structure functions $\chi(k)^{(k)}\colon B(k)\to W^{(k)}$ for $k=0,1,2,\dots$.
And $\chi(k_0)^{(k_0)}$ fulfills the complete set of~invariants of $\varphi$, where $k_0$ is an integer such that $H^1_{k}(\g_{-},\lf/\bar\g)=0$ for all $k>k_0$.
\end{thm}

Let $G_{-}$ (resp.~$L_{-}$) be the Lie group with the Lie algebra $\g_{-}$ (resp.~$\lf_{-}$). Then we have local embeddings uniquely defined up to $\cL$-equivalence:
\[
\varphi_{{\rm model}}\colon\ G_{-}\to L_{-}\to (N,\f_N),
\]
which we call \emph{the standard embedding}.
\begin{cor}
	Let $\varphi$, $\chi(k_0)^{(k_0)}$ be as in the preceding theorem. If $\chi(k_0)^{(k_0)}=0$ identically, in particular, if $H^1_r(\g_{-},\lf/\bar \g)=0$ for $r>0$, then $\varphi$ is formally $\cL$-equivalent to the standard embedding at every point of $M$.
\end{cor}

In fact, if the hypothesis of the corollary holds, then the differential equation for an $\cL$-equivalence between $\varphi$ and $\varphi_{{\rm model}}$ is proved to be formally integrable and involutive, and therefore has a formal solution. Moreover, it has a local analytic solution in the analytic category. If $\lf$ is finite-dimensional, then the differential equation is of finite type and has a local smooth solution even in the $C^\infty$-category.

\subsection{Examples}

\begin{ex}
	As an example, consider the embeddings $\varphi\colon (M,\f_M)\to (N,\f_N)$, where $(N,\f_N)$ is a contact manifold of dimension $2n+1$. In~other words, $(N,\f_N)$ is a filtered manifold of constant type $\m=\m_{-2}\oplus\m_{-1}=\n_{2n+1}$, where $\n_{2n+1}$ is a $(2n+1)$-dimensional Heisenberg Lie algebra. Note that the Lie bracket $\wedge^2 \m_{-1}\to \m_{-2}$ defines a symplectic form on $\m_{-1}$ up to a~non-zero scale.
	
	In this case, $\lf$ is the (infinite-dimensional) graded Lie algebra associated with the filtered Lie algebra of all contact vector fields on $N$. Let $(x_i,y,z_i)$ be a local coordinate system on $N$ such that the contact distribution $\f_{N}^{-1}$ is defined as an annihilator of the 1-form
	\[
	\theta = {\rm d}y - z_1\,{\rm d}x_1 - \dots - z_n\,{\rm d}x_n.
	\]
	Then it is well-known that all contact vector fields $X_f$ are parametrized by a single function $f$ on $N$ such that
	\[
	\theta(X_f)=f,\qquad (L_{X_f} \theta)\wedge \theta = 0.
	\]
	The Lie bracket of contact vector fields induces the bracket $\{f,g\}$ of functions $f$, $g$ on $N$ such that
	\[
	X_{\{f,g\}} = [X_f,X_g].
	\]
	Explicitly, we have:
	\[
	\{f,g\} = f\frac{\partial g}{\partial y} - g\frac{\partial f}{\partial y} +
	\sum_{i=1}^n \bigg(\frac{{\rm d}f}{{\rm d}x_i}\frac{\partial g}{\partial z_i} - \frac{{\rm d}g}{{\rm d}x_i}\frac{\partial f}{\partial z_i} \bigg),
	\]
	where $\frac{\rm d}{{\rm d}x_i} = \frac{\partial}{\partial x_i} + z_i \frac{\partial}{\partial y}$
for all $i=1,\dots, n$.
	
	The Lie algebra $\lf$ can be described as a set of all contact vector fields $X_f$, where $f$ runs through all polynomials in $(x_i,y,z_i)$, $1\le i\le n$. In~the following, we identify $\lf$ with the space of such polynomials equipped with the above Lie bracket.
	
	The Lie algebra $\m$ can be identified with $\lf_{-}$, which corresponds to the subspace
	\[
	\langle 1, x_i, z_i\mid 1\le i \le n \rangle.
	\]
	Let us classify all graded subalgebras $\g_{-}$ of $\m$ up to $\Aut_0(\m)$.
	
$\bullet$ {\it Case} 1: $\g_{-}=\g_{-1}\subset \m$. Then $\g_{-}$ is necessarily commutative and thus, is an isotropic subspace of $\m_{-1}$. As the symplectic group acts transitively on isotropic subspaces of the same dimension in $\m_{-1}$, we see that the $\g_{-}$ is uniquely determined up to $\Aut_0(\m)$ (=the grading preserving group of automorphisms of $\m$) by its dimension.
	
$\bullet$ {\it Case} 2: $\g_{-}$ is 2-graded. Then we have $\g_{-2}=\m_{-2}$. Assuming that $\g_-$ is generated by $\g_{-1}$, we get that $\g_{-1}$ can be an arbitrary non-isotropic subspace of $\m_{-1}$. Again, modulo the action of $\Aut_0(\m)$ such subspaces are uniquely determined by its dimension and the dimension of the kernel of the restriction of the symplectic form to $\g_{-1}$.
	
	Let us compute $\Prol(\g_{-}, \lf)$ in cases, when $\g_{-}$ is maximal isotropic in $\m_{-1}$ or when $\g_{-}$ is two-graded, and the restriction of the symplectic form to $\g_{-1}$ is non-degenerate.
	
$\bullet$ {\it Case} 1: $\g_{-}$ is a maximal isotropic subspace in $\lf_{-1}$. Then up to $\Aut_0(\m)$, we can assume that
	\[
	\g_{-} = \langle z_1,\dots, z_n\rangle.
	\]
	Let $\g=\Prol(\g_{-},\lf)$. Then it is easy to see that
	\[
	\g_0 = \langle y,\, x_iz_j,\, z_iz_j \mid 1\le i,\,j\le n\rangle.
	\]
	Further a simple computation shows that
	\[
	\g = I(y,z_1,\dots,z_n),
	\]
	the ideal generated by $y$ and $z_1,\dots,z_n$. In~particular, the subspace $\g'$ consisting of all polynomials in variables $x_i$, $1\le i \le n$, forms a complementary linear subspace to $\g$ in $\lf$. Moreover, it~is invariant with the action of $\g_{-}$.
	
	Since $\{x_i,z_j\} = \delta_{ij}$ for all $1\le i,j\le n$, we see that the cohomology complex $C(\g_{-},\g')$ coincides with the \emph{canonical Koszul complex}, which is known to be acyclic. In~particular, we get that $H^1(\g_{-},\g')=0$.
	
$\bullet$ {\it Case} 2: $\g_{-}$ is 2-graded and $\g_{-1}$ is non-degenerate. Let $\dim \g_{-1} = 2m$. Then up to $\Aut_0(\m)$, we can assume that
	\[
	\g_{-} = \langle 1, x_i,z_i \mid 1\le i \le m \rangle.
	\]
	
	Computing $\g_0$ we get
	\[
	\g_0 = \langle y,\, x_ix_j, x_iz_j, z_iz_j,\, x_ax_b, x_az_b, z_az_b \rangle,
	\]
	where $1\le i,k\le m$; $m+1\le a,b\le n$.
	
	Further computations show that a polynomial $f$ belongs to $\g$ if and only if it has no terms linear in $(x_a,z_a)$, $m+1\le a\le n$. In~particular, as complementary linear subspace $\g'$, we can choose
	\[
	\g'= \big\langle x_a P_a(x_i,y,z_i), z_a Q_a(x_i,y,z_i) \big\rangle,
	\]
	$1\le i,k\le m$; $m+1\le a,b\le n$, where $P_a$, $Q_a$, $m+1\le a\le n$, are arbitrary polynomials in~$(x_i,y,z_i)$, $1\le i\le m$.
	
	Note that this space is invariant with respect to $\g_{-}$ and splits into the direct sum of $2(n-m)$ copies of submodules, which are parametrized by a single polynomial $P(x_i,y,z_i)$ each. In~its turn, each such submodule is isomorphic to the graded Lie algebra $\lf'$ associated to the filtered Lie algebra of all contact vector fields on the $(2m+1)$-dimensional contact manifold. Since such Lie algebra coincides with the Tanaka prolongation of $\g_{-}$, we get that $H^1_{+}(\g_{-},\lf')=0$ and hence $H^1_{+}(\g_{-},\g')=0$.
	
	Thus, corollary to Theorem~\ref{thm:inf} can be applied in both cases.
\end{ex}

\begin{ex}
	Consider now the extrinsic geometry of scalar 2nd order ODEs (ordinary differential equations) under contact and point transformations.
	
	We can consider a scalar 2nd order ODE as an embedding $\varphi\colon (M,\f_M)\to (N,\f_N)$, where $(N,\f_N)$ is a $4$-dimensional jet space $J^2(\R,\R)$ with the canonical jet space filtration $\f_N$:
	\[
	\f_N^{-1} \subset f_N^{-2} \subset \f_N^{-3} = TN,
	\]
	and $M$ is a $3$-dimensional manifold (a 2nd order ODE) such that $\varphi(M)$ is transversal to the projection $\pi\colon J^2(\R,\R)\to J^1(\R,\R)$ at each point $x\in\varphi(M)$. The filtration $\f_M$ is induced by the filtration $\f_N$:
	\[
	\f_M^{i} = \f_N^{i}\cap TM\qquad\text{for all}\quad i\in \Z.
	\]
	
{\samepage	Let $\m$ be the symbol algebra of $(N,\f_N)$. It is well-known that $\m=\m_{-1}\oplus \m_{-2}\oplus \m_{-3}$ such that
	\begin{gather*}
	\m_{-1} = \langle E_2, X \rangle, \qquad
	\m_{-2} = \langle E_1 \rangle, \qquad
	\m_{-3} = \langle E_0 \rangle,
	\end{gather*}
where
	\begin{gather*}
[X,E_i] = E_{i-1},\qquad [E_i,E_j]=0.
	\end{gather*}

}

\noindent
Note that the fibres of the projection $\pi$ are tangent to the characteristic direction of $\f_{N}^{-2}$, which lies inside $\f_N^{-1}$ and corresponds to the subspace $\langle E_2 \rangle$ in $\m_{-1}$.
	
	Since $M$ is transversal to this direction, the space $\varphi_*(T_xM)$ is a $3$-dimensional vector space complementary to it, and $\gr \varphi_*(T_xM) = \gr (\f_M)_{\varphi(x)}$ is a $3$-dimensional graded subalgebra of $\m$ complementary to $\langle E_2 \rangle$. It is easy to see that all such subspaces are equivalent to the fixed subspace
	\[
	\g_{-} = \langle X,E_1, E_0 \rangle
	\]
	under the action of $\Aut_0(\m)$.
	
	Consider now the extrinsic invariants of the embedding $\varphi$ under contact transformations. Let $\lf=\lf_{\text{cont}}$ be is the (infinite-dimensional) graded Lie algebra associated with the filtered Lie algebra of all vector fields on $N$ preserving the filtration $\f_N$. Due to Lie's theorem all such vector fields are the lifts of the contact vector fields on $J^1(\R,\R)$.
	
	As in the previous example, the Lie algebra $\lf$ can be described as a set of all contact vector fields $X_f$, where $f$ runs through all polynomials in $(x,y,z)$. The grading of $\lf$ is defined as follows. First, define weighted degree on $\R[x,y,z]$ by setting $\deg x = 1$, $\deg y = 3$, $\deg z = 2$. Then for any weighted homogeneous polynomial $f\in\R[x,y,z]$ we define
	\[
	\deg X_f = \deg f - 3.
	\]
	In the following we identify $\lf$ with the space of such polynomials equipped with the above Lie bracket.
	\begin{rem}
		This construction easily generalizes to the description of the graded Lie algebra associated with the infinite-dimensional Lie algebra of vector fields preserving the contact distribution on $J^k(\R,\R)$ for any $k\ge 1$. Namely, in the case the (abstract) Lie algebra is still isomorphic to $\lf$, but the grading is defined differently according to the rules: $\deg x = 1$, $\deg y = k+1$, $\deg z = k$, and $\deg X_f =\deg f-k-1$.
	\end{rem}
	
	In this notation, the Lie algebra $\m=\lf_{-}$ has the form
	\[
	\langle 1, x, x^2, z\rangle.
	\]
	The subalgebra $\g_{-}$ then corresponds to the subspace
	\[
	\g_{-} = \langle 1, x, z \rangle.
	\]
	
	Let us compute $\g=\Prol(\g_{-}, \lf)$. It is easy to see that
	\begin{gather*}
	\lf_0 = \big\langle y, x^3, xz \big\rangle,\qquad
	\g_0 = \langle y, xz \rangle,\\
	\lf_1 = \big\langle xy, x^4, x^2z, z^2\big\rangle,\qquad
	\g_1 = \big\langle x(y-xz), z^2 \big\rangle.
	\end{gather*}
	
	Further simple computations show that
	\[
	\g = \{ P(z,y-xz)+Q(z,y-xz)x \},
	\]
	where $P$, $Q$ are two arbitrary polynomials in 2 variables. Note that $\g$ coincides with all contact polynomial symmetries of the trivial differential equation $y''=0$, or, in other words, with all polynomial vector fields on $J^1(\R,\R)$ which preserve the $1$-dimensional vector distribu\-tion~$\langle {\rm d}/{\rm d}x \rangle$.
	
	{\samepage
Let us show that $H^1_+(\g_{-},\lf/\g)=0$. The short exact sequence
	\[
	0\to \g \to \lf \to \lf/\g \to 0
	\]
	induces the long exact sequence
	\begin{equation}\label{lex}
	\cdots \to H^1(\g_{-}, \lf) \to H^1(\g_{-}, \lf/\g) \to H^2(\g_{-},\g) \to \cdots.
	\end{equation}
	So, it is sufficient to prove that $H^1_{+}(\g_{-}, \lf) = 0$ and $H^2_{+}(\g_{-},\g)=0$. This is done in two steps below.

}
	
	\textit{Step} 1. Denote by $\g_{-}[-1]\subset \lf[-1]$ the same abstract Lie algebras, but with the gra\-ding induced from $J^1(\R,\R)$, that is by assuming that $\deg x = 1$, $\deg y= 2$, $\deg z= 1$ and $\deg X_f = \deg f - 2$. Then it is well known that $\lf[-1]$ is the Tanaka prolongation of $\g_{-}[-1]$ and hence $H^1_r\left(\g_{-}[-1],\lf[-1]\right)=0$ for all $r\ge 0$. We note that $H^1\left(\g_{-}[-1],\lf[-1]\right)$ may have non-zero components in negative degree. But it is easy to check that the negative part of $C^1\left(\g_{-}[-1],\lf[-1]\right)=\Hom\left(\g_{-}[-1],\lf[-1]\right)$ belongs to the negative part of $C^1(\g_{-},\lf)$. Hence, we also get $H^1_+(\g_{-},\lf)=0$.
	
	\textit{Step} 2. Let us prove that $H^2_{+}(\g_{-},\g)=0$. The Legendre transform
	\[
	(x,y,z) \mapsto (-z, y-xz, x)
	\]
	induces an automorphism of $\lf$ that maps this vector distribution to $\langle \partial/\partial z \rangle$ and, hence, maps the subalgebra $\g$ to the subalgebra of all polynomial vector fields on the plane $(x,y)$ lifted to~$J^1(\R,\R)$.
	
	Thus, computing $H(\g_{-},\g)$ we can identify $\g$ with the Lie algebra of all polynomial vector fields on the plane with the grading induced from the conditions $\deg x = 2$, $\deg y = 3$. In~particular, under this identification we have:
	\[
	\g_{-3} = \langle \partial/\partial y \rangle,\qquad
	\g_{-2} = \langle \partial/\partial x \rangle, \qquad
	\g_{-1} = \langle x\partial/\partial y \rangle.
	\]
	
	We note that $V = \langle \partial/\partial x, \partial/\partial y\rangle$ forms an ideal in $\g_{-}$. Moreover, the cohomology complex for the $V$-module $\g$ is exactly the standard Koszul cohomology complex, which is known to be acyclic. Thus, we have $H^q(V,\g)=0$ for all $q>0$ and $H^0(V,\g) = V$.
	
	Now use Serre--Hochschild spectral sequence~\cite{sh1953} associated with the ideal $V\subset \g_{-}$. Its second term $E_2^{p,q}$ is equal to
	\[
	E_2^{p,q} = H^p \big(\g_{-}/V, H^q(V,\g)\big).
	\]
	Since we want to compute $H^2(\g_{-},\g)$, we can assume that $p+q=2$. If $q>0$, then by the above $H^q(V,\g)=0$. Finally, if $q=0$ and $p=2$, then
	\[
	H^2 \big(\g_{-}/V, H^0(V,\g)\big) = 0,
	\]
	since $\g_{-}/V$ is $1$-dimensional. This proves that $H^2(\g_{-},\g)=0$. In~particular, this implies that an arbitrary 2nd order ODE is contactly
	equivalent to the trivial equation $y''=0$.
	
	Consider now the case of the pseudogroup of point transformations on the plane. Let $\lf=\lf_{\textrm{point}}$ be an infinite-dimensional graded Lie algebra associated with the (filtered) Lie algebra of~all infinitesimal point transformations acting on $J^2(\R,\R)$. It is well-known that $\lf_{\textrm{point}}$ can be identified with the subalgebra of $\lf_{\textrm{cont}}$ consisting of all vector fields $X_f$, where $f$ is a polynomial in~$x$, $y$, $z$ linear in $z$. We note that the grading of $\lf$ is induced from its prolongation to $J^2(\R,\R)$ and the standard filtration on $J^2(\R,\R)$ (see above), which is different from the more conventional grading on the Lie algebra of polynomial vector fields on the place.
	
	We have the same $\g_{-}\subset \lf_{-}$ as in the contact case. But the prolongation is now different. Similar to the contact case, it is easy to show by direct computation that the prolongation of~$\g_{-}$ in $\lf_{\textrm{point}}$ coincides with the point symmetry algebra of the trivial equation $y''=0$, which is exactly the Lie algebra $\mathfrak{sl}(3,\R)$ of all projective vector fields on the plane:
	\begin{gather*}
	\g_{-3} = \langle \partial/\partial y \rangle,\qquad
	\g_{-2} = \langle \partial/\partial x \rangle, \qquad
	\g_{-1} = \langle x\partial/\partial y \rangle, \\
	\g_{0} = \langle x\partial/\partial x, \, y\partial/\partial y\rangle,\\
	\g_{1} = \langle y\partial/\partial x \rangle, \qquad
	\g_{2} = \langle x U \rangle, \qquad
	\g_{3} = \langle y U \rangle,
	\end{gather*}
	where $U= x\partial/\partial x+y\partial/\partial y$.
	
	As above, we can compute $H_{+}^1(\g_{-},\lf/\g)$ using the long exact sequence~\eqref{lex}. Next, we can prove that $H^k_{+}(\g_{-},\lf)=0$ for $k=1$ and $k=2$ by means of the Serre--Hochschild spectral sequence as in Step 2 above. Then~\eqref{lex} implies the isomorphism
	\[
	H^1_{+}(\g_{-}, \lf/\g) \cong H^2_{+}(\g_{-},\g).
	\]
	We note that the cohomology space on the right (with a different grading) corresponds to the invariants of the parabolic geometry associated with a scalar 2nd order ODE~\cite{capslovak,tan79}. So we see that the extrinsic geometry of the 2nd order ODE under point transformations agrees the parabolic geometry associated with this equation.
\end{ex}

\subsection*{Acknowledgements}

The third author is partially supported by JSPS KAKENHI Grant Number 17K05232.


\pdfbookmark[1]{References}{ref}
\LastPageEnding

\end{document}